\documentclass[12pt]{article}
\usepackage{amsfonts}
\usepackage{amssymb}
\usepackage{amsbsy}
\usepackage{amsthm}
\usepackage{amsmath}
\usepackage{amscd}
\usepackage{color}
\usepackage{mathrsfs}
\usepackage{holtpolt}
\usepackage{mathabx}
\usepackage{ifsym}
\usepackage{ulem}
\usepackage{graphicx}
\usepackage{mathtools}
\usepackage[all]{xy}
\usepackage{mathdots}
\usepackage{tikz}
\usepackage[titletoc,toc,title]{appendix}

\newtheorem{thm}{Theorem}[section]
\newtheorem{cor}[thm]{Corollary}
\newtheorem{lem}[thm]{Lemma}
\newtheorem{prop}[thm]{Proposition}
\theoremstyle{definition}
\newtheorem{defn}[thm]{Definition}
\theoremstyle{remark}
\newtheorem{rem}[thm]{Remark}
\newtheorem{ex}[thm]{\bf Example}

\newcommand{\bt}{\begin{thm}}
\newcommand{\et}{\end{thm}}
\newcommand{\bc}{\begin{cor}}
\newcommand{\ec}{\end{cor}}
\newcommand{\bl}{\begin{lem}}
\newcommand{\el}{\end{lem}}
\newcommand{\bp}{\begin{prop}}
\newcommand{\ep}{\end{prop}}
\newcommand{\bd}{\begin{defn}}
\newcommand{\ed}{\end{defn}}
\newcommand{\br}{\begin{rem}}
\newcommand{\er}{\end{rem}}
\newcommand{\bpr}{\begin{proof}}
\newcommand{\epr}{\end{proof}}
\newcommand{\bex}{\begin{ex}}
\newcommand{\eex}{\end{ex}}
\newcommand{\bcd}{\begin{CD}}
\newcommand{\ecd}{\end{CD}}

\newcommand{\bi}{\begin{itemize}}
\newcommand{\ei}{\end{itemize}}
\newcommand{\be}{\begin{enumerate}}
\newcommand{\ee}{\end{enumerate}}
\newcommand{\ba}{\begin{array}}
\newcommand{\ea}{\end{array}}
\newcommand{\beq}{\begin{equation}}
\newcommand{\eeq}{\end{equation}}
\newcommand{\beqa}{\begin{eqnarray}}
\newcommand{\eeqa}{\end{eqnarray}}
\newcommand{\bca}{\begin{cases}}
\newcommand{\eca}{\end{cases}}
\newcommand{\bal}{\begin{aligned}}
\newcommand{\eal}{\end{aligned}}

\newcommand{\ds}{\displaystyle}
\newcommand{\ts}{\textstyle}
\newcommand{\scs}{\scriptstyle}
\newcommand{\sss}{\scriptscriptstyle}

\newcommand{\N}{{\mathbb N}}
\newcommand{\Z}{{\mathbb Z}}
\newcommand{\R}{{\mathbb R}}
\newcommand{\C}{{\mathbb C}}

\newcommand{\PP}{{\mathbb P}}
\newcommand{\LL}{{\mathbb L}}

\newcommand{\bs}{\boldsymbol}

\newcommand{\DVZ}{{\operatorname{DVZ}}}
\newcommand{\Sz}{{\operatorname{Sz}}}
\newcommand{\spn}{{\operatorname{span}}}

\newcommand{\re}{{\operatorname{Re}}}

\newcommand{\sg}{{\operatorname{sg}}}

\newcommand{\ran}{{\operatorname{ran}}}

\newcommand{\<}{\langle}
\renewcommand{\>}{\rangle}
\newcommand{\uk}{|\kern-3.3pt\uparrow\>}
\newcommand{\dk}{|\kern-3.3pt\downarrow\>}
\newcommand{\ub}{\<\uparrow\kern-3.3pt|}
\newcommand{\db}{\<\downarrow\kern-3.3pt|}

\begin{document}

\title{\bf Darboux transformations \break for CMV matrices}

\author{
M. J. Cantero$^{1}$, 
F. Marcell\'an$^{2}$, 
L. Moral$^{1}$, 
L. Vel\'azquez$^{1,}$\footnote{Corresponding author (\texttt{velazque@unizar.es})}
}

\date{\footnotesize 
$^1$ Departamento de Matem\'atica Aplicada and Instituto Universitario de Matem\'aticas y Aplicaciones (IUMA), Universidad de Zaragoza, Spain
\break 
$^2$ Instituto de Ciencias Matem\'aticas (ICMAT) and Departamento de Matem\'aticas, Universidad Carlos III de Madrid, Spain
}

\maketitle

\vspace{-15pt}

\begin{abstract}

{\footnotesize

We develop a theory of Darboux transformations for CMV matrices, canonical representations of the unitary operators. In perfect analogy with their self-adjoint version -- the Darboux transformations of Jacobi matrices -- they are equivalent to Laurent polynomial modifications of the underlying measures. We address other questions which emphasize the similarities between Darboux transformations for Jacobi and CMV matrices, like their (almost) isospectrality or the relation that they establish between the corresponding orthogonal polynomials, showing also that both transformations are connected by the Szeg\H o mapping. 

Nevertheless, we uncover some features of the Darboux transformations for CMV matrices which are in striking contrast with those of the Jacobi case. In particular, when applied to CMV matrices, the matrix realization of the inverse Darboux transformations -- what we call `Darboux transformations with parameters' -- leads to spurious solutions whose interpretation deserves future research. Such spurious solutions are neither unitary nor band matrices, so Darboux transformations for CMV matrices are much more subject to the subtleties of the algebra of infinite matrices than their Jacobi counterparts. 

A key role in our theory is played by the Cholesky factorizations of infinite matrices. Actually, the Darboux transformations introduced in this paper are based on the Cholesky factorizations of degree one Hermitian Laurent polynomials evaluated on CMV matrices. These transformations are also generalized to higher degree Laurent polynomials, as well as to the extension of CMV matrices to quasi-definite functionals -- what we call `quasi-CMV' matrices.

Furthermore, we show that this CMV version of Darboux transformations plays a role in integrable systems like the Schur flows or the Ablowitz-Ladik model which parallels that of Darboux for Jacobi matrices in the Toda lattice.  

} 

\end{abstract}

{\footnotesize


\smallskip

\noindent{\it (2010) AMS Mathematics Subject Classification}: 42C05, 47B36, 15A23.

}

\tableofcontents

\section{Introduction} 
\label{sec:INT}

Darboux transformations, originally tied to Schr\"odinger operators \cite{Darboux1,Darboux2,I1941,IH1948,IH1951,Moutard,Sch1,Sch2}, have become a powerful tool in many areas of mathematical physics (see \cite{Gr2011,HK,MatSal,MiRO,Teschl,Toda} and references therein). For instance, the role of 1D Schr\"odinger operators in the Lax pair of the KdV equation led to export Darboux transformations to KdV and generalizations \cite{MatSal}. The discrete version of these relations links Darboux for Jacobi matrices (discrete 1D Schr\"odinger) and the Toda lattice (discrete KdV) \cite{Teschl,Toda}. 

Darboux transformations of Jacobi matrices are at the center of a rich interplay among integrable systems, orthogonal polynomials, bispectral problems and numerical linear algebra \cite{BuMar,BuMar2,BuIs,EK,Galant,Galant2,Gau,Gau2,GoKa,GH96,GH97,GHH,HTI,KaGo,MatSal2,SpViZh,SpZh,W1993,Yoon,Zhe}. These topics are also closely related to the unitary counterpart of Jacobi, the CMV matrices, which date back to works on the unitary eigenproblem \cite{AGR1986,BGE1991,W1993}, a decade before their rediscovery in the context of orthogonal polynomials on the unit circle \cite{CMV2003,Si2005,SiII2005}. It was later realized that CMV matrices provide the Lax pair of integrable systems known under the name of Schur flows (discrete mKdV and unitary analogue of Toda) and Ablowitz-Ladik (discrete nonlinear Schr\"odinger) \cite{AFMa,AG94,FaGe,GeNe,Gol,KN2007,LN12,MuNa,Nenciu05,Nenciu06,Si2007bis}. Yet, Darboux has not been applied to CMV matrices so far, the closest precedents being on related issues for isometric Hessenberg matrices \cite{DarHerMar,GarHerMar,GarMar2,Gragg,HuVanB,WG02}. 

This paper develops a theory of Darboux transformations for CMV matrices which intends to be useful for applications which parallel those of the Jacobi case. Indeed, we will see that this extension of Darboux has a role in Ablowitz-Ladik and Schur flows which mimics that of Darboux for Jacobi in Toda: Darboux for CMV is an integrable discretization of such flows which also generates new flows from known ones. Besides, the success of Darboux in Jacobi bispectral problems suggests that Darboux for CMV may open new ways in the search for bispectral situations on the unit circle, where it seems difficult to escape from trivial instances. Further, the recent discovery that CMV drives the evolution of 1D quantum walks \cite{CGMV} (quantum version of random walks coming from the discrete 1D Dirac equation \cite{QRW,Meyer}) could add new uses of Darboux for CMV. 

Let us summarize the main ideas and results of the paper using the Jacobi case as a guide and counterpoint. 

The strategy for the extension of Darboux to CMV is to transform every CMV into a self-adjoint band matrix $M$ to which we apply standard Darboux transformations based on triangular band factorizations followed by a commutation of the factors: $M=AB \to N=BA$, with $A,B$ lower/upper triangular band matrices. 

Darboux is usually implemented on a Jacobi matrix via the LU factorization of a similar non-symmetric tridiagonal matrix. However, it is possible to rewrite these transformations by performing Cholesky factorizations ($B=A^+$, the adjoint of $A$) directly onto the Jacobi matrix, which avoids loosing the hermiticity in the process. We will use the Cholesky approach for the CMV version of Darboux because it permits a closer control of unitarity, a more intricate property than hermiticity. 

In the case of a Jacobi matrix $\cal J$, the Cholesky factorization may require a previous real shift $M={\cal J}+\beta I$ ($I$ stands for the identity matrix) to deal with a positive definite matrix. As a natural CMV translation of this we will introduce, not only an eventual real shift, but also a Hermitian linear combination of a CMV $\cal C$ and its adjoint ${\cal C}^+={\cal C}^{-1}$ mapping it into a self-adjoint band matrix $M=\alpha {\cal C} + \overline\alpha {\cal C}^{-1} + \beta I$. In other words, while Darboux for Jacobi involves a real polynomial $\wp(x)=x+\beta$ evaluated on the Jacobi matrix, in the CMV case we will evaluate a Hermitian Laurent polynomial $\ell(z)=\alpha z+\overline\alpha z^{-1}+\beta$ on the CMV matrix.  

Despite the above difference, we will discover that Darboux for Jacobi and CMV have a similar meaning since they are respectively equivalent to polynomial and Laurent polynomial modifications of the underlying measure. This is part of the first main result of the paper, Theorem~\ref{thm:Darboux}, which yields several equivalences for Darboux transformations of CMV matrices in perfect analogy with the Jacobi case. Thus, we will show that Darboux for Jacobi and CMV share many properties concerning not only the matrices, but also the related measures and orthogonal polynomials. Some of these properties are reflected in the following table of analogies, which also includes a summary of notations used along the paper.

\begin{center}
\renewcommand{\arraystretch}{1.5}
\begin{tabular}{|c|c|}
\hline
 	\parbox{148pt}
	{\begin{center} \vskip-7pt Darboux for Jacobi \end{center} 
	\vskip-14pt 
	with polynomial $\wp(x)=x+\beta$ \vskip5pt}   
 &	\parbox{200pt}
	{\begin{center} \vskip-7pt Darboux for CMV \end{center} 
	\vskip-14pt 
	with Laurent pol. $\ell(z)=\alpha z+\overline\alpha z^{-1}+\beta$ 	
	\vskip5pt} 
\\ \hline
 	$\begin{gathered}
	\\[-15pt]
 	\begin{aligned}
	& \text{Jacobi} 
	& \kern5pt & {\cal J} \overset{\wp}{\mapsto} {\cal K}
	\\
	& \parbox{63pt}{Measures on \vskip-3pt the real line} 
	& & \mu \overset{\wp}{\mapsto} \nu
	\\
	& \parbox{62pt}{Orthogonal \vskip-3pt polynomials} 
	& & p \overset{\wp}{\mapsto} q
	& & \kern-7pt \begin{cases} 
		p=(p_n) \\[-3pt] q=(q_n) 
		\end{cases} 
		\kern-20pt
	\end{aligned}
	\\[3pt]
 	\end{gathered}$  
 & 	$\begin{gathered}
 	\\[-15pt]
	\begin{aligned}
	& \text{CMV} 
	& \kern5pt & {\cal C} \overset{\ell}{\mapsto} {\cal D}
	\\
	& \parbox{70pt}{Measures on \vskip-3pt the unit circle} 
	& & \mu \overset{\ell}{\mapsto} \nu
	\\
	& \parbox{63pt}{Orthogonal \vskip-3pt Laurent pol.} 
	& & \chi \overset{\ell}{\mapsto} \omega
	& & \begin{cases} 
		\chi=(\chi_n) \\[-3pt] \omega=(\omega_n) 
		\end{cases} \kern-20pt
	\end{aligned}
	\\[3pt]
 	\end{gathered}$    
\\ \hline
  	$\begin{gathered}
	\\[-15pt]
 	{\cal K}=A^{-1}{\cal J}A
 	\\
 	\wp({\cal J})=AA^+ \qquad \wp({\cal K})=A^+A
	\\
	A \text{ 2-band lower triangular}
	\\[2pt]
	\end{gathered}$ 
 & 	$\begin{gathered}
 	\\[-15pt]
 	{\cal D}=A^{-1}{\cal C}A
 	\\
 	\ell({\cal C})=AA^+ \qquad \ell({\cal D})=A^+A
	\\
	A \text{ 3-band lower triangular}
	\\[2pt]
	\end{gathered}$ 
\\ \hline
  	$\begin{gathered} d\nu = \wp \, d\mu \\[2pt] \end{gathered}$
 & 	$\begin{gathered} d\nu = \ell \, d\mu \\[2pt] \end{gathered}$
\\ \hline
 	$\begin{aligned}
	\\[-14pt]
 	& p=Aq 
	& \kern5pt & p_n \in \spn\{q_{n-1},q_n\}
	\\
	& \wp q=A^+p 
	& & \wp q_n \in \spn\{p_n,p_{n+1}\}
	\\[4pt]
	\end{aligned}$  
	\kern-14pt
 & 	$\begin{aligned}
 	& \chi=A\omega 
	& \kern5pt & \chi_n \in \spn\{\omega_{n-2},\omega_{n-1},\omega_n\}
	\\
	& \ell\omega=A^+\chi 
	& & \ell\omega_n \in \spn\{\chi_n,\chi_{n+1},\chi_{n+2}\}
	\end{aligned}$
\\ \hline
 	\parbox{171pt}
	{\begin{center} \vskip-7pt Almost isospectral \end{center}
	\vskip-14pt 
	(isospectral up to at most 1 point) \vskip5pt}
 &	\parbox{175pt}
	{\begin{center} \vskip-7pt Almost isospectral \end{center}
	\vskip-14pt 
	(isospectral up to at most 2 points) \vskip5pt}
\\ \hline
\end{tabular}
\end{center}

The above similarities are not the only results supporting this version of Darboux for CMV as the unitary analogue of Darboux for Jacobi. What is more, both transformations turn out to be directly related by a natural link between the real line and the unit circle, the so called Szeg\H{o} connection \cite{Sz}. This relation is given in Theorem~\ref{thm:A-AA}, which can be also considered a central result of the paper.

Having said that, major differences between Jacobi and CMV will appear in the matrix realization of the inverse Darboux transformations, $M=AA^+ \leftarrow N=A^+A$. We refer to such a matrix realization as the Darboux transformations with parameters because they provide parametric solutions originated by the lack of uniqueness of the reversed Cholesky factorizations $N=A^+A$. In the Jacobi case these parametric solutions coincide with the Jacobi matrices obtained by inverse Darboux and depend on a single real parameter. However, they depend on four real parameters in the CMV case, a symptom of a deeper dissimilarity between Jacobi and CMV: Darboux with parameters for CMV not only yields the solutions of inverse Darboux, but also leads to spurious solutions which are not CMV, not even band nor unitary. The origin of this difference with respect to Jacobi is that transforming a unitary matrix into a Hermitian one requires a Laurent polynomial with at least two zeros instead of a polynomial of degree one. This resembles the situation in the case of higher degree Darboux for Jacobi, based on transforming the Jacobi matrix by a polynomial of degree greater than one. Hence, the matrix implementation of inverse Darboux for CMV needs a constraint on the reversed Cholesky factorizations to select those leading to CMV solutions. This constraint is given in Theorem~\ref{thm:spurious}, the second main result of the paper, which summarizes equivalent ways of separating CMV and spurious solutions. 

The need to deal with spurious solutions which are neither unitary nor band matrices requires a careful manipulation of infinite matrices for the Darboux transformations of CMV matrices. Properties like the associativity of matrix multiplication or the uniqueness of the inverse matrix may fail \cite{Cooke}. These troubles are not so evident for Jacobi matrices due to the absence of spurious solutions. Luckily enough, the situation in the CMV case will be successfully handled thanks to the lower Hessenberg type structure of the spurious solutions (only finitely many upper diagonals are non-null). These reasons make it advisable to specify from the very beginning the kind of matrix operations that will be admissible and their properties. This is necessary to legitimate matrix manipulations, but also sheds light on some aspects of Darboux transformations which, even in the Jacobi case, are not usually explicitly addressed.  

Darboux transformations for CMV matrices can be compared with previous matrix transformations based on factorizations. A precedent of this is the QR algorithm for CMV matrices, equivalent to a special type of Laurent polynomial modifications of the orthogonality measure \cite{W1993}. Nevertheless, we will see that QR has a more limited applicability than Darboux, which can be considered as a way to extend the QR algorithm to general Laurent polynomial modifications of measures. 

Analogously to the Jacobi case, Darboux makes sense for higher degree Laurent polynomial transformations of CMV matrices, as well as for situations related to quasi-definite functionals, although this requires the generalization of Cholesky factorizations and CMV matrices beyond the positive definite case \cite{CMV2003}.  

The above ideas are developed in the paper according to the following schedule: Section~\ref{sec:DT-LP-HM} introduces the basic setting in which the rest of the paper will be conducted. It covers the explicit description of the algebra of infinite matrices that will be used, the analysis of the matrix representations of the multiplication operator with respect to ordered bases of Laurent polynomials (`zig-zag' bases) and general Darboux factorizations for Hermitian Laurent polynomials evaluated on such matrix representations. A particularization of these factorizations leads to the Darboux transformations for CMV matrices in Section~\ref{sec:Chris}, whose main result, Theorem~\ref{thm:Darboux}, gives several characterizations of such transformations. Inverse Darboux transformations and their matrix realization, the Darboux transformations with parameters, are accounted for in Section~\ref{sec:Ger}. It includes the discussion about spurious solutions and the characterization of the CMV ones in Theorem~\ref{thm:spurious}, the central result. Section~\ref{sec:J-CMV} compares Darboux for Jacobi and CMV, linking them via the Szeg\H o mapping in Theorem~\ref{thm:A-AA}, and showing their relation with a new connection between the real line and the unit circle recently obtained \cite{DVZ}. In Section~\ref{sec:D-QR} we find a comparison between Darboux transformations and the QR algorithm for CMV matrices, which also serves to rewrite the former ones in operator language. Section~\ref{sec:IS} uncovers the close relation between Darboux for CMV and certain integrable systems, namely, the Schur flows and the Ablowitz-Ladik model. Higher degree Darboux transformations and the extension to quasi-definite functionals appear in Section~\ref{sec:FE}. Section~\ref{sec:CO} summarizes the conclusions, pointing out some open problems suggested by the previous results. Finally, Appendix~\ref{app:cholesky} deals with the existence and uniqueness of Cholesky factorizations for infinite matrices, crucial issues in the development of the paper. Some illustrative examples of Darboux transformations for CMV matrices can be found at the end of Sections~\ref{sec:Ger}, \ref{sec:J-CMV} and \ref{ssec:quasi}.

\section{Hessenberg type matrices, zig-zag bases and Darboux factorizations}
\label{sec:DT-LP-HM}

The Darboux transformation of a Jacobi matrix $\cal J$ follows from a factorization of the symmetric matrix polynomial $\wp({\cal J}) = {\cal J}+\beta I$ for some parameter $\beta\in\R$ (see Section~\ref{sec:J-CMV}). The translation of this idea to the case of a CMV matrix $\cal C$ requires the substitution of the polynomial $\wp(x)=x+\beta$ by a Hermitian Laurent polynomial $\ell(z)=\alpha z+\beta+\overline\alpha z^{-1}$, $\alpha\in\C\setminus\{0\}$, $\beta\in\R$, so that the matrix Laurent polynomial $\ell({\cal C})=\alpha {\cal C}+\beta I+\overline\alpha {\cal C}^{-1}$ becomes self-adjoint too due to the unitarity of $\cal C$. A suitable factorization of $\ell({\cal C})$ should provide the CMV version of the Darboux transformation.

A central role in Darboux for Jacobi matrices is played by the orthogonal polynomials on the real line since their recurrence relation is codified by a Jacobi matrix. The corresponding CMV analogue are the orthogonal Laurent polynomials on the unit circle, which are expected to be an essential ingredient in Darboux for CMV. Among other things, the orthogonal polynomials give to Jacobi matrices the meaning of a matrix representation of a symmetric multiplication operator. This is also true for orthogonal Laurent polynomials and CMV matrices, but in this case the multiplication operator is unitary instead of symmetric. 

A special feature of Darboux for CMV matrices is the apparition of spurious transformations involving non-CMV (actually, even non-unitary) matrices with no band structure, which are related to non-orthogonal Laurent polynomials. Therefore, we will need to deal with general ordered sequences of Laurent polynomials (assuming no orthogonality property) and related matrix representations of a multiplication operator, which forces us to work in a general setting, avoiding any {\it a priori} assumption of a Hilbert space structure. This means that the space of functions that we will consider is simply the complex vector space of Laurent polynomials $\Lambda=\spn\{z^n\}_{n\in\Z}=\C[z,z^{-1}]$. Besides, we will work with general infinite matrices not necessarily attached to operators on Hilbert spaces, a fact that requires a special care with matrix manipulations because the algebra of general infinite matrices suffers from certain pathological problems \cite{Cooke}. We will briefly describe some of these problems below with the aim of stating a simple setting where such pathologies can be well controlled by a few easy rules. This will be enough for our purposes, but the reader can find more comprehensive treatments of infinite matrices in classical references like \cite{Cooke,Maddox} or the recent review \cite{ShSi}. 

The first difficulty of dealing with infinite matrices is that matrix products can be ill defined, as it is the case of $S^tS$ where
\begin{equation} \label{eq:S}
S = 
\left(
\begin{smallmatrix}
	1 & 0 & 0 & 0 & 0 & \cdots
	\\[3pt]
	1 & 1 & 0 & 0 & 0 & \cdots
	\\[3pt]
	1 & 1 & 1 & 0 & 0 & \cdots
	\\[3pt]
	1 & 1 & 1 & 1 & 0 & \cdots
	\\[3pt]
	\cdots & \cdots & \cdots & \cdots & \cdots & \cdots
\end{smallmatrix}
\right).
\end{equation}
To avoid this problem we will restrict matrix products to those whose coefficients can be computed with a finite number of algebraic operations on the matrix coefficients of the factors.  

\begin{defn} \label{def:MP}
We will say that the product $AB$ of two matrices $A,B$ is {\sl admissible} if any matrix entry $(AB)_{i,k}=\sum_jA_{i,j}B_{j,k}$ involves only a finite ($i,k$-dependent) number of non-null summands.
\end{defn} 

The above definition applies not only to square matrices, but also to rectangular ones, so we can talk about admissible products between matrices and vectors, or between vectors. In what follows we will consider only admissible matrix products.

The following special type of infinite matrices, which will arise in the Darboux transformations for CMV matrices, provide admissible matrix products.

\begin{defn} \label{def:Hess}
We will say that a matrix $A$ is lower (upper) {\sl Hessenberg type} if $A_{i,j}$ can be non-null only for $j-i\le N$ ($i-j\le N$) for some $N\in\Z$.
\end{defn}

Graphically, lower Hessenberg type matrices are characterized by the shape 
$$
\left(
\begin{smallmatrix}
	* & * & \cdots & * & 0 & 0 & 0 & 0 & \cdots
	\\[3pt]
	* & * & \cdots & * & * & 0 & 0 & 0 & \cdots
	\\[3pt]
	* & * & \cdots & * & * & * & 0 & 0 & \cdots
    \\[3pt]
	* & * & \cdots & * & * & * & * & 0 & \cdots
    \\[3pt]  
    \cdots & \cdots & \cdots & \cdots & \cdots & \cdots & \cdots & \cdots & 	\cdots 
\end{smallmatrix}
\right),
$$
while upper Hessenberg type matrices correspond to the transpose of this structure. 

Hessenberg type matrices are not necessarily square matrices. In particular, any column (row) vector is trivially a lower (upper) Hessenberg type matrix, and if it has finitely many non-null components then it is also an upper (lower) Hesssenberg type matrix.
 
Any product $AB$ with $A$ lower Hessenberg type or $B$ upper Hessenberg type is admissible. Instances of this are $ST=I=TS$ and $S^tT=T-I$, where $I$ is the infinite identity matrix, $S$ is given in \eqref{eq:S} and 
\begin{equation} \label{eq:T}
T = 
\left(
\begin{smallmatrix}
	1 & 0 & 0 & \kern1pt 0 & \kern2pt 0 & \kern1pt \cdots
	\\[3pt]
	-1 & 1 & 0 & \kern1pt 0 & \kern2pt 0 & \kern1pt \cdots
	\\[3pt]
	0 & -1 & 1 & \kern1pt 0 & \kern2pt 0 & \kern1pt \cdots
	\\[3pt]
	0 & 0 & -1 & \kern1pt 1 & \kern2pt 0 & \kern1pt \cdots
	\\[3pt]
	\cdots & \cdots & \cdots & \kern1pt \cdots & \kern2pt \cdots 
	& \kern1pt \cdots
\end{smallmatrix}
\right).
\end{equation}
Indeed, the set of lower (upper) Hessenberg type matrices is closed under multiplication. Band matrices are precisely those which are simultaneously upper and lower Hessenberg type, thus they are closed under multiplication too, providing always admissible matrix products regardless whether they act as left or right factors.  

Some of the properties of the multiplication of finite matrices are also valid for admissible products of infinite matrices. For instance, the distributive property $A(B+C)=AB+AC$ holds as long as the products $AB$ and $AC$ are admissible, and analogously for $(B+C)A=BA+CA$. Besides, if $AB$ is admissible, then $B^+A^+$ is admissible too and $(AB)^+=B^+A^+$, where $A^+$ denotes the adjoint (transpose conjugate) of $A$. 

However, even if all the involved matrix products are admissible, the associativity can fail, as it is shown by the example
$$
S^t(TS)=S^tI=S^t, \qquad (S^tT)S = (T-I)S = I-S.
$$ 
The root of this problem is the non-associative character of oscillatory series which appear in the multiplication process. To see this clearly, let us fix our attention on the (0,0)-entry of the above products. Denoting $X=(1,1,1,\dots)$,
$$
\begin{aligned}
& (S^t(TS))_{0,0} = X(TX^t) = 
\sum_{j\ge0} X_j \left(\sum_{k\ge0} T_{j,k}X_k\right) = 
X_0^2 + \sum_{j\ge1} X_j (-X_{j-1}+X_j) = 1, 
\\
& ((S^tT)S)_{0,0} = (XT)X^t = 
\sum_{k\ge0} \left(\sum_{j\ge0} X_jT_{j,k}\right) X_k = 
\sum_{k\ge0} (X_k-X_{k+1}) X_k = 0,
\end{aligned}
$$
thus the different value of these entries is a consequence of the non-associativity of the oscillatory series $1-1+1-1+\cdots$.

We can guarantee the associativity of matrix products requiring the absence of infinite series in the multiplication process. More precisely, assuming that the products $(AB)C$ and $A(BC)$ are admissible, a sufficient condition for the associativity $(AB)C = A(BC)$ is the existence for any indices $i,l$ of only a finite ($i,l$-dependent) number of non-null terms $A_{i,j}B_{j,k}C_{k,l}$ since in that case 
$$
\sum_k \left(\sum_j A_{i,j}B_{j,k}\right) C_{k,l} 
= \sum_{j,k} A_{i,j}B_{j,k}C_{k,l} = 
\sum_j A_{i,j} \left(\sum_k B_{j,k}C_{k,l}\right)
$$
due to the presence of a finite number of non-null summands in all the sums.
 
The presence of Hessenberg type matrix factors which provide admissible matrix products also guarantee their associativity according to the following simple rules.

\begin{prop} \label{prop:asoc}
The associative property $(AB)C=A(BC)$ of a matrix product is valid in any of the following cases:
\begin{enumerate}
	\item $A$ and $B$ are lower Hessenberg type.
	\item $B$ and $C$ are upper Hessenberg type.
	\item $A$ is lower Hessenberg type and $C$ upper Hessenberg type.
\end{enumerate}
\end{prop}
 
\begin{proof}
Suppose for instance that $A$ and $B$ are lower Hessenberg type. Then, $AB$ is admissible and lower Hessenberg type too, thus $(AB)C$ is admissible. On the other hand, $BC$ is admissible since $B$ is lower Hessenberg type, thus $A(BC)$ is admissible because $A$ is also lower Hessenberg type.
Besides, due the lower Hessenberg type structure of $A$ and $B$, 
$$
\kern-5pt
\left.
\begin{aligned}
	& A_{i,j} \text{ can be non-null only for } j-i\le N
	\\
	& B_{j,k} \text{ can be non-null only for } k-j\le M
\end{aligned}
\right\}
\,\Rightarrow\,
\begin{gathered} 
	A_{i,j}B_{j,k}C_{k,l} \text{ can be non-null only for}
	\\
	j\le i+N \text{ and } k \le i+N+M.
\end{gathered}
$$ 
This ensures the associativity according to the sufficient condition given above. A similar argument proves the associativity in the two remaining cases.
\end{proof} 

Another pathology of admissible products of infinite matrices is that, in contrast to the case of finite matrices, the existence of inverse can be consistent with a non-trivial left or right kernel. This is illustrated by the matrix $T$ given in \eqref{eq:T} which, despite having the matrix \eqref{eq:S} as an inverse, satisfies
$$
(1,1,1,\dots)T = 0.
$$
In what follows we will use the following notation for the left and right kernel of a matrix $A\in\C^{M \times N}$, $M,N\in\N\cup\{\infty\}$,
$$
\ker_L(A) = \{X\in\C^{1\times M}:XA=0\}, 
\qquad 
\ker_R(A) = \{X\in\C^{N\times1}:AX=0\},
$$
where we understand that the vectors $X$ are such that the products are admissible. So, in the case of the matrix $T$ introduced in \eqref{eq:T}, $\ker_L(T)=\spn\{(1,1,1,\dots)\}$ and $\ker_R(T)=\{0\}$. 

The above fact also reveals possible uniqueness problems for the inverse of a general infinite matrix because such an inverse could be modified by adding a matrix with rows and columns lying on the left and right kernel respectively. An example of this is given by
$$
J = 
\left(
\begin{smallmatrix}
	-1 & 1
	\\[3pt] 
	1 & -2 & 1
	\\[3pt]
	& 1 & -2 & 1
	\\[-3pt]
	& & \ddots & \ddots & \ddots
\end{smallmatrix}
\right),
\qquad
K = 
\left(
\begin{smallmatrix}
	0 & 1 & 2 & 3 & \cdots
	\\[3pt] 
	1 & 1 & 2 & 3 & \cdots
	\\[3pt]
	2 & 2 & 2 & 3 & \cdots
	\\[3pt]
	3 & 3 & 3 & 3 & \cdots
	\\[3pt]
	\cdots & \cdots & \cdots & \cdots & \cdots
\end{smallmatrix}
\right)
+ c
\left(
\begin{smallmatrix}
	1 & 1 & 1 & 1 & \cdots
	\\[3pt] 
	1 & 1 & 1 & 1 & \cdots
	\\[3pt]
	1 & 1 & 1 & 1 & \cdots
	\\[3pt]
	1 & 1 & 1 & 1 & \cdots
	\\[3pt]
	\cdots & \cdots & \cdots & \cdots & \cdots
\end{smallmatrix}
\right),
$$
which satisfy $JK=I=KJ$ for any $c\in\C$ because $\ker_L(J)=\spn\{(1,1,1,\dots)\}$ and $\ker_R(J)=\spn\{(1,1,1,\dots)^t\}$. 

The previous example not only shows that a Hessenberg type matrix can have multiple inverses, but also that these inverses can be non-Hessenberg type. Nevertheless, if a lower (upper) Hessenberg type matrix $A$ has a lower (upper) Hessenberg type inverse $B$, no other inverse exists. For instance, if $A$ and $B$ are lower Hessenberg type, due to Proposition \ref{prop:asoc}.1, any other inverse $C$ should satisfy $C = (BA)C = B(AC) = B$. In the upper Hessenberg case the associativity requirements force us to modify the uniqueness proof as $C=C(AB)=(CA)B=B$ according to Proposition \ref{prop:asoc}.2. This result is the origin of the following definition. 

\begin{defn} \label{def:Hinv}
An infinite lower (upper) Hessenberg type matrix $A$ with a lower (upper) Hessenberg type inverse will be called a lower (upper) {\sl H-matrix}. Since the inverse is unique in this case, it will be denoted by $A^{-1}$.
\end{defn}

Examples of H-matrices are $S$ and $T$ given in \eqref{eq:S} and \eqref{eq:T}, with $S=T^{-1}$. Although H-matrices could seem quite scarce at first sight, we will see that they are precisely the kind of matrices arising in the CMV version of Darboux transformations. Actually, CMV matrices themselves are examples of matrices which are simultaneously lower and upper H-matrices.  

A triangular infinite matrix with non-zero diagonal entries is a special case of H-matrix because it always has a triangular inverse. The condition on the diagonal guarantees for the leading submatrices the existence of triangular inverses obtained by enlarging the smaller ones. An induction process on the size of the leading submatrices generates the triangular inverse of the infinite triangular matrix.  

In general, we will denote by $A^{-1}$ the inverse of a matrix $A$ whenever this inverse is unique. Some of the rules for the manipulation of inverses of finite matrices also hold for infinite ones as long as they have a unique inverse. For instance, taking adjoints we see that $AB=I=BA$ is equivalent to $B^+A^+=I=A^+B^+$, thus $A$ and $A^+$ have a unique inverse simultaneously, and in this case $(A^+)^{-1}=(A^{-1})^+$. In particular, this equality is valid for any H-matrix.

H-matrices are essential in the Darboux transformation for CMV matrices since they allow us to perform standard algebraic manipulations which are forbidden for general infinite matrices. For instance, solving a simple matrix equation in the following way, 
\begin{equation} \label{eq:eq}
AB=C \;\Rightarrow\; B=(A^{-1}A)B=A^{-1}(AB)=A^{-1}C,
\end{equation}
is possible if $A$ is a lower H-matrix due to Proposition \ref{prop:asoc}.1, but it is not possible for general infinite matrices. 
As an application of these ideas, we have the following result.

\begin{prop} \label{prop:kerH}
If $A$ is a lower (upper) H-matrix, then $\ker_R(A)=\{0\}$ {\rm(}$\ker_L(A)=\{0\}${\rm)} and $\ker_L(A)\setminus\{0\}$ {\rm(}$\ker_R(A)\setminus\{0\}${\rm)} does not contain finitely non-null vectors, i.e. vectors with finitely many non-null entries.
\end{prop}

\begin{proof}
Suppose that $A$ is a lower H-matrix. Using \eqref{eq:eq} we find that $AX=0 \Rightarrow X=0$. 
On the other hand, if $X$ is a finitely non-null row vector, it can be considered as a lower Hessenberg type matrix, thus $XA=0 \Rightarrow X=X(AA^{-1})=(XA)A^{-1}=0$ due to Proposition \ref{prop:asoc}.1.
The result for upper H-matrices follows analogously from Proposition \ref{prop:asoc}.2.
\end{proof}

The set of lower (upper) H-matrices is closed under multiplication and inversion, and contains the identity matrix. Hence, this set is a multiplicative group since associativity holds for lower (upper) Hessenberg type matrices due to Proposition \ref{prop:asoc}.1 (Proposition \ref{prop:asoc}.2). Actually, the previous results show that the set of infinite lower (upper) Hessenberg type matrices with the sum and multiplication is a ring whose group of units is the set of lower (upper) H-matrices. Interesting subgroups of this multiplicative group are the sets of infinite lower (upper) triangular matrices with positive or non-null diagonal entries.  

Following the general idea of the Darboux transformation, and mimicking the Jacobi case according to previous comments, Darboux for CMV should deal with a factorization $\ell({\cal C})=AB$ for a Hermitian Laurent polynomial $\ell(z)=\alpha z+\beta+\overline\alpha z^{-1}$ ($\alpha\in\C$, $\beta\in\R$), and a matrix $\cal C$ (eventually a CMV matrix) representing in some basis of $\Lambda$ the {\sl multiplication operator}
\begin{equation} \label{eq:z}
 z \colon \kern-5pt
 \mathop{\Lambda \xrightarrow{\kern15pt} \Lambda}
 \limits_{\text{\footnotesize $f(z) \mapsto zf(z)$}}
\end{equation}

Different matrix representations $\cal C$ of this multiplication operator are related by changes of basis and give rise to different factors $A,B$. Besides, we will see that the factors $A,B$ themselves are closely related to certain changes of basis in $\Lambda$. Thus, changes of basis will play a central role in the factorizations related to the Darboux transformation for CMV matrices.
 
In the following section we will analyze the matrix representations of the multiplication operator \eqref{eq:z} for bases of $\Lambda$ with a structure similar to that one of a CMV basis, but assuming no orthogonality requirement. We will see that these representations are H-matrices, a key result to avoid the pathologies of general infinite matrices.

\subsection{Matrix representations of the multiplication operator}
\label{ssec:C}
    
A CMV matrix can be understood as the matrix representation of the multiplication operator \eqref{eq:z} with respect to a basis $\chi=(\chi_0,\chi_1,\chi_2,\dots)^t$ of orthonormal Laurent polynomials. This basis is obtained by applying to the canonical one $\eta=(1,z,z^{-1},z^2,z^{-2},\dots)^t$ the Gram-Schmidt orthonormalization with respect to a positive Borel measure supported on an infinite subset of the unit circle (hereinafter `measure on the unit circle'). Thus, the orthonormal basis $\chi$ satisfies
\begin{equation} \label{eq:zigzag}
\chi_n \in \LL_n:=\left\{
\begin{aligned}
	& \spn\{z^{-k},\dots,z^k\} 
	& & (\text{coefficient of } z^{-k}>0), \quad 
	& & n=2k, 
	\\
	& \spn\{z^{-k},\dots,z^{k+1}\} 
	& & (\text{coefficient of } z^{k+1}>0), 
	& & n=2k+1.
\end{aligned}
\right.
\end{equation}
Due to the needs of Darboux for CMV pointed out previously, we will consider in this section bases $\chi$ of $\Lambda$ satisfying simply \eqref{eq:zigzag}, regardless of their orthogonality properties. 

\begin{defn} \label{def:zigzag}
A basis $\chi$ of $\Lambda$ satisfying \eqref{eq:zigzag} will be called a {\sl zig-zag} basis. 
\end{defn}

The nested structure of the subspaces $\LL_n$ implies that, for any zig-zag basis $\chi$,
$$
z\chi_0 \in \LL_1\setminus\LL_0,
\qquad
z\chi_n \in 
\begin{cases}
	\LL_{n+1}, & \text{ even } n,
	\\
	\LL_{n+2}\setminus\LL_{n+1}, & \text{ odd } n.
\end{cases}
$$
This can be rewritten in matrix form as  
\begin{equation} \label{eq:C}
z\chi = {\cal C} \chi,
\qquad
{\cal C} = 
\left(
\begin{smallmatrix}  
	* & + 
	\\[3pt] 
	* & * & * & + 
	\\[3pt]
	* & * & * & * 
	\\[3pt]
	* & * & * & * & * & + 
	\\[3pt]
	* & * & * & * & * & * 
	\\[3pt]
	* & * & * & * & * & * & * & + 
	\\[3pt]
	* & * & * & * & * & * & * & *
	\\[3pt]
	\cdots & \cdots & \cdots & \cdots & \cdots & \cdots & \cdots & \cdots 
	& \cdots
\end{smallmatrix}
\right),
\end{equation}
where the coefficients denoted by $+$ are positive and the omitted ones are zero. The matrix representation ${\cal C}$ of the multiplication operator \eqref{eq:z} in a zig-zag basis $\chi$ is therefore lower Hessenberg type.

The inverse $z^{-1}$ of the multiplication operator is represented in a zig-zag basis $\chi$ by a lower Hessenberg type matrix too. Indeed,
$$
z^{-1}\chi_0 \in \LL_2\setminus\LL_1,
\qquad
z^{-1}\chi_n \in 
\begin{cases}
	\LL_{n+2}\setminus\LL_{n+1}, & \text{ even } n,
	\\
	\LL_{n+1}, & \text{ odd } n,
\end{cases}
$$
or equivalently  
\begin{equation} \label{eq:Cinv}
z^{-1}\chi = \widetilde{\cal C} \chi,
\qquad
\widetilde{\cal C} = 
\left(
\begin{smallmatrix}  
	* & * & + 
	\\[3pt] 
	* & * & *  
	\\[3pt]
	* & * & * & * & + 
	\\[3pt]
	* & * & * & * & * 
	\\[3pt]
	* & * & * & * & * & * & +
	\\[3pt]
	* & * & * & * & * & * & *
	\\[3pt]
	\cdots & \cdots & \cdots & \cdots & \cdots & \cdots & \cdots & \cdots
\end{smallmatrix}
\right).
\end{equation}
On the other hand, from Proposition \ref{prop:asoc}.1 we find that $(\widetilde{\cal C}{\cal C}-I)\chi = (z^{-1}z-1)\chi=0$, thus $\widetilde{\cal C}{\cal C}=I$ by the linear independence of $\chi$. A similar proof yields ${\cal C}\widetilde{\cal C}=I$. Hence, $\widetilde{\cal C}$ is a lower Hessenberg type inverse of $\cal C$, which thus has a unique inverse ${\cal C}^{-1}=\widetilde{\cal C}$. 

In consequence, the matrix representation $\cal C$ of the multiplication operator in a zig-zag basis $\chi$ is always an H-matrix.
  
The simplest case corresponds to the canonical basis $\eta$, which leads to     
\begin{equation} \label{eq:shift}
z \eta = {\cal S} \eta,
\qquad
{\cal S} = 
\left(
\begin{smallmatrix}
  	0&1
	\\[3pt]
  	0 & 0 & 0 & 1
	\\[3pt]
  	1 & 0 & 0 & 0
	\\[3pt]
  	0 & 0 & 0 & 0 & 0 & 1
	\\[3pt]
  	0 & 0 & 1 & 0 & 0 & 0
	\\[3pt]
  	\cdots & \cdots & \cdots & \cdots & \cdots & \cdots & \cdots
\end{smallmatrix}
\right).
\end{equation}
The matrix $\cal S$ is unitary, i.e.  ${\cal S}^{-1} = {\cal S}^+$. We will refer to $\cal S$ as the {\sl shift matrix} since it represents the multiplication shift operator $z^n \mapsto z^{n+1}$ with respect to the basis $\eta$.

Any zig-zag basis $\chi$ is characterized up to a positive factor by the corresponding matrix representation $\cal C$ of the multiplication operator. A way to understand this is to note that $\cal C$ determines $\widetilde{\cal C}={\cal C}^{-1}$, and the first and even equations of \eqref{eq:C} combined with the odd equations of \eqref{eq:Cinv} give each Laurent polynomial $\chi_n$ as a linear combination of the previous ones $\chi_0,\dots,\chi_{n-1}$ and 
$$
\begin{aligned}
& z\chi_0 & & n=0,
\\
& z\chi_{n-2} & & \text{odd } n, 
\\
& z^{-1}\chi_{n-2} & & \text{even } n>0.
\end{aligned}
$$
Therefore, $\cal C$ and $\chi_0$ determine $\chi_n$ for $n\ge1$.  

The matrix representations of the multiplication operator in a zig-zag basis do not cover all the H-matrices with the shape given in \eqref{eq:C}. A simple counterexample is obtained by slightly perturbing the shift matrix changing 0 by 1 in the (1,2) and (2,3) coefficients,
$$
H = \left(
\begin{smallmatrix}
  	0&1
	\\[3pt]
  	0 & 0 & \kern2.5pt 1 \kern-7.95pt \bigboxvoid & 1
	\\[2pt]
  	1 & 0 & 0 & \kern2.5pt 1 \kern-7.95pt \bigboxvoid
	\\[2pt]
  	0 & 0 & 0 & 0 & 0 & 1
	\\[3pt]
  	0 & 0 & 1 & 0 & 0 & 0
	\\[3pt]
  	0 & 0 & 0 & 0 & 0 & 0 & 0 & 1
	\\[3pt]
  	0 & 0 & 0 & 0 & 1 & 0 & 0 & 0
	\\[3pt]
  	\cdots & \cdots & \cdots & \cdots & \cdots & \cdots & \cdots & \cdots & 	\cdots
\end{smallmatrix}
\right), 
\qquad\quad
H^{-1} =
\left(
\begin{smallmatrix}
  	0 & -1 & 1 & 0 & \kern2.5pt 1 \kern-7.95pt\bigovoid
	\\[1pt]
  	1 & 0 & 0
	\\[3pt]
  	0 & 0 & 0 & 0 & 1
	\\[3pt]
  	0 & 1 & 0 & 0 & -1
	\\[3pt]
  	0 & 0 & 0 & 0 & 0 & 0 & 1
	\\[3pt]
  	0 & 0 & 0 & 1 & 0 & 0 & 0
	\\[3pt]
  	0 & 0 & 0 & 0 & 0 & 0 & 0 & 0 & 1
	\\[3pt]
  	\cdots & \cdots & \cdots & \cdots & \cdots & \cdots & \cdots & \cdots & 	\cdots & \cdots
\end{smallmatrix}
\right).
$$
Thus, $H$ is an H-matrix with the shape \eqref{eq:C}, but its inverse does not fit with the structure \eqref{eq:Cinv} due to the encircled coefficient 1 at site (0,4).    
 
Changes of basis allow us to specify the set of matrix representations under study. By definition, zig-zag bases $\chi$ are those related to the canonical one $\eta$ by
$$
\chi = T \eta, \qquad 
T \in \mathscr{T} = \; \parbox{195pt}{set of lower triangular infinite matrices \centerline{with positive diagonal entries.}}
$$     
Therefore, Proposition \ref{prop:asoc}.1 leads to $({\cal C}T-T{\cal S})\eta = {\cal C}\chi-z\chi=0$, so the linear independence of $\eta$ implies that ${\cal C}T=T{\cal S}$. Then, using again Proposition \ref{prop:asoc}.1 yields ${\cal C}=T{\cal S}T^{-1}$, where parenthesis are omitted due to the associativity of the product. 

Hence, the matrix representations of the multiplication operator in a zig-zag basis constitute the equivalence class   
$$
[{\cal S}]=\{T{\cal S}T^{-1} : T\in\mathscr{T}\}
$$
of the shift matrix with respect to the relation ``conjugation by the subgroup $\mathscr{T}$", which is an equivalence relation in the multiplicative group of lower H-matrices due to the associativity properties of Proposition \ref{prop:asoc}. 

Another subgroup of H-matrices with special interest for us is the set of lower (upper) H-matrices $A$ which are unitary, i.e. $A^{-1}=A^+$. Every such a matrix $A$ is necessarily banded because $A=(A^{-1})^+$ is simultaneously lower and upper Hessenberg type. This ensures that both $A^+A$ and $AA^+$ are admissible products for unitary H-matrices. However, analyzing the unitarity of an H-matrix whose band structure is not guaranteed {\it a priori} requires avoiding to deal with the product $A^+A$ ($AA^+$) for a lower (upper) H-matrix $A$, since it could be non-admissible. Nevertheless, the following result shows that the unitarity of an H-matrix can be checked resorting only to admissible products.  

\begin{prop} \label{prop:uni-H}
A lower (upper) H-matrix $A$ is unitary iff $AA^+=I$ ($A^+A=I$).
\end{prop} 

\begin{proof}
Suppose that $A$ is a lower H-matrix, so that it has a unique inverse $A^{-1}$ and this inverse is lower Hessenberg type. Then, $A^{-1}=A^+$ implies that $AA^+=I$. Conversely, $AA^+=I$ leads to $A^{-1}=A^{-1}(AA^+)=(A^{-1}A)A^+=A^+$, where we have used Proposition \ref{prop:asoc}.1. A similar proof runs for upper H-matrices.
\end{proof}

The unitary representatives of the class $[{\cal S}]$ have a special meaning.
   
\begin{prop} \label{prop:CMVshape}
If $\cal C \in [\cal S]$ is unitary, then it is a band matrix with zig-zag shape
\begin{equation} \label{eq:CMVshape}
\cal C =
\left(
\begin{smallmatrix}
	* & + 
	\\[3pt]
   	* & * & * & + 
   	\\[3pt]
     + & * & * & * 
     \\[3pt]
   	&& * & * & * & + 
	\\[3pt]
   	&& + & * & * & * 
    \\[3pt]
   	&&&& * & * & \kern-1pt * & \kern-3pt + 
	\\[3pt]
   	&&&& + & * & \kern-1pt * & \kern-3pt * 
	\\[3pt]
    &&&&&& \cdots & \kern-1pt \cdots & \kern-1pt \cdots
\end{smallmatrix}
\right).
\end{equation}
\end{prop}
\begin{proof} 
Bearing in mind the shapes of $\cal C$ and $\widetilde{\cal C}={\cal C}^{-1}$ given in \eqref{eq:C} and \eqref{eq:Cinv}, respectively, the result follows directly from the equality ${\cal C}^+={\cal C}^{-1}$. 
\end{proof}

It was proved in \cite[Theorem 3.2]{CMV2005} that the unitary matrices with the zig-zag shape \eqref{eq:CMVshape} are exactly the CMV matrices. Therefore, we get the following identification, which provides an alternative definition of CMV matrices.
       
\begin{cor} \label{cor:CMVdef} 
${\cal C}$ is a CMV matrix iff ${\cal C} \in [{\cal S}]$ and is unitary.
\end{cor}
 
The previous result is an algebraic characterization of CMV matrices based on a previous characterization in terms of their shape. An explicit parametrization of CMV matrices is known, and we will introduce it later on.

Zig-zag bases which are orthonormal with respect to a measure on the unit circle yield the CMV matrix representations of the multiplication operator. Bearing in mind the one-to-one correspondence between elements of $[{\cal S}]$ and zig-zag bases of $\Lambda$ (up to a positive factor), Corollary \ref{cor:CMVdef} implies that the unitarity of an element of $[{\cal S}]$ is equivalent to the orthonormality of the corresponding zig-zag basis with respect to a measure on the unit circle. 

The class $[{\cal S}]$ constitutes the framework where we will develop the factorization leading to the Darboux transformation of CMV matrices. The study of this factorization is the objective of the following section.
  
\subsection{Change of basis and Darboux factorization}
\label{ssec:Dfac}
    
Let us introduce the Darboux factorization for the class $[\cal S]$. For this purpose, consider a Hermitian Laurent polynomial $\ell\in\Lambda\setminus\{0\}$, i.e. such that $\ell_*=\ell$, where the substar operation in $\Lambda$ is defined by 
\begin{equation} \label{eq:substar}
\ell_*(z):=\overline{\ell(1/\overline{z})}. 
\end{equation}
The general form of such an $\ell$ is
$$
\ell(z) = \sum_{j=-d}^d \alpha_j z^j,
\qquad \alpha_j\in\C, 
\qquad \alpha_{-j}=\overline\alpha_j,
\qquad \alpha_d\ne0.
$$
Although most of the discussions along the paper are for the simplest non-trivial case, $d=1$, bearing in mind future extensions to $d>1$ (see Section~\ref{ssec:deg}), we will not assume any restriction on $d$ in the present section. For convenience, sometimes we will distinguish the `degree' $d$ of $\ell$ using the number $N_z=2d$ of zeros of $\ell$ (counting multiplicity) rather than the value of $d$ itself.

Due to the ring structure of the set of infinite lower Hessenberg type matrices, $\ell({\cal C})$ is of this type for any ${\cal C}\in[{\cal S}]$. The alluded factorization will have the form $\ell({\cal C})=AB$ for some lower Hessenberg type matrices $A,B$. The hermiticity of $\ell$ ensures that $\ell({\cal C})$ is self-adjoint for a unitary $\cal C$, i.e. when $\cal C$ is a CMV matrix. This situates the Darboux factorization of CMV matrices close to the standard setting of Darboux for self-adjoint operators. 

Let $\chi$ be a zig-zag basis related to ${\cal C}\in[{\cal S}]$. We will see that any choice of a new zig-zag basis $\omega$ generates a factorization of $\ell({\cal C})$. First, note that $\chi$ and $\omega$ are related by a triangular change of basis with positive diagonal entries, i.e. 
\begin{equation} \label{eq:A}
\chi = A \omega, \qquad A\in\mathscr{T}.
\end{equation}
A new matrix $B$ comes from expressing $\ell\omega$ in the basis $\chi$,   
\begin{equation} \label{eq:B}
\ell\omega=B\chi.
\end{equation}
While $A$ is a lower H-matrix because it lies on $\mathscr{T}$, for the moment we only can assure that $B$ is lower Hessenberg type. This last statement follows from the relation
$$
\ell\omega_n \in \ell(\LL_n\setminus\LL_{n-1}) 
\subset \LL_{N_z+n}\setminus\LL_{N_z+n-1},
\qquad 
N_z = \text{number of zeros of } \ell,
$$
which shows that $B$ has the shape
\begin{equation} \label{eq:Bshape}
\begin{aligned}
& \kern36pt \overbracket{\kern43pt}^{N_z}
\\[-7pt]
& B = 
\left(
\begin{smallmatrix}
	* & * & \cdots & * & \circledast 
	\\[3pt]
	* & * & \cdots & * & * & \circledast
	\\[3pt]
	* & * & \cdots & * & * & * & \circledast
    \\[3pt]
	* & * & \cdots & * & * & * & * & \circledast
    \\[3pt]  
    \cdots & \cdots & \cdots & \cdots & \cdots & \cdots & \cdots & \cdots & 	\cdots 
\end{smallmatrix}
\right),
\end{aligned}
\end{equation}
where the symbol $\circledast$ stands for a non-null coefficient. 

The matrix $B$ is never a lower H-matrix for a non-constant $\ell$ because it has a non-trivial right kernel. This follows from \eqref{eq:B}, which implies that $B\chi(\zeta)=0$ for any zero $\zeta$ of $\ell$. Actually, $B$ may have no inverse at all. This occurs for instance if $\ell(1)=0$, $\omega=\eta$ and $\chi=T\eta$ with $T$ given by \eqref{eq:T}, because then $0=B\chi(1)=B(1,0,0,\dots)^t$ so that the first column of $B$ vanishes. Nevertheless, we will see later on that in cases of interest $B$ becomes an upper H-matrix with a band structure.

The matrices $A$ and $B$ are the key ingredients of a factorization not only of $\ell({\cal C})$, but also of $\ell({\cal D})$, where $\cal D \in [{\cal S}]$ is the representation of the multiplication operator in the basis $\omega$.
  
\begin{prop} \label{prop:D-fac} 
Let ${\cal C},{\cal D}$ be the matrix representations of the multiplication operator with respect to the zig-zag bases $\chi,\omega$ respectively, and let $\ell$ be a Hermitian Laurent polynomial. Then, the matrices $A,B$ defined by \eqref{eq:A} and \eqref{eq:B} satisfy the following identities:
\begin{enumerate}
	\item $\ell({\cal C})=AB$ and $\ell({\cal D})=BA$.
	\item ${\cal C}A=A{\cal D}$ and $B{\cal C}={\cal D}B$.
	\item ${\cal C}=A{\cal D}A^{-1}$ and ${\cal D}=A^{-1}{\cal C}A$.
\end{enumerate}
\end{prop}

\begin{proof} From \eqref{eq:C}, \eqref{eq:Cinv}, the analogous relations for $\cal D$ and ${\cal D}^{-1}$, \eqref{eq:A}, \eqref{eq:B} and Proposition \ref{prop:asoc}.1, we obtain
$$
\begin{aligned}
	& (\ell({\cal C})-AB)\chi = \ell\chi-\ell A\omega = 0,
	\\[3pt]
	& (\ell({\cal D})-BA)\omega = \ell\omega-B\chi = 0,
	\\[3pt]
	& ({\cal C}A-A{\cal D})\omega = {\cal C}\chi-zA\omega 
	= z\chi-zA\omega = 0, 
	\\[3pt]
	& (B{\cal C}-{\cal D}B)\chi = zB\chi-\ell{\cal D}\omega 
	= zB\chi-z\ell\omega = 0. 
\end{aligned}
$$
Relations 1 and 2 follow from the linear independence of $\chi$ and $\omega$. The remaining equalities are a consequence of the first identity in 2 and Proposition \ref{prop:asoc}.1. 
\end{proof}

Parenthesis are omitted in Proposition \ref{prop:D-fac}.3 due to the associativity of products of lower H-matrices. We will use this standard convention for associative products in what follows, making explicit the parenthesis only when possible non-associative products appear.  

Proposition \ref{prop:D-fac}.1 is what we will call the Darboux factorization for the matrices in the class $[{\cal S}]$. It is worth remarking that this factorization is not defined for single elements of $[{\cal S}]$, but for every ordered pair of matrices in this class. Besides, there is a trivial dependence on the zig-zag basis chosen for each pair, since these bases are defined up to a positive factor. This degree of freedom amounts to the substitution $A\to c A$, $B\to c^{-1}B$, $c>0$, in Proposition \ref{prop:D-fac}. In the next section we will study a particularization of these factorizations for CMV matrices. 

Proposition \ref{prop:D-fac}.3 is a direct consequence of identifying the factor $A$ as the change of basis between $\chi$ and $\omega$. We cannot formulate the analogue of this statement for the second identity of Proposition \ref{prop:D-fac}.2 due to the invertibility problems of $B$. 

The identities of Proposition \ref{prop:D-fac}.3 provide a way to obtain $\cal D$ starting from $\cal C$ and the factorization $\ell({\cal C})=AB$, as well as a way to obtain $\cal C$ starting from $\cal D$ and the factorization $\ell({\cal D})=BA$. Despite the apparent symmetry, there is a significant difference between these two identities when compared to the equality ${\cal C}A=A{\cal D}$ of Proposition \ref{prop:D-fac}.2 which originates them. As in \eqref{eq:eq}, the equation ${\cal C}A=A{\cal Y}$ has a unique solution in the unknown matrix $\cal Y$ because $A$ is a lower H-matrix. However, the equation ${\cal X}A=A{\cal D}$ can have multiple solutions in the unknown matrix $\cal X$ because $\ker_L(A)$ can be non-trivial. The general solution has the form ${\cal X} ={\cal C}+K_A$ where ${\cal C}=A{\cal D}A^{-1}$ and $K_A$ is any matrix whose rows belong to $\ker_L(A)$. Of course, due to Proposition \ref{prop:asoc}.1, we know that imposing a lower Hessenberg type structure on $\cal X$ leaves only the solution $\cal C$. The conclusion is that no non-null lower Hessenberg type matrix $K_A$ can be generated from $\ker_L(A)$. This is in agreement with Proposition~\ref{prop:kerH}.

We will finish this section with an exact description of $\ker_R(B)$ which will be of interest later on. 

\begin{prop} \label{prop:kerB}
Let ${\cal C}$ be the matrix representation of the multiplication operator with respect to the zig-zag basis $\chi$ and let $\ell$ be a Hermitian Laurent polynomial with $N_z$ zeros counting multiplicity. If $\omega$ is any other zig-zag basis and $B$ is the matrix given by \eqref{eq:B}, then $\ker_R(\ell({\cal C})) = \ker_R(B)$ and a basis of this subspace is given by
$$ 
\bigcup_{\substack{
	\zeta \text{ {\rm zero of} } \ell 	
	\\
	m \text{ {\rm multiplicity of} } \zeta}}
\kern-20pt
\{\chi(\zeta),\chi'(\zeta),\dots,\chi^{(m-1)}(\zeta)\}.
$$
Moreover, a matrix of $\mathscr{K}_B=\{K=K^+:BK=0\}$ vanishes iff its leading submatrix of order $N_z$ is null.
\end{prop}  
 
\begin{proof} 
First of all note that, due to the structure \eqref{eq:Bshape} of $B$, $\dim\ker_R(B)$ cannot be greater than the number $N_z$ of zeros of $\ell$. 

Besides, using the fact that $A$ is a lower H-matrix and $B$ is lower Hessenberg type we find from Proposition \ref{prop:asoc}.1 that $(AB)X=0 \Leftrightarrow BX=0$. Bearing in mind Proposition \ref{prop:D-fac}.1, this means that $\ker_R(\ell({\cal C})) = \ker_R(B)$.

Taking derivatives in \eqref{eq:B} we conclude that $\chi^{(j)}(\zeta)\in\ker_R(B)$ for any zero $\zeta$ of $\ell$ and any $j=0,1,\dots,m-1,$ smaller than its multiplicity $m$. Since the set
$$
\bigcup_{\substack{
	\zeta \text{ {\rm zero of} } \ell 	
	\\
	m \text{ {\rm multiplicity of} } \zeta}}
\kern-20pt
\{\chi(\zeta),\chi'(\zeta),\dots,\chi^{(m-1)}(\zeta)\}
$$
has $N_z$ vectors lying on $\ker_R(B)$ and $\dim\ker_R(B)\le N_z$, to prove the first statement of the proposition it only remains to see that this set is linearly independent. 

For this purpose it is enough to show that $\det\Delta_{N_z}\ne0$ where, denoting by $\zeta_1,\dots,\zeta_k$ the zeros of $\ell$, $\Delta_{N_z}$ is the leading submatrix of order $N_z$ of the matrix $\Delta$ of order $\infty\times N_z$ given by
$$
\begin{gathered}
\Delta = (\Delta(\zeta_1),\Delta(\zeta_2),\dots,\Delta(\zeta_k)),
\qquad
\begin{aligned}
& \Delta(\zeta) = (\chi(\zeta),\chi'(\zeta),\dots,\chi^{(m-1)}(\zeta)),
\\
& m = \text{multiplicity of } \zeta \text{ as a zero of } \ell.
\end{aligned}
\end{gathered}
$$
If $\det\Delta_{N_z}=0$ then there exists a row vector $X\ne0$ such that $X\Delta_{N_z}=0$. This is also equivalent to state that the Laurent polynomial $f=X(\chi_0,\chi_1,\dots,\chi_{N_z-1})^t$ satisfies $f^{(j)}(\zeta)=0$ for any zero $\zeta$ of $\ell$ and any $j=0,1,\dots,m-1,$ smaller than its multiplicity $m$. In other words, $f$ should have at least $N_z$ zeros, counting multiplicities. This is impossible because $f$ is a linear combination of $\chi_0,\dots,\chi_{N_z-1}$, thus it has no more than $N_z-1$ zeros due to the zig-zag structure of the basis $\chi$.

Suppose now that an infinite square matrix $K$ satisfies $BK=0$. This is equivalent to state that the columns of $K$ belong to $\ker_R(B)$, i.e. $K=\Delta V$ for a matrix $V$ of size $N_z\times\infty$. Indicating by the subindex $N_z$ the leading submatrix of order $N_z$, we have that $K_{N_z}=\Delta_{N_z}V_{N_z}$. Since $\det\Delta_{N_z}\ne0$, the condition $K_{N_z}=0$ implies that $V_{N_z}=0$, so the first $N_z$ columns of $V$ and $K$ are null. If, besides, $K=K^+$, then the first $N_z$ rows of $K$ vanish too, which can be expressed as $\Delta_{N_z}V=0$. The condition $\det\Delta_{N_z}\ne0$ implies that $V=0$, thus $K=0$.
\end{proof}

We have described the Darboux factorization in the general framework of matrix representations of the multiplication operator with respect to zig-zag bases. In the next section we will explore some consequences of the previous results for the case of CMV matrix representations, i.e. when the zig-zag bases are orthonormal with respect to a measure on the unit circle. Our purpose is to show that suitable changes of basis connect standard factorizations (usual in the context of Darboux for self-adjoint operators) with well known transformations of measures on the unit circle.

\section{Darboux for CMV and Christoffel transformations}
\label{sec:Chris}

In this section we will analyze an especially interesting Darboux factorization for CMV matrices which is behind the CMV version of Darboux transformations. 

The zig-zag bases leading to CMV matrices are precisely those which are orthonormal with respect to a measure on the unit circle. This establishes a one-to-one correspondence between CMV matrices and such measures, up to normalization.  

Apart from introducing Darboux transformations for CMV matrices, one of the purposes of the present section is to translate these transformations to the corresponding measures. However, in tackling the CMV version of Darboux transformations we will need to deal with zig-zag bases which are not necessarily orthogonal with respect to a measure on the unit circle. This forces us to work in the more general setting of linear functionals in $\Lambda$, a subset of which can be identified with the set of measures on the unit circle. Besides, this general setting will allow us to simplify the notation along the paper, providing at the same time a direct extension of Darboux transformations to quasi-CMV matrices related to quasi-definite functionals which are not necessarily associated with positive measures (see Section~\ref{ssec:quasi}). Thus, we will first comment on the relation between measures on the unit circle and linear functionals in $\Lambda$.

Any measure $\mu$ on the unit circle generates a linear functional $u$ in $\Lambda$ defined by $u[f]=\int f\,d\mu$, so that no different measures give rise to the same functional \cite{GrSz}. For convenience, we will write $u \equiv d\mu$ to indicate the functional $u$ generated by $\mu$. Behind this notational convention lies the identification of a measure on the unit circle and the corresponding functional, much in the same way as in the case of functions and distributions.  

The inner product in $\Lambda$ associated with the measure $\mu$ can be rewritten in terms of its functional $u$ as $\int f\overline{g}\,d\mu = u[fg_*]$, the substar operation in $\Lambda$ being as in \eqref{eq:substar}. This provides a one-to-one correspondence between measures on the unit circle and positive definite Hermitian linear functionals in $\Lambda$, i.e. those linear functionals $u\colon\Lambda\to\C$ satisfying
$$
u=u_* 
\quad (\text{Hermitian}),
\qquad\qquad
u[ff_*]>0, \quad f\in\Lambda\setminus\{0\} 
\quad (\text{positive definite}),
$$
where the substar operation is defined for linear functionals in $\Lambda$ by
$$
u_*[f]=\overline{u[f_*]}, \qquad f\in\Lambda.
$$
Equivalently, $u_*[z^n]=\overline{u[z^{-n}]}$ for every $n\in\Z$.  

We will extend the substar operation to Laurent polynomial matrices $M(z)$ by $M_*(z) := \overline{M(1/\overline{z})}$, defining also the new operation 
$M^+(z) := M_*(z)^t = M(1/\overline{z})^+$. For convenience, in what follows $M^+$ will refer to $M^+(z)$, while the adjoint of $M(z)$ will be explicitly denoted by $M(z)^+$. Then, using the natural notation $u[M]=(u[M_{i,j}])$, the orthonormality of a zig-zag basis $\chi$ with respect to a measure $\mu$ on the unit circle can be compactly expressed as $u[\chi\chi^+]=I$ in terms of the functional $u\equiv d\mu$. This orthonormal basis $\chi$ is completely determined by the positive definite functional $u$.

If $\chi$ is the zig-zag basis which is orthonormal with respect to a measure $\mu$, the functional $u\equiv d\mu$ is determined by the conditions $u[\chi_n\chi_{0*}]=\delta_{n,0}$. Actually, these conditions determine a linear functional $u$ in $\Lambda$ for any zig-zag basis $\chi$ (hereinafter `the functional of the zig-zag basis'). Since the zig-zag basis $\chi$ of a given matrix $\cal C\in[{\cal S}]$ is determined up to a positive factor, this associates a unique functional $u$, up to normalization, to any matrix $\cal C\in[{\cal S}]$. When ${\cal C}$ is a CMV matrix, $u$ is the orthogonality functional of $\chi$, otherwise $u$ can be even non-Hermitian.

The Gram matrix of a general linear functional $u$ in $\Lambda$ with respect to an arbitrary basis $l=(l_0,l_1,\dots)^t$ of $\Lambda$ is defined as $u[l \kern1pt l^+]$, so that it coincides with the standard Gram matrix of an inner product when $u$ is positive definite. The Hermitian functionals are precisely those with a Hermitian Gram matrix and, among them, the positive definite functionals are characterized by a positive definite Gram matrix, i.e. $Xu[l \kern1pt l^+]X^+>0$ for every finitely non-null row vector $X\ne0$ (equivalently, the leading principal minors of the Gram matrix are all positive). Therefore, a basis $l$ of $\Lambda$ is orthonormal with respect to a measure on the unit circle iff there exists a linear functional $u$ such that $u[l \kern1pt l^+]=I$ because the fact that this Gram matrix is trivially Hermitian and positive definite implies that $u$ is Hermitian and positive definite.

Any transformation of a measure on the unit circle can be understood as a transformation of the corresponding functional. For instance, if a Laurent polynomial $\ell$ is non-negative in the support of a positive measure $\mu$ on the unit circle, then $\ell d\mu$ is again a positive measure with associated functional $u\ell[f]:=u[\ell f]$. We use this identity to extend the meaning of $u\ell$ to any linear functional $u$ in $\Lambda$ and any $\ell\in\Lambda$, non of them necessarily Hermitian. Then, it easy to see that $(u\ell)_*=u_*\ell_*$.   
 
Let us introduce now the specific factorization which will be the origin of the Darboux transformations for CMV matrices.
 
Given a CMV matrix $\cal C$, suppose that a Hermitian Laurent polynomial $\ell$ of degree one makes the self-adjoint matrix $\ell({\cal C})$ positive definite. This is equivalent to state that there exists a Cholesky factorization
$$
\ell({\cal C}) = AA^+, \quad A\in\mathscr{T},
$$
which is known to be unique (see Appendix~\ref{app:cholesky}). 

Moreover, in view of the zig-zag shape \eqref{eq:CMVshape} of $\cal C$, the matrix $\ell({\cal C})$ has the five-diagonal structure 
\begin{equation} \label{eq:LC}
\ell({\cal C}) = 
\left(
\begin{smallmatrix}
	* & \kern2pt * & \kern2pt \circledast 
	\\[3pt]
   	* & \kern2pt * & \kern2pt * & \kern1pt \circledast 
   	\\[3pt]
    \circledast & \kern2pt * & \kern2pt * & \kern1pt * & \circledast 
    \\[3pt]
   	& \kern2pt \circledast & \kern2pt * & \kern1pt * & * & \circledast 
	\\[3pt]
   	&& \kern2pt \circledast & \kern1pt * & * & * & \circledast
	\\[3pt]
    &&& \kern1pt \cdots & \cdots & \cdots & \cdots & \cdots
\end{smallmatrix}
\right)
\end{equation}
with non-null entries in the upper and lower diagonals. This implies that $A$ is not only lower triangular with positive main diagonal, but also 3-band with non-null entries in the lower subdiagonal,
\begin{equation} \label{eq:A-3band}
A = 
\left(
\begin{smallmatrix}
	+ 
	\\[3pt]
   	* & \kern1pt + 
   	\\[3pt]
    \circledast & \kern1pt * & \kern1pt + 
    \\[3pt]
   	& \kern1pt \circledast & \kern1pt * & \kern0.5pt + 
	\\[3pt]
   	&& \kern1pt \circledast &  \kern0.5pt * & +
	\\[3pt]
    &&&  \kern0.5pt \cdots & \cdots & \cdots
\end{smallmatrix}
\right).
\end{equation}

The above Cholesky factorization is a particular case of the Darboux factorization analyzed in the previous section. To show this, consider a zig-zag basis $\chi$ related to $\cal C$ and a new one $\omega$ given by $\chi=A\omega$. Then, Proposition \ref{prop:D-fac}.1 shows that $\ell({\cal C})=AB$, where $B$ is the lower Hessenberg type matrix defined by $\ell\omega=B\chi$. Since $\ell({\cal C})=AA^+$ and $A$ is a lower H-matrix we conclude by Proposition \ref{prop:asoc}.1 that $B=A^+$. Hence, if $\cal D$ is the matrix representation  of the multiplication operator in the basis $\omega$, Proposition \ref{prop:D-fac} reads in this case as
\begin{align}
& \ell({\cal C})=AA^+, & & \kern-70pt \ell({\cal D})=A^+A, 
\label{eq:LC-LD}
\\ 
& {\cal C}A=A{\cal D}, & & \kern-70pt {\cal C}^+A=A{\cal D}^+,
\label{eq:CA-AD}
\\ 
& {\cal C}=A{\cal D}A^{-1}, & & \kern-70pt {\cal D}=A^{-1}{\cal C}A.
\label{eq:CD-DC}
\end{align} 
Thus, the triangular matrix $A$ coming from the Cholesky factorization of $\ell({\cal C})$ provides a `reversed' Cholesky factorization of $\ell({\cal D})$. Nevertheless, as we pointed out after Proposition \ref{prop:D-fac}, the matrix $\cal D$ can be obtained directly via the second identity in \eqref{eq:CD-DC}. Note that the reversed factorization is admissible because $A$ is a band matrix. The freedom of the zig-zag basis $\chi\to c\chi$ in a positive factor $c$ changes $\omega\to c\omega$, so it does not alter the matrix $\cal D$ neither the relations \eqref{eq:LC-LD}, \eqref{eq:CA-AD}, \eqref{eq:CD-DC}.

Summarizing, any Hermitian Laurent polynomial $\ell$ of degree one defines a mapping ${\cal C}\overset{\ell}{\mapsto}{\cal D}$ between the set 
$$
\mathscr{C}_\ell := 
\{{\cal C} \text{ CMV} : \ell({\cal C}) \text{ is positive definite}\}
$$
and the class $[{\cal S}]$. This mappping satisfies \eqref{eq:LC-LD}, \eqref{eq:CA-AD}, \eqref{eq:CD-DC} with $A$ as in \eqref{eq:A-3band}, and, up to a positive rescaling, any zig-zag bases $\chi$, $\omega$ associated with $\cal C$, $\cal D$, respectively, are related by
\begin{equation} \label{eq:chi-omega}
\chi=A\omega, \qquad\qquad \ell\omega=A^+\chi.
\end{equation}
The mapping ${\cal C} \overset{\ell}{\mapsto}{\cal D}$ can be defined using the first relation of \eqref{eq:LC-LD} and the second one of \eqref{eq:CD-DC} together with the condition $A\in\mathscr{T}$.
  
The following questions regarding the mapping ${\cal C} \overset{\ell}{\mapsto}{\cal D}$ appear naturally:

\begin{itemize}
\item[(Q1)] When is ${\cal D}$ a CMV matrix? This is equivalent to ask about the unitarity of $\cal D$, or, alternatively, about the orthonormality of $\omega$.
\item[(Q2)] Is ${\cal C} \overset{\ell}{\mapsto}{\cal D}$ an isospectral mapping? If not, what changes may it induce in the spectrum?
\item[(Q3)] What is the relation between two linear functionals $u,v$ associated with the matrices ${\cal C},{\cal D}$? In other words, what is the transformation $u\overset{\ell}{\mapsto}v$ generated by the mapping ${\cal C}\overset{\ell}{\mapsto}{\cal D}$?
\item[(Q4)] Relations \eqref{eq:chi-omega} yield $\chi_n\in\spn\{\omega_{n-2},\omega_{n-1},\omega_n\}$ and $\ell\omega_n\in\spn\{\chi_n,\chi_{n+1},\chi_{n+2}\}$. Does any of these conditions characterize the relation between ${\cal C}$ and ${\cal D}$ given by the mapping ${\cal C}\overset{\ell}{\mapsto}{\cal D}$?
\item[(Q5)] It is known that CMV matrices are parametrized by sequences of complex numbers in the unit disk, the so-called Schur parameters. If $\cal D$ is a CMV matrix, how is the relation between the Schur parameters of $\cal C$ and $\cal D$ encoded in the factor $A$? 
\end{itemize}

The aim of the rest of this section is to answer these questions. 

The answer to (Q1) follows easily from the properties of the mapping ${\cal C}\overset{\ell}{\mapsto}{\cal D}$, using Proposition \ref{prop:asoc} to take care of the associativity.

\begin{prop} \label{prop:D-CMV}
The mapping \kern1pt ${\cal C}\overset{\ell}{\mapsto}{\cal D}$ preserves the set $\mathscr{C}_\ell$, i.e. it transforms every CMV matrix $\cal C$ with $\ell({\cal C})$ positive definite into a CMV matrix $\cal D$ with $\ell({\cal D})$ positive definite.
\end{prop}
 
\begin{proof}
Since $\cal C$ is a CMV matrix, from the second relation in \eqref{eq:CA-AD} and Proposition \ref{prop:asoc}.1 we obtain ${\cal D}^+=A^{-1}{\cal C}^+A$. Combining this identity with the second relation in \eqref{eq:CD-DC} and bearing in mind that the product of lower Hessenberg type matrices is associative, yields
$
{\cal D}{\cal D}^+ = A^{-1}{\cal C}A A^{-1}{\cal C}^+A = I. 
$
Thus, Proposition \ref{prop:uni-H} implies that $\cal D$ is unitary and hence CMV. 

From the second identity in \eqref{eq:LC-LD}, and taking into account that any finitely non-null row vector $X$ is a lower Hessenberg type matrix, using again Proposition \ref{prop:asoc}.1 we find that
$
X\ell({\cal D})X^+ = XA^+AX^+ = \|AX^+\|^2 \ge0,
$
so $X\ell({\cal D})X^+=0$ iff $AX^+=0$, which implies that $X=0$ because $A$ is an H-matrix. Therefore, $\ell({\cal D})$ is positive definite.
\end{proof}  

We will refer to the mappings ${\cal C}\overset{\ell}{\mapsto}{\cal D}$ as the {\sl Darboux transformations for CMV matrices}. Both, the definition of these mappings (which uses the Cholesky factorization of $\ell({\cal C})$) as well as their domains ($\mathscr{C}_\ell$) depend on the choice of the Hermitian Laurent polynomial $\ell$. Nevertheless, a rescaling $\ell \to c \ell$, $c>0$, changes $A \to \sqrt{c}A$, so it yields the same transformation ${\cal C}\overset{\ell}{\mapsto}{\cal D}$.

An important consequence of Proposition~\ref{prop:D-CMV} is that it guarantees that Darboux transformations of CMV matrices can be iterated because $\mathscr{C}_\ell$ constitutes the class of CMV matrices $\cal C$ for which the Cholesky factorization of $\ell(\cal C)$ exists (see Appendix~\ref{app:cholesky}).

As for (Q3) and (Q4), Proposition \ref{prop:D-CMV} states that the transformation $u\overset{\ell}{\mapsto}v$ preserves the positive definiteness of the functionals, so when answering these questions we can suppose without loss of generality that $u$, $v$ are positive definite and $\chi$, $\omega$ the corresponding orthonormal zig-zag bases. The answer to (Q3) and (Q4) is given by the following theorem, which is the main result of this section.

\begin{thm} \label{thm:Darboux}
Let $u$, $v$ be positive definite functionals in $\Lambda$ with orthonormal zig-zag bases $\chi$, $\omega$ and CMV matrices $\cal C$, $\cal D$ respectively. Then, the following statements are equivalent:
\begin{itemize}
\item[(i)] ${\cal C}\overset{\ell}{\mapsto}{\cal D}$ for some Hermitian Laurent polynomial $\ell$ of degree one, i.e. ${\cal C}\in\mathscr{C}_\ell$ and ${\cal D}=A^{-1}{\cal C}A$, where $A\in\mathscr{T}$ comes from the Cholesky factorization $\ell({\cal C})=AA^+$.
\item[(ii)] $v=u\ell$ for some Hermitian Laurent polynomial $\ell$ of degree one.
\item[(iii)] $\chi_n\in\spn\{\omega_{n-2},\omega_{n-1},\omega_n\}$ for $n\ge2$ and $\chi_n\notin\spn\{\omega_n\}$ for some $n$.
\item[(iv)] There exists a Hermitian Laurent polynomial $\ell\notin\spn\{\chi_1,\chi_2\}$ of degree one such that $\ell\omega_n\in\spn\{\chi_n,\chi_{n+1},\chi_{n+2}\}$ for $n\ge0$.
\end{itemize}
The Laurent polynomials $\ell$ mentioned in $(i)$ and $(ii)$ coincide up to a constant positive factor, and they also coincide with that one in $(iv)$ up to a constant real factor. Moreover, if $\ell$ coincides exactly in $(i)$ and $(ii)$ then $\chi=A\omega$ and $\ell\omega=A^+\chi$ without any rescaling.
\end{thm}

\begin{proof}
The implications $(i)\Rightarrow(iii)$ and $(i)\Rightarrow(iv)$ are already proved because we know that, up to a positive rescaling, $(i)$ gives the relation \eqref{eq:chi-omega} between $\chi$ and $\omega$, and $A$ has the structure \eqref{eq:A-3band}.

\smallskip

\noindent \fbox{$(i)\Leftrightarrow(ii)$} 

Assuming $(i)$, we can suppose without loss of generality that $\chi=A\omega$ and $\ell\omega=A^+\chi$ by rescaling $\ell$ with a positive factor if necessary. Then, bearing in mind that any column (row) vector is a lower (upper) Hessenberg type matrix and using Proposition \ref{prop:asoc}, we find that $u\ell[\omega\omega^+] = u[\omega(\ell\omega)^+] = A^{-1}u[\chi\chi^+]A = I$. Due to the uniqueness of the orthonormality functional of $\omega$, we conclude that $v=u\ell$.

Suppose now that $v=u\ell$. Then, \eqref{eq:C} and Proposition \ref{prop:asoc} imply that $v[\chi\chi^+]=u[\ell\chi\chi^+]=\ell({\cal C})u[\chi\chi^+]=\ell({\cal C})$. This means that, assuming $u$ positive definite, $v$ is positive definite too iff ${\cal C}\in\mathscr{C}_\ell$. Therefore, the hypothesis of the theorem guarantee the existence of the Cholesky factorization $\ell({\cal C})=AA^+$. Also, $\chi=A\omega$ because, using again Proposition \ref{prop:asoc}, we get $v[(A^{-1}\chi)(A^{-1}\chi)^+]=A^{-1}v[\chi\chi^+](A^+)^{-1}=A^{-1}\ell({\cal C})(A^+)^{-1}=I$. Then, the previous results prove that $\ell\omega=A^+\chi$ and ${\cal D}=A^{-1}{\cal C}A$.  

\smallskip

\noindent \fbox{$(iii)\Rightarrow(ii)$}

Let us look for a solution $\ell(z)=\alpha z+\beta+\gamma z^{-1}$ of $u\ell=v$, once $(iii)$ is assumed. 

First, for any Laurent polynomial of degree not bigger than one, $(u\ell-v)[\chi_n]=0$ when $n\ge3$ because, from \eqref{eq:C} and the five-diagonal structure \eqref{eq:LC} of $\ell({\cal C})$, 
$$
\begin{aligned}
& \ell\chi_n = (\ell({\cal C})\chi)_n \in \spn\{\chi_{n-2},\dots,\chi_{n+2}\}
\;\Rightarrow\;
u\ell[\chi_n]=u[\ell\chi_n]=0, \quad n\ge3,
\\
& \chi_n\in\spn\{\omega_{n-2},\omega_{n-1},\omega_n\}
\;\Rightarrow\;
v[\chi_n]=0, \quad n\ge3.
\end{aligned}
$$
Hence, $u\ell-v$ vanishes on the orthogonal complement of $\LL_2$ with respect to $u$.

Second, there is a choice of $\ell$ such that $u\ell-v$ vanishes on $\LL_2$. Using the canonical basis $\eta^{(2)}=(1,z,z^{-1})^t$ and expressing $\ell=\eta^{(2)+}(\beta,\gamma,\alpha)^t$, the condition $(u\ell-v)[\eta^{(2)}]=0$ becomes the linear system $u[\eta^{(2)}\eta^{(2)+}](\beta,\gamma,\alpha)^t=v[\eta^{(2)}]$. Since $u$ is positive definite, $\det(u[\eta^{(2)}\eta^{(2)+}])>0$ and this system has a unique solution $(\alpha,\beta,\gamma)$. 

We conclude that $v=u\ell$ for a unique Laurent polynomial $\ell$ of degree not bigger than one. Since $u$ and $v$ are Hermitian, taking adjoints we find that $v=u\ell_*$. Therefore, the uniqueness of $\ell$ implies that $\ell=\ell_*$, so that $\gamma=\overline\alpha$. If $\alpha=0$, then $v$ is proportional to $u$, thus $\chi$ is proportional to $\omega$, which contradicts the hypothesis. Therefore, $\ell$ is a Hermitian polynomial of degree one.

\smallskip

\noindent \fbox{$(iv)\Rightarrow(ii)$} 

The condition $\ell\omega_n\in\spn\{\chi_n,\chi_{n+1},\chi_{n+2}\}$ implies that $u\ell[\omega_n]=u[\ell\omega_n]=0$ for $n\ge1$. 
Hence, $(uc\ell-v)[\omega_n]=0$ for $n\ge1$ and any $c\in\C$. Since $\ell\in\LL_2\setminus\spn\{\chi_1,\chi_2\}$, necessarily $u[\ell]\ne0$. Therefore, we can choose $c\in\R$ such that $uc\ell=v$ because $(uc\ell-v)[\omega_0]=0$ iff $c=v[1]/u[\ell]$, which is real because $u$, $v$ and $\ell$ are Hermitian.
\end{proof}

This theorem states that the Darboux transformations of CMV matrices correspond to the Laurent polynomial modifications of degree one $d\mu\overset{\ell}{\mapsto}\ell d\mu$ for measures on the unit circle. In the particular case $\ell(z)=(z-\zeta)(z^{-1}-\overline\zeta)$, $|\zeta|<1$, these are known as the Christoffel transformations on the unit circle \cite{DarHerMar,MarHer}, which preserve the whole set of positive definite functionals (i.e. $\mathscr{C}_\ell$ is the whole set of CMV matrices in this case) because such an $\ell$ is non-negative on the unit circle. Nevertheless, a general Laurent polynomial modification $d\mu\overset{\ell}{\mapsto}\ell d\mu$ preserves positive definiteness only when $\ell$ is non-negative on the support of $\mu$. The proof of $(i)\Leftrightarrow(ii)$ in Theorem \ref{thm:Darboux} uncovers a matrix characterization of such measures, which we enunciate separately. 

\begin{cor} \label{cor:pd}
Given a Hermitian Laurent polynomial $\ell$ and a positive definite functional $u$ with CMV matrix $\cal C$, 
$$
u\ell \text{ is positive definite} 
\;\Leftrightarrow\; 
\ell({\cal C}) \text{ is positive definite}. 
$$ 
\end{cor}

Although the proof of this result given in Theorem \ref{thm:Darboux} was enunciated for the case of a degree one $\ell$, the proof for an arbitrary degree remains unchanged. This characterization is important when an unknown measure is determined only by its Schur parameters, which provide directly the corresponding CMV matrix (see \eqref{eq:CMV-schur}). 

The equivalence with Laurent polynomial modifications of measures gives information about the spectral behaviour of Darboux transformations. Let $\cal C$ be a CMV matrix with associated measure $\mu$ and consider the Hilbert space $L^2_\mu$ of $\mu$-square-integrable functions. Then, $\cal C$ can be understood as the representation of the unitary multiplication operator 
\begin{equation} \label{eq:Umu}
 \mathsf{U}_\mu \colon 
 \mathop{L^2_\mu \xrightarrow{\kern15pt} L^2_\mu}
 \limits_{\text{\footnotesize $f(z) \mapsto zf(z)$}}
\end{equation}
with respect to the orthonormal basis of $L^2_\mu$ given by a zig-zag basis $\chi$ of $\cal C$. Therefore, the spectrum of $\cal C$ coincides with that of $\mathsf{U}_\mu$, which is the support of $\mu$. The mass points are the corresponding eigenvalues, which are simple because their eigenvectors are spanned by the characteristic functions of the mass points \cite{CMV2005,Si2005}. On the other hand, every transformation $\mu\xmapsto{\ell}\nu$ preserves the support of $\mu$ up to the mass points located at the zeros of $\ell$, which do not appear in $\nu$. Therefore, we have the following spectral consequence of the previous theorem, which answers (Q2).

\begin{cor} \label{cor:spec}
Given a Hermitian Laurent polynomial $\ell$ of degree one, the mapping ${\cal C}\xmapsto{\ell}{\cal D}$ preserves the spectrum for every CMV matrix $\cal C\in\mathscr{C}_\ell$, with the exception of at most two points: The spectrum of $\cal D$ is obtained from that of $\cal C$ by excluding the eigenvalues which are zeros of $\ell$.    
\end{cor}

In particular, ${\cal C}\xmapsto{\ell}{\cal D}$ is a isospectral transformation whenever $\ell$ has its zeros outside of the unit circle. Otherwise it is almost isospectral, in the sense that it preserves the spectrum up to finitely many points. More precisely, if $\ell$ has its zeros on the unit circle the only spectral changes that the transformation may produce is the elimination of one or two eigenvalues depending whether $\ell$ has one or two different zeros.    

Remember that the factor $B$ of general Darboux factorizations was only lower Hessenberg type and not necessarily invertible. However, in the case of the Darboux transformations for CMV, the factor $B=A^+$ is not only invertible but also 3-band and upper triangular, so it is an upper H-matrix. Therefore, in this case we can complete the relations \eqref{eq:LC-LD}, \eqref{eq:CA-AD}, \eqref{eq:CD-DC} with the additional ones
$$
{\cal C} = B^{-1}{\cal D}B, \qquad {\cal D} = B{\cal C}B^{-1}.
$$
Moreover, the relation between the factors $A$ and $B=A^+$ gives information about the left kernel of $A$ (the right one is trivial): $\ker_L(A)=\ker_R(B)^+$, where $\ker_R(B)$ is given by Proposition \ref{prop:kerB}.  

Question (Q5) refers to the explicit parametrization of CMV matrices given by \cite{CMV2003,Si2005,W1993} (our notation is related to that of \cite{Si2005} by $a_n=-\overline\alpha_{n-1}$, while we take the transposed of the CMV matrix primarily used in \cite{Si2005}),
\begin{equation} \label{eq:CMV-schur}
{\cal C} = 
\begin{pmatrix} 
	-\overline{a}_0a_1 & \overline{a}_0\rho_1
	\\
	-\rho_1a_2 & -\overline{a}_1a_2 & 
	-\rho_2a_3 & \rho_2\rho_3 
	\\
    \rho_1\rho_2 & \overline{a}_1\rho_2 & 
    -\overline{a}_2a_3 & \overline{a}_2\rho_3
    \\
    && -\rho_3a_4 & -\overline{a}_3a_4 &
    -\rho_4a_5 & \rho_4\rho_5
    \\
    && \rho_3\rho_4 & \overline{a}_3\rho_4 &
    -\overline{a}_4a_5 & \overline{a}_4\rho_5
    \\
    &&&& \cdots & \cdots & \cdots
\end{pmatrix},
\qquad
a_0=1.
\end{equation}
where $a_{n}\in\C$ are the so-called Schur parameters (or Verblunsky coefficients), which satisfy $|a_n|<1$ for $n\ge1$, and $\rho_{n}=\sqrt{1-|a_{n}|^{2}}$. Schur parameters establish a one-to-one correspondence between sequences in the open unit disk and CMV matrices, thus they parametrize also the measures on the unit circle up to normalization. The Schur parameters also determine the orthonormal polynomials $\varphi_n$ with respect to the corresponding functional $u\equiv d\mu$ via the forward and backward recurrence relations
\begin{equation} \label{eq:RR}
\begin{cases}
 \rho_n \varphi_n(z) = z \varphi_{n-1}(z) + a_n \varphi_{n-1}^*(z),
 \\[3pt]
 \rho_n z\varphi_{n-1}(z) = \varphi_n(z) - a_n \varphi_ n^*(z),
\end{cases}
\qquad \varphi_n^*(z)=z^n\varphi_{n*}(z),
\end{equation}
where $\varphi_n^*$ is known as the reversed polynomial of $\varphi_n$. As a consequence,
\begin{equation} \label{eq:e}
\varphi_n(z)=\kappa_n(z^n+\cdots+a_n),
\qquad \kappa_{n-1}=\rho_n\kappa_n,
\qquad \kappa_0=\frac{1}{\sqrt{u[1]}}.
\end{equation} 
The orthonormal polynomials (ONP) are connected to the orthonormal Laurent polynomials (ONLP) by the relations
\begin{equation} \label{eq:OLP-OP}
\chi_{2n}(z)= z^{-n}\varphi_{2n}^*(z), 
\qquad
\chi_{2n+1}(z)=z^{-n} \varphi_{2n+1}(z),
\end{equation}
which, combined with \eqref{eq:RR}, lead to 
\begin{equation} \label{eq:OLP*}
\rho_{2n}\chi_{2n} = \chi_{2n-1*} + \overline{a}_{2n}\chi_{2n-1},
\qquad
\rho_{2n}\chi_{2n-1} = \chi_{2n*} - a_{2n}\chi_{2n}.
\end{equation}

Before answering (Q5) let us fix a notation concerning the mapping ${\cal C}\overset{\ell}{\mapsto}{\cal D}$, which will be used in the rest of the paper. 
\begin{equation} \label{eq:table}
\renewcommand{\arraystretch}{1.5}
\begin{tabular}{|c|c|c|c|c|c|}
\hline
CMV & Functional & ONLP & ONP & Schur & 
\\ \hline
$\cal C$ & $u\equiv d\mu$ & $\chi_n$ & $\varphi_n$ & $a_n$ & 
$\begin{aligned}
&
\\[-15pt]
& \rho_n=\sqrt{1-|a_n|^2}
\\[-2pt]
& \kappa_n^{-1}=\rho_n\cdots\rho_1\sqrt{u[1]}
\\[-15pt]
&
\end{aligned}$
\\ \hline
$\cal D$ & $v\equiv d\nu$ & $\omega_n$ & $\psi_n$ & $b_n$ & 
$\begin{aligned}
&
\\[-15pt]
& \sigma_n=\sqrt{1-|b_n|^2} 
\\[-2pt]
& \lambda_n^{-1}=\sigma_n\cdots\sigma_1\sqrt{v[1]}
\\[-15pt]
&
\end{aligned}$
\\ \hline
\end{tabular}
\end{equation}
The coefficients of the factor $A$ will be denoted as  
\begin{equation} \label{eq:Aexplicit}
A = 
\begin{pmatrix} 
	r_0
	\\ 
	s_0 & r_1
	\\ 
	t_0 & \overline{s}_1 & r_2
	\\
	& \overline{t}_1 & s_2 & r_3
	\\
	&& t_2 & \overline{s}_3 & r_4
	\\
	&&& \overline{t}_3 & s_4 & r_5
	\\
	&&&& \ddots & \ddots & \ddots
\end{pmatrix},
\qquad 
r_n>0, \quad s_n\in\C, \quad t_n\in\C\setminus\{0\}.
\end{equation}
From the explicit parametrization \eqref{eq:CMV-schur} of $\cal C$ we see that, not only the upper and lower diagonals of the five-diagonal matrix $\ell({\cal C})$ are non-null, but its $(1,0)$ and $(0,1)$ matrix coefficients cannot vanish either, so that the structure \eqref{eq:LC} becomes 
\begin{equation} \label{eq:LC-bis}
\ell({\cal C}) = 
\left(
\begin{smallmatrix}
	* & \kern2pt \circledast & \kern2pt \circledast 
	\\[3pt]
   	\circledast & \kern2pt * & \kern2pt * & \kern1pt \circledast 
   	\\[3pt]
    \circledast & \kern2pt * & \kern2pt * & \kern1pt * & \circledast 
    \\[3pt]
   	& \kern2pt \circledast & \kern2pt * & \kern1pt * & * & \circledast 
	\\[3pt]
   	&& \kern2pt \circledast & \kern1pt * & * & * & \circledast
	\\[3pt]
    &&& \kern1pt \cdots & \cdots & \cdots & \cdots & \cdots
\end{smallmatrix}
\right).
\end{equation}
This implies that not only $t_n\ne0$ for all $n\ge0$, but also $s_0\ne0$.

Using \eqref{eq:OLP-OP}, the relations $\chi=A\omega$ and $\ell\omega=A^+\chi$ read as (taking the reverse of the equations for even $n$)
$$
\begin{aligned}
 & \varphi_{n}(z) = 
 \overline{t}_{n-2} z\psi_{n-2}(z) + s_{n-1} \psi_{n-1}^*(z) + 
 r_{n} \psi_{n}(z),
 \\
 & z\ell(z)\psi_{n}(z) = 
 r_{n} z\varphi_{n}(z) + s_{n} \varphi_{n+1}^*(z) + 
 t_{n} \varphi_{n+2}(z),
\end{aligned}
\kern15pt 
n\ge0,
\kern15pt 
t_k=s_k=0 \,\text{ if }\, k<0.
$$
If $\ell(z)=\alpha z+\beta+\overline\alpha z^{-1}$, identifying in the above equalities the terms of highest and lowest degree and using \eqref{eq:e} leads to
\begin{equation} \label{eq:a-b}
\begin{aligned}
 & r_n = \frac{\kappa_n}{\lambda_n}, 
 \qquad 
 t_n = \alpha\frac{\lambda_n}{\kappa_{n+2}} =
 \alpha\rho_{n+2}\frac{\rho_{n+1}}{r_n},
 \\[3pt]
 & b_{n} 
 = a_{n} - s_{n-1}\frac{\lambda_{n-1}}{\kappa_n}
 = a_{n} - s_{n-1}\frac{\rho_{n}}{r_{n-1}},
 \\[3pt]
 & \overline\alpha b_{n} 
 = \alpha a_{n+2} + s_{n} \frac{\kappa_{n+1}}{\lambda_{n}}
 = \alpha a_{n+2} + s_{n} \frac{r_{n}}{\rho_{n+1}}, 
\end{aligned}
 \qquad 
 n\ge0,
 \qquad 
 a_0=b_0=1.
\end{equation}
The equation of the second line is the translation to the Schur parameters of the Darboux transformations for CMV matrices. The last equation is an additional constraint that appears when fixing the Laurent polynomial $\ell$. It is worth noting that the identities of the first line lead to the simple relations
\begin{equation} \label{eq:r-t}
 \frac{r_{n-1}}{r_n} = \frac{\rho_n}{\sigma_n},
 \qquad\qquad 
 \frac{t_{n-1}}{t_n} = \frac{\sigma_n}{\rho_{n+2}}.
\end{equation}
Besides, combining the two last equations in \eqref{eq:a-b} we get
\begin{equation} \label{eq:D=dS}
 b_n - a_n = 
 \frac{\lambda_{n-1}^2}{\kappa_n^2} (\alpha a_{n+1}-\overline\alpha b_{n-1}),
\end{equation}
a nonlinear recurrence relation connecting directly the Schur parameters $a_n$ and $b_n$. Actually, given $a_n$, \eqref{eq:D=dS} generates inductively $b_n$ starting from $b_0=1$.

The factorization $\ell({\cal C})=AA^+$, explicitly written in terms of \eqref{eq:Aexplicit}, reads as
\begin{equation} \label{eq:direct} 
\left\{
\begin{aligned}
& |t_{n-2}|^2+|s_{n-1}|^2+r_n^2 = 
\beta-2\re(\alpha\overline{a}_na_{n+1}),
\\
& s_{n-1}\overline{t}_{n-1}+r_ns_n = 
\rho_{n+1}(\overline\alpha a_n-\alpha a_{n+2}),
\\
& r_nt_n = \alpha\rho_{n+1}\rho_{n+2},
\end{aligned}
\right.
\qquad 
n\ge0,
\end{equation}
which, starting from the initial conditions $t_{-2}=t_{-1}=s_{-1}=0$, determines inductively the coefficients $r_n$, $s_n$, $t_n$ for $n\ge0$. This shows that, if a solution of \eqref{eq:direct} exists, it is unique. The existence of such a solution requires the positivity of $\beta-2\re(\alpha\overline{a}_na_{n+1})-|s_{n-1}|^2-|t_{n-2}|^2$ at every induction step, and it is equivalent to the positive definiteness of $\ell({\cal C})$ (see Appendix~\ref{app:cholesky}).
 
\section{Inverse Darboux for CMV and Geronimus transformations}
\label{sec:Ger}

In this section we will study the `inverse' of the Darboux transformations for CMV matrices introduced in the previous section. More precisely, given a Hermitian Laurent polynomial $\ell$ of degree one and a CMV matrix $\cal D$, we will search for the CMV matrices $\cal C$ which are transformed into $\cal D$ by the mapping ${\cal C}\overset{\ell}{\mapsto}{\cal D}$. In view of Theorem  \ref{thm:Darboux}, this amounts to characterizing the CMV matrices $\cal C$ of the positive measures $\mu$ satisfying $\ell d\mu=d\nu$, where $\nu$ is a measure associated with $\cal D$. Corollary \ref{cor:spec} implies that the inverse Darboux transformations are almost isospectral, since they preserve the spectrum up to the addition of at most two eigenvalues.

In contrast to Darboux transformations, the inverse Darboux transformations can have no solution or multiple solutions. For instance, if $\ell(z)=(z-\zeta)(z^{-1}-\overline\zeta)$, $|\zeta|=1$, no measure $\mu$ on the unit circle solves the equation $\ell d\mu=d\vartheta$, where $d\vartheta(e^{i\theta})=d\theta/2\pi$ is the Lebesgue measure. On the other hand, under the same choice of $\ell$, if $\ell d\mu=d\nu$, then $\ell(d\mu+m\delta_\zeta)=d\nu$ for any positive mass $m$, where $\delta_\zeta$ is the Dirac delta at $\zeta$.

The parametric structure of the set of solutions for the inverse Darboux transformations depends on the localization of the zeros $\zeta_1$, $\zeta_2$ of $\ell$ with respect to the unit circle. Due to the hermiticity of $\ell$, the transformation $\zeta\to\overline\zeta^{-1}$ leaves invariant the set of zeros of $\ell$. Hence, given a measure $\nu$ on the unit circle and a Hermitian Laurent polynomial $\ell$, we have the following possibilities for the solutions $\mu$ of $\ell d\mu=d\nu$:

\begin{itemize}
\item Zeros outside of the unit circle: $\zeta_1=\zeta=\overline\zeta_2^{-1}$, $|\zeta|<1$.

Then, $\ell(z)=c|z-\zeta|^2$, $c\in\R\setminus\{0\}$, for $|z|=1$, so $\ell d\mu=d\nu$ is solved by a unique positive measure $d\mu=d\nu/\ell$ or by no one depending whether $c>0$ or $c<0$. 

\item Zeros on the unit circle: $\zeta_1=e^{i\theta_1}$, $\zeta_2=e^{i\theta_2}$.

Then, $\ell(z)=c\sin(\frac{\theta-\theta_1}{2})\sin(\frac{\theta-\theta_2}{2})$, $c\in\R\setminus\{0\}$, for $z=e^{i\theta}$. Thus, depending on the integrability and positivity of $d\nu/\ell$, either no positive measure solves $\ell d\mu=d\nu$, or there are infinitely many solutions 
$$
\begin{aligned}
& d\mu=\frac{d\nu}{\ell}+m\delta_{\zeta}, 
& &  m\ge0,
& & \text{ if } \zeta_1=\zeta=\zeta_2,
\\
& d\mu=\frac{d\nu}{\ell}+m_1\delta_{\zeta_1}+m_2\delta_{\zeta_2}, 
& & m_1,m_2\ge0, 
& & \text{ if } \zeta_1\ne\zeta_2,
\end{aligned}
$$
parametrized by one or two real parameters, namely, the masses of $\mu$ at $\zeta_1$, $\zeta_2$.

\end{itemize}

The first of these cases is closely related to the Geronimus transformations on the unit circle \cite{GarHerMar,GarMar2}, defined by 
$$
d\mu=\frac{d\nu}{\ell}+m\delta_\zeta+\overline{m}\delta_{\overline{\zeta}^{-1}},
\qquad
\ell(z)=(z-\zeta)(z^{-1}-\overline\zeta), \qquad |\zeta|<1, \qquad m\in\C.
$$ 
Starting from a positive measure $\nu$ on the unit circle, the Geronimus transformations yield in general a complex measure $\mu$ which generates a Hermitian linear functional $u[f]=\int fd\mu$ in $\Lambda$. Since positive definite functionals are represented by measures supported on the unit circle, this functional is positive definite iff $m=0$, which is the case covered by the inverse Darboux transformations. Geronimus transformations can provide quasi-definite functionals $u$ for $m\ne0$, but the study of these cases in the setting of Darboux transformations requires their generalization to quasi-CMV matrices, i.e. to the matrices playing the role of CMV in the case of quasi-definite functionals on the unit circle (see Section~\ref{ssec:quasi}).

Summarizing, the inverse Darboux transformation of a CMV matrix $\cal D$ leads to a set 
$$
\mathscr{C}_\ell({\cal D}) := 
\{{\cal C} \text{ CMV} : {\cal C}\overset{\ell}{\mapsto}{\cal D}\}
\subset \mathscr{C}_\ell
$$
which can be eventually empty, and whose elements can be naturally labelled by at most two real parameters. These elements $\cal C$ must satisfy the identities \eqref{eq:LC-LD}, \eqref{eq:CA-AD}, \eqref{eq:CD-DC}, and their orthonormal zig-zag bases $\chi$ must be related to an orthonormal zig-zag basis $\omega$ of $\cal D$ by \eqref{eq:chi-omega}. 

According to the results of the previous section, a matrix procedure to obtain the inverse Darboux transforms $\mathscr{C}_\ell({\cal D})$ of $\cal D$ starts by performing a reversed Cholesky factorization $\ell({\cal D})=A^+A$, $A\in\mathscr{T}$, which implies that $A$ must have the 3-band structure \eqref{eq:A-3band}. It is possible to prove that, similarly to the standard Cholesky factorization, the reversed one is possible iff ${\cal D}\in\mathscr{C}_\ell$ while, in contrast to the case of finite matrices, reversed Cholesky factorizations of infinite matrices are not unique in general (see Appendix~\ref{app:cholesky}). This is in agreement with the fact that $\mathscr{C}_\ell({\cal D})$ can have more than one element. For each of these reversed factorizations we can obtain a matrix ${\cal C}=A{\cal D}A^{-1}\in[{\cal S}]$, and any solution of the inverse Darboux transformation should be obtained in this way. Since we expect in general a parametric solution for this matrix procedure, we will refer to it as the {\sl Darboux transformation with parameters} corresponding to the Hermitian polynomial $\ell$.

Let us analyze in detail the Darboux transformation with parameters. Suppose ${\cal D}\in\mathscr{C}_\ell$ with a zig-zag basis $\omega$, and consider a reversed Cholesky factorization $\ell({\cal D})=A^+A$, $A\in\mathscr{T}$. Defining the zig-zag basis $\chi=A\omega$ and the matrix $B$ by $\ell\omega=B\chi$ leads to $B\chi=\ell\omega=\ell({\cal D})\omega=A^+A\omega=A^+\chi$, so that \eqref{eq:chi-omega} holds and $B=A^+$ due to the linear independence of $\chi$. Then, we know by Proposition \ref{prop:D-fac} that \eqref{eq:LC-LD}, \eqref{eq:CA-AD} and \eqref{eq:CD-DC} are valid for the matrix ${\cal C}\in[{\cal S}]$ related to $\chi$. Hence, $\cal C \in \mathscr{C}_\ell({\cal D})$ as long as it is a CMV matrix. However, up to now there is no reason to assume that $\cal C$ is CMV. Bearing in mind Corollary \ref{cor:CMVdef}, to prove this we only need to show that $\cal C$ is unitary, which in view of Proposition \ref{prop:uni-H} is equivalent to the identity ${\cal C}{\cal C}^+=I$.

Trying for $\cal C$ a proof similar to that of the unitarity of $\cal D$ for the direct Darboux transformations would start from the identities ${\cal C}^+A=A{\cal D}^+$ and ${\cal C}=A{\cal D}A^{-1}$ in \eqref{eq:CA-AD} and \eqref{eq:CD-DC}. However, the temptation to use the first of these identities to write ${\cal C}^+=A{\cal D}^+A^{-1}$ fails because, using Proposition \ref{prop:asoc}, we can only deduce that
\begin{equation} \label{eq:no-unit-1}
({\cal C}^+A)A^{-1}=(A{\cal D}^+)A^{-1}=A{\cal D}A^{-1},
\end{equation}
since $A$ and ${\cal D}^+$ are band matrices. The associativity of the first term is not guaranteed by Proposition \ref{prop:asoc} because $A$ is band but $A^{-1}$ is lower Hessenberg type and ${\cal C}^+$ is upper Hessenberg type (we do not know yet if $\cal C$ is a CMV!). 

An attempt to overcome this problem resorts to the direct use of the two identities in \eqref{eq:CA-AD}, together with Proposition \ref{prop:asoc}, as follows
\begin{equation} \label{eq:no-unit-2}
({\cal C}{\cal C}^+)A = {\cal C}({\cal C}^+A) = {\cal C}(A{\cal D}^+) = 
({\cal C}A){\cal D}^+ = (A{\cal D}){\cal D}^+ = A({\cal D}{\cal D}^+) = A.
\end{equation}
However, the same associativity problem found in \eqref{eq:no-unit-1} reappears now when trying to prove that ${\cal C}{\cal C}^+=I$ using \eqref{eq:no-unit-2}. The only conclusion from \eqref{eq:no-unit-2} is that the rows of ${\cal C}{\cal C}^+-I$ belong to $\ker_L(A)=\ker_R(A^+)^+$, which is known to be non trivial by Proposition \ref{prop:kerB}. Since ${\cal C}{\cal C}^+-I$ is Hermitian, this is equivalent to state that their columns lie on $\ker_R(A^+)$, or using the notation of Proposition \ref{prop:kerB}, 
\begin{equation} \label{eq:no-unit}
{\cal C}{\cal C}^+ = I + K,
\qquad 
K \in \mathscr{K}_{A^+} = \{K=K^+ : A^+K=0\}. 
\end{equation}

There is reason for the failure of the previous attempts to prove the unitarity of ${\cal C}$: not all the reversed Cholesky factorizations $\ell({\cal D})=A^+A$, $A\in\mathscr{T}$, lead to a unitary ${\cal C}=A{\cal D}A^{-1}$. In other words, the Darboux transformation with parameters not only yields the set $\mathscr{C}_\ell({\cal D})$ of inverse Darboux transforms of $\cal D$, but it also gives spurious solutions ${\cal C}\in[{\cal S}]$ which are not CMV. 

Take for instance $\nu=\vartheta$ as the Lebesgue measure and $\ell(z)=(z-\zeta)(z^{-1}-\overline\zeta)$, $|\zeta|=1$. Then ${\cal D}={\cal S}$ and there is no solution $\mu$ of $\ell d\mu=d\nu$, i.e. $\mathscr{C}_\ell({\cal S})=\emptyset$. However, since $\ell d\vartheta$ is a positive measure on the unit circle, we know by Corollary \ref{cor:pd} that ${\cal S}\in\mathscr{C}_\ell$ and, hence, there exist reversed Cholesky factorizations $\ell({\cal S})=A^+A$, $A\in\mathscr{T}$, leading to spurious solutions ${\cal C}=A{\cal D}A^{-1}$. Actually, the functionals of these spurious solutions can be even non-Hermitian. This will be illustrated with an explicit example later on.

The appearance of spurious solutions can be made apparent by rewriting explicitly the reversed factorization $\ell({\cal D})=A^+A$ in terms of \eqref{eq:Aexplicit} and $\ell(z)=\alpha z+\beta+\overline\alpha z^{-1}$,
\begin{equation} \label{eq:inverse}
\left\{
\begin{aligned}
& r_n^2+|s_n|^2+|t_n|^2 = 
\beta-2\re(\alpha\overline{b}_nb_{n+1}),
\\
& s_nr_{n+1}+t_ns_{n+1} = 
\sigma_{n+1}(\overline\alpha b_n-\alpha b_{n+2}),
\\
& t_nr_{n+2} = \alpha\sigma_{n+1}\sigma_{n+2},
\end{aligned}
\right.
\qquad n\ge0.
\end{equation}
In contrast to \eqref{eq:direct}, these relations do not determine the coefficients $r_n$, $s_n$, $t_n$. Actually, expressing \eqref{eq:inverse} as
\begin{equation} \label{eq:inverse2}
\begin{gathered}
r_{n+2} = |\alpha|
\frac{\sigma_{n+1}\sigma_{n+2}}{\tau_n},
\quad
s_{n+1} = \frac{\overline\alpha}{|\alpha|} 
\frac{\sigma_{n+1}(\overline\alpha b_n-\alpha b_{n+2})-s_nr_{n+1}}{\tau_n},
\quad 
t_n = \frac{\alpha}{|\alpha|}\tau_n,
\\[2pt] 
\tau_n = \sqrt{\beta-2\re(\alpha\overline{b}_nb_{n+1})-r_n^2-|s_n|^2},
\end{gathered}
\end{equation}
shows that $r_n$, $s_n$, $t_n$ are determined inductively for $n\ge0$ by the three parameters $r_0$, $s_0$, $r_1$, which enclose the freedom in the reversed factorization of $\ell({\cal D})$. Nevertheless, this does not mean that any choice of $r_0$, $s_0$, $r_1$ should give such a reversed factorization, a fact that depends on the positivity of $\beta-2\re(\alpha\overline{b}_nb_{n+1})-r_n^2-|s_n|^2$ at every induction step. The existence of a solution of \eqref{eq:inverse} for at least one choice of the ``free" parameters $r_0$, $s_0$, $r_1$ is equivalent to state that $\ell({\cal D})$ is positive-definite (see Appendix~\ref{app:cholesky}). 

Therefore, the general solution of a Darboux transformation with parameters depends on two real ($r_0$, $r_1$) and one complex ($s_0$) initial conditions, i.e. four real parameters. This is in striking contrast with the corresponding inverse Darboux transforms, which depend on at most two real parameters. Of course, this counting is indicative of the amount of spurious solutions only up to initial conditions $r_0$, $s_0$, $r_1$ giving no reversed factorization $\ell({\cal D})=A^+A$.

Before presenting an explicit example of a spurious solution, we will give a characterization which allows us to distinguish the CMV solutions from the spurious ones with a few calculations. Such a characterization is based on Proposition \ref{prop:kerB} and some results concerning the hermiticity of the solutions $u$ of $u\ell=v$. 

\begin{lem} \label{lem:herm}
Let $\ell$ be a Hermitian Laurent polynomial of degree one and $v$ a Hermitian linear functional in $\Lambda$ with $v[1]\ne0$. Then, a linear functional solution $u$ of $u\ell=v$ is Hermitian iff $u[l_{i*}]=\overline{u[l_i]}$ for a basis $l_0,l_1$ of $\LL_1$. 
\end{lem}

\begin{proof}
The hermiticity of $u$ on a basis of $\LL_1$, necessary for its hermiticity on $\Lambda$, is equivalent to its hermiticity on $\LL_1$. On the other hand, the relation $u\ell=v$, together with the hermiticity of $v$ and $\ell$, leads to
$$
u[(\ell f)_*] = u[\ell f_*] = v[f_*] = \overline{v[f]} = \overline{u[\ell f]}, 
\qquad f\in\Lambda,
$$
which means that $u$ is Hermitian in the subspace $\ell\Lambda$ of $\Lambda$. This, combined with the hermiticity of $u$ in $\LL_1$, implies its hermiticity on $\Lambda=\LL_1\oplus \ell\Lambda$.
\end{proof}

The previous lemma has the following consequences concerning the hermiticity of the functionals associated with the solutions of the Darboux transformations with parameters.

\begin{prop} \label{prop:herm}
Let $\ell$ be a Hermitian Laurent polynomial of degree one, ${\cal D}\in\mathscr{C}_\ell$ with Schur parameters $b_n$, $\sigma_n=\sqrt{1-|b_n|^2}$ and $v$ a related functional. Consider an arbitrary solution ${\cal C}\in[{\cal S}]$ of the corresponding Darboux transformation with parameters, i.e. ${\cal C}=A{\cal D}A^{-1}$ with $\ell({\cal D})=A^+A$, $A\in\mathscr{T}$. If $u$ is the functional of a zig-zag basis $\chi$ related to $\cal C$, then $u\ell=v$ up to a positive rescaling of $\ell$. Besides, $u$ is Hermitian iff any of the following equivalent conditions is satisfied:
\begin{itemize}
\item[(i)] $u[z^{-1}]=\overline{u[z]}$.
\item[(ii)] $u[\chi_{1*}]=0$.
\item[(iii)] $\chi_{1*}\in\spn\{\chi_1,\chi_2\}$.
\item[(iv)] The matrix $A$ has the form \eqref{eq:Aexplicit} with
\begin{equation} \label{eq:det=0}
 r_0^2 = \beta - 2\re(\alpha a), \qquad a = b_1+\frac{s_0}{r_1}\sigma_1.
\end{equation}
\end{itemize}
\end{prop}

\begin{proof}
Assume the hypothesis of the statement and let $\omega$ be the zig-zag basis which is orthonormal with respect to $v$. Using Proposition \ref{prop:asoc} we find that
$
{\cal C}A\omega = A{\cal D}A^{-1}A\omega = A{\cal D}\omega = zA\omega,
$
which means that $A\omega$ is a zig-zag basis of $\cal C$, so a positive rescaling of $\ell$ yields $\chi=A\omega$ and $\ell\omega=A^+\chi$. Since $A$ has the structure \eqref{eq:Aexplicit}, from the latter relation we conclude that $\ell\omega_n\in\spn\{\chi_n,\chi_{n+1},\chi_{n+2}\}$ and the functional $u$ of $\chi$ satisfies $u[\ell\omega_0]=r_0u[\chi_0]>0$ because $\chi_0$ is a positive constant and, by definition, $u[\chi_n\chi_{0*}]=\delta_{n,0}$. As a consequence, $u[\ell]>0$ and $(uc\ell-v)[\omega_n] = cu[\ell\omega_n]-v[\omega_n] = 0$ for $n\ge1$ and any $c\in\C$. Since $(uc\ell-v)[\omega_0]=0$ iff $c=v[1]/u[\ell]$, we get $uc\ell=v$ for such a positive value of $c$.

The characterizations of the hermiticity of $u$ given in $(i)$ and $(ii)$ follow from the previous result, Lemma \ref{lem:herm} and the fact that  $u[\chi_1]=0$ and $u[\chi_0]=u[\chi_{0*}]>0$ for the functional $u$ of $\chi$. 

Since $u[\chi_n]$ is non-zero only for $n=0$, the relation $\chi_{1*}\in\spn\{1,z^{-1}\}\subset\LL_2=\spn\{\chi_0,\chi_1,\chi_2\}$ shows that $u[\chi_{1*}]=0$ iff $\chi_{1*}\in\spn\{\chi_1,\chi_2\}$. This proves the equivalence $(ii)\Leftrightarrow(iii)$. 

As for the equivalence $(iii)\Leftrightarrow(iv)$, note that $\chi_{1*}\in\spn\{\chi_1,\chi_2\}$ iff $\{\chi_1,\chi_2,\chi_{1*}\}$ is linearly dependent because $\{\chi_1,\chi_2\}$ is linearly independent. From \eqref{eq:OLP*} we know that $\sigma_2\omega_2=\omega_{1*}+\overline{b}_2\omega_1$, which, combined with $\chi=A\omega$, provides the expansion
$$
\chi_{1*} = \overline{s}_0\omega_0+r_1\omega_{1*} =
\overline{s}_0\omega_0-r_1\overline{b}_2\omega_1+r_1\sigma_2\omega_2.
$$
Since $\chi_1=s_0\omega_0+r_1\omega_1$ and $\chi_2=t_0\omega_0+\overline{s}_1\omega_1+r_2\omega_2$, the linear dependence of $\{\chi_1,\chi_2,\chi_{1*}\}$ reads as  
$$
\left|
\begin{matrix}
s_0 & r_1 & 0
\\
t_0 & \overline{s}_1 & r_2
\\
\overline{s}_0 & -r_1\overline{b}_2 & r_1\sigma_2 
\end{matrix}
\right|
=0.
$$
Introducing the expressions of $t_0$, $s_1$, $r_2$ in terms of $r_0$, $s_0$, $r_1$ given in \eqref{eq:inverse2}, this determinantal condition becomes $(iv)$ after direct algebraic manipulations. 
\end{proof}

We are now ready to prove the characterization of the CMV solutions for the Darboux transformations with parameters.

\begin{thm} \label{thm:spurious}
Let $\ell$ be a Hermitian Laurent polynomial of degree one, ${\cal D}\in\mathscr{C}_\ell$ with Schur parameters $b_n$, $\sigma_n=\sqrt{1-|b_n|^2}$ and $v$ a related functional. Consider a reversed Cholesky factorization $\ell({\cal D})=A^+A$, $A\in\mathscr{T}$, the corresponding solution ${\cal C}=A{\cal D}A^{-1}$ of the Darboux transformation with parameters and a related zig-zag basis $\chi$. Then, $\cal C$ is CMV iff any of the following equivalent conditions is satisfied:
\begin{itemize}
\item[(i)] The first two rows of $\cal C$ constitute an orthonormal system, i.e. the leading submatrix of ${\cal C}{\cal C}^+$ of order 2 is the identity.
\item[(ii)] There exists a linear functional $u$ which solves $u\ell=v$, up to a positive rescaling of $\ell$, and such that $u[\chi^{(1)}\chi^{(1)+}]$ is the identity, where $\chi^{(1)}=(\chi_0,\chi_1)^t$.
\item[(iii)] The functional $u$ of $\chi$ satisfies $u[\chi_{1*}]=0$ and $u[\chi_1\chi_{1*}]=1$.
\item[(iv)] The basis $\chi$ satisfies $\chi_{1*}\in\spn\{\chi_1,\chi_2\}$ and
$$
(\chi_1/\chi_0)(z)=\rho^{-1}(z+a),
\qquad
|a|<1, \qquad \rho=\sqrt{1-|a|^2}.
$$
\item[(v)] If $\ell(z)=\alpha z+\beta+\overline\alpha z^{-1}$, the matrix $A$ has the form \eqref{eq:Aexplicit} with 
\begin{equation} \label{eq:CMVcond} 
 \kern-15pt 
 r_0^2 = \beta-2\re(\alpha a),
 \kern15pt 
 a = b_1 + \frac{s_0}{r_1}\sigma_1,
 \kern15pt 
 r_1 = r_0\frac{\sigma_1}{\rho}, 
 \kern15pt
 \begin{aligned} 
 	& |a|<1, 
 	\\ 
 	& \rho=\sqrt{1-|a|^2}.
 \end{aligned}
\end{equation}
\end{itemize}
Moreover, the parameter $a$ in $(iv)$ and $(v)$ is the first Schur parameter $a_1$ of $\cal C$.
\end{thm} 

\begin{proof}
From Proposition \ref{prop:uni-H} and Corollary \ref{cor:CMVdef}, the identity ${\cal C}{\cal C}^+=I$ is equivalent to the unitarity of $\cal C$ and thus to stating that it is CMV. Therefore, the equivalence with $(i)$ follows directly from \eqref{eq:no-unit} and the second statement of Proposition \ref{prop:kerB}, bearing in mind that $B=A^+$ and the number of zeros of $\ell$ is $N_z=2$ in the present case. 

Also, $\cal C$ is CMV iff $u[\chi\chi^+]=I$ for a linear functional $u$ in $\Lambda$, and in this case Theorem \ref{thm:Darboux} states that $u\ell=v$ up to a positive rescaling of $\ell$, thus $u$ satisfies $(ii)$. To prove the converse, consider the zig-zag basis $\omega$ which is orthonormal with respect to $v$. As in the proof of Proposition \ref{prop:herm}, $\chi=A\omega$ and $\ell\omega=A^+\chi$ up to a positive rescaling of $\ell$. Suppose that $u$ is an arbitrary linear functional solution of $u\ell=v$. Using Proposition \ref{prop:asoc} we get
$
A^+ = u\ell[\omega\omega^+] A^+ = 
u[\ell\omega(A\omega)^+] = A^+u[\chi\chi^+],
$
so that
\begin{equation} \label{eq:uK} 
u[\chi\chi^+] = I + K, \qquad A^+K=0.
\end{equation}
If $u[\chi^{(1)}\chi^{(1)+}]$ is the identity, then $u[\chi_0]=u[\chi_{0*}]>0$ and $u[\chi_1]=0=u[\chi_{1*}]$. This guarantees the hermiticity of $u$ due to Lemma \ref{lem:herm}. As a consequence, $K=K^+$, i.e. $K\in\mathscr{K}_{A^+}$. Then, \eqref{eq:uK} and the second statement of Proposition \ref{prop:kerB} ensure that $u[\chi\chi^+] = I$ whenever the leading submatrix of $u[\chi\chi^+]$ of order 2, which is $u[\chi^{(1)}\chi^{(1)+}]$, is the identity.

Concerning the equivalence with $(iii)$, note that a linear functional $u$ satisfying $u[\chi\chi^+]=I$ is necessarily the functional of $\chi$, which therefore must satisfy $u[\chi_{1*}]=0$ and $u[\chi_1\chi_{1*}]=1$ when $\cal C$ is unitary. Conversely, if the functional $u$ of $\chi$ satisfies these conditions, then $u[\chi^{(1)}\chi^{(1)+}]$ is the identity because by definition $u[\chi_n\chi_{0*}]=\delta_{n,0}$. Proposition \ref{prop:herm} ensures that the functional $u$ of $\chi$ solves $u\ell=v$ up to a positive rescaling of $\ell$, thus we conclude that $u$ satisfies $(ii)$ and hence $\cal C$ is CMV. 

As for the equivalence with $(iv)$, let us write explicitly $(\chi_1/\chi_0)(z)=\kappa(z+a)$, $\kappa>0$, $a\in\C$. If $\cal C$ is CMV then, from ${\cal C}\chi=z\chi$ and \eqref{eq:CMV-schur}, we get $a=a_1$ and $\kappa=\rho_1^{-1}$. 
Besides, \eqref{eq:OLP*} gives $\rho_2\chi_2=\chi_{1*}+\overline{a}_2\chi_1$, so $\chi_{1*}\in\spn\{\chi_1,\chi_2\}$. For the converse, note that the functional $u$ of $\chi$ satisfies $u[\chi_1]=u[\chi_2]=0$, thus $u[\chi_{1*}]=0$ under the condition $\chi_{1*}\in\spn\{\chi_1,\chi_2\}$. This condition also implies that $u[\chi_1\chi_{1*}]=\kappa^2(1-|a|^2)$ because
$
\chi_1\chi_{1*} = \chi_0\chi_{0*} \kappa^2 (1+|a|^2+\overline{a}z+az^{-1}) =
\chi_0\chi_{0*}\kappa^2(1-|a|^2) + 
\kappa\overline{a}\chi_1\chi_{0*} + \kappa a\chi_0\chi_{1*}
$
and $u[\chi_n\chi_{0*}]=\delta_{n,0}$. Bearing in mind the equivalence with $(iii)$, this proves that $\cal C$ is CMV under the hypothesis given in $(iv)$.

Finally, let us prove the equivalence with $(v)$. If $\omega$ is the orthonormal basis with respect to $v$ we know that $\chi=A\omega$ is a zig-zag basis of $\cal C$ and $\ell\omega=A^+\chi$. If $\cal C$ is CMV, using Proposition \ref{prop:asoc} we get $u\ell[\omega\omega^+] = u[\omega(\ell\omega)^+] = A^{-1}u[\chi\chi^+]A = I$, so $v=u\ell$. From ${\cal C}\chi=z\chi$ and \eqref{eq:CMV-schur} we find that $\chi_1=\chi_0\rho_1^{-1}(z+a_1)$, thus
$
\ell\chi_0\chi_{0*} = \alpha\rho_1\chi_1\chi_{0*} + \overline\alpha\chi_0\chi_{1*} + (\beta-2\re(\alpha a_1))\chi_0\chi_{0*}.
$
Since $\chi_0=r_0\omega_0$, we get $r_0^2=v[\chi_0\chi_{0*}]=u[\ell\chi_0\chi_{0*}]=\beta-2\re(\alpha a_1)$, which proves the first identity of $(v)$ for $a=a_1$. The rest of the identities follow from \eqref{eq:a-b}. Let us prove now the converse. The Schur parametrization \eqref{eq:CMV-schur} of $\cal D$ combined with ${\cal D}\omega=z\omega$ yields $-b_1\omega_0+\sigma_1\omega_1=z\omega_0$. Therefore, assuming $(v)$, $\chi=A\omega$ yields
$$
\chi_0 = r_0\omega_0,
\qquad
\chi_1 = s_0\omega_0+r_1\omega_1 = 
\frac{r_0}{\rho} (a\omega_0-b_1\omega_0+\sigma_1\omega_1) = 
\chi_0\frac{z+a}{\rho},
$$
which gives the second condition in $(iv)$. To prove that $\cal C$ is CMV we only need to show additionally that $\chi_{1*}\in\spn\{\chi_1,\chi_2\}$. This is equivalent to \eqref{eq:det=0}, which coincides with part of the conditions given in $(v)$. 
\end{proof} 

Proposition \ref{prop:herm}.$(iv)$ and Theorem \ref{thm:spurious}.$(v)$ provide characterizations of those reversed factorizations $\ell({\cal D})=A^+A$ leading respectively to Hermitian or positive definite functionals, via the Darboux transformation with parameters. Both characterizations are in the same spirit, i.e. they are given as restrictions on the initial parameters $r_0$, $s_0$, $r_1$ determining the matrix factor $A$. The (hermiticity) positive definiteness restrictions have the effect of reducing from four to (three) two the number of ``free" real parameters. Of course, not any choice of the parameters $r_0$, $s_0$, $r_1$ satisfying such restrictions leads necessarily to a reversed factorization $\ell({\cal D})=A^+A$, so the only conclusion is that the solutions of the Darboux transformation with parameters leading to (Hermitian) positive definite functionals are parametrized by at most (three) two real parameters. This is in agreement with our previous discussion based on the interpretation of inverse Darboux transformations in terms of measures on the unit circle. We will refer to \eqref{eq:det=0} and \eqref{eq:CMVcond} as the hermiticity and the CMV conditions respectively. Note that the hermiticity condition \eqref{eq:det=0} coincides with the first part of the CMV conditions \eqref{eq:CMVcond}. Therefore, the spurious solutions leading to Hermitian functionals are characterized by satisfying \eqref{eq:det=0}, but not the rest of the CMV conditions \eqref{eq:CMVcond}, i.e. either $|a|\ge1$ or $r_1\rho \ne r_0\sigma_1$.

When the CMV conditions are taken into account, the Darboux transformations with parameters can be iterated because the set $\mathscr{C}_\ell({\cal D})$ of CMV solutions is a subset of $\mathscr{C}_\ell$, the class of CMV matrices $\cal C$ which allow for a reversed Cholesky factorization of $\ell({\cal C})$ (see Appendix~\ref{app:cholesky}). 

We will finish this section illustrating in a detailed example the coexistence of CMV and spurious solutions of the Darboux transformation with parameters. To develop the example it will be useful to rewrite the relations \eqref{eq:a-b} from the perspective of the inverse Darboux transformations, that is,
\begin{equation} \label{eq:b-a}
\begin{gathered}
 r_n=\frac{\kappa_n}{\lambda_n}, 
 \qquad 
 t_n=\alpha\sigma_{n+1}\frac{\sigma_{n+2}}{r_{n+2}},
 \\[3pt]
 a_{n} = b_{n} + s_{n-1}\frac{\sigma_{n}}{r_n},
 \qquad
 \alpha a_{n+2} = \overline\alpha b_{n} - s_{n}\frac{r_{n+1}}{\sigma_{n+1}}, 
\end{gathered}
\qquad 
n\ge0.
\end{equation}

\begin{ex} \label{ex:PD-spurious} 
{\it Darboux transformation with parameters for $\ell$ given by $\alpha=-1$, $\beta=2$,
$$
\ell(z)=2-z-z^{-1}=(1-z)(1-z^{-1}),
$$
and ${\cal D}$ with constant positive Schur parameters $b_n=b\in(0,1)$ for $n\ge1$, 
$$
{\cal D} = 
\left(
\begin{smallmatrix} 
	-b & \kern1pt \sigma
	\\[3pt]
	-\sigma b & \kern1pt -b^2 & 
	\kern1pt -\sigma b & \kern1pt \sigma^2 
	\\[3pt]
    \sigma^2 & \kern1pt \sigma b & 
    \kern1pt -b^2 & \kern1pt \sigma b
    \\[3pt]
    & & \kern1pt -\sigma b & \kern1pt -b^2 &
    \kern1pt -\sigma b & \kern3pt \sigma^2
    \\[3pt]
    & & \kern1pt \sigma^2 & \kern1pt \sigma b &
    \kern1pt -b^2 & \kern3pt \sigma b
    \\[3pt]
    & & & & \kern1pt -\sigma b & \kern1pt -b^2 &
    \kern1pt -\sigma b & \kern3pt \sigma^2
    \\[3pt]
    & & & & \kern1pt \sigma^2 & \kern1pt \sigma b &
    \kern1pt -b^2 & \kern3pt \sigma b
    \\[3pt]
    & & & & & & \kern1pt \cdots & \kern3pt \cdots & \kern6pt \cdots
\end{smallmatrix}
\right),
\qquad
\sigma=\sqrt{1-b^2}.
$$} 

For simplicity, we will study only the solutions with $s_1=0$, a choice which fixes two of the four real parameters describing the solutions. Then, the relations \eqref{eq:inverse2} giving inductively the coefficients of the matrix $A$ can be written as
\begin{equation} \label{eq:inv-ex}
\begin{gathered}
s_0 = -\frac{\sigma}{r_1}(1-b),
\qquad
r_0^2 = 2(1+b)-s_0^2-\frac{\sigma^4}{r_2^2} 
= 2(1+b)-\frac{\sigma^2}{r_1^2}(1-b)^2-\frac{\sigma^4}{r_2^2},
\\
s_n = 0,
\qquad
r_{n+2}^2 = \frac{\sigma^4}{2(1+b^2)-r_n^2},
\qquad
t_{n-1} = -\frac{\sigma^2}{r_{n+1}}, 
\qquad 
n\ge1. 
\end{gathered}
\end{equation}
These relations have been expressed using $r_1$, $r_2$, instead of $r_0$, $r_1$, as the remaining two free parameters. This change is suggested by the role of $r_1^2$, $r_2^2$ as initial conditions of two identical recurrence relations, namely
\begin{equation} \label{eq:rec-x}
x_{n+1} = \frac{\sigma^4}{2(1+b^2)-x_n}, 
\qquad
x_n = 
\begin{cases} 
r_{2n+1}^2, 
\\ 
\kern9pt \text{\footnotesize or} 
\\ 
r_{2n+2}^2, 
\end{cases}
\quad
n\ge0.
\end{equation}
The recurrence relation with initial condition $r_1^2$ ($r_2^2$) yields the coefficients $r_n$, hence also $t_{n-2}=-\sigma^2/r_n$, for any odd (even) index $n$. 

Among the possible initial conditions $x_0=r_1^2,r_2^2>0$, those generating solutions of the Darboux transformation with parameters are characterized by the fact that \eqref{eq:inv-ex} gives $x_n=r_{2n+1}^2,r_{2n+2}^2>0$ for $n\ge0$, as well as by the condition 
\begin{equation} \label{eq:cond-r0}
\frac{\sigma^2}{r_1^2}(1-b)^2+\frac{\sigma^4}{r_2^2} < 2(1+b)
\end{equation} 
guaranteeing that $r_0^2>0$.

Due to the positivity requirement of $x_n$, the initial condition must satisfy $0<x_0<2(1+b^2)=\xi$. The recurrence relation \eqref{eq:rec-x} has two fixed points, $\xi_\pm=(1\pm b)^2<\xi$, which provide its constant positive solutions $x_n=\xi_\pm$. As for the non-constant solutions $x_n$, a quick analysis of their positivity follows from the interpretation of \eqref{eq:rec-x} as the implementation of Newton's method to the function $h$ given by
$$
h(x) =  \frac{|x-\xi_-|^{\frac{\xi_+}{4b}}}{|x-\xi_+|^{\frac{\xi_-}{4b}}}.
$$
The typical behaviour of $h$ is represented in Figure \ref{fig:h}. It is a non-negative concave function with a zero at $\xi_-$, a divergence at $\xi_+$ and a local minimum at $\xi=2(1+b^2)$. This guarantees that $x_n$ converges monotonically to $\xi_-$ whenever $x_0<\xi_+$, being strictly increasing if $x_0<\xi_-$ and strictly decreasing if $\xi_-<x_0<\xi_+$. On the other hand, every $x_0>\xi_+$ leads eventually to $x_n>\xi$ for some $n$, thus no such a choice yields a positive solution of \eqref{eq:rec-x}. Therefore, the positive solutions $x_n$ of \eqref{eq:rec-x} are those associated with an initial condition $x_0\in(0,\xi_+]$, and, up to the fixed point $x_0=\xi_+$, all these solutions converge monotonically to $\xi_-$.

From these results, and bearing in mind the restriction \eqref{eq:cond-r0}, we conclude that the solutions of the Darboux transformation with parameters corresponding to $s_1=0$ are given by \eqref{eq:inv-ex} for any choice of $r_1$, $r_2$ with the constraints
\begin{equation} \label{eq:sol-ex}
0 < r_1,r_2 \le 1+b, 
\qquad
\frac{1-b}{r_1^2}+\frac{1+b}{r_2^2} < \frac{2}{(1-b)^2}.
\end{equation}
 
\begin{figure}
\begin{center}
\includegraphics[width=12.5cm]{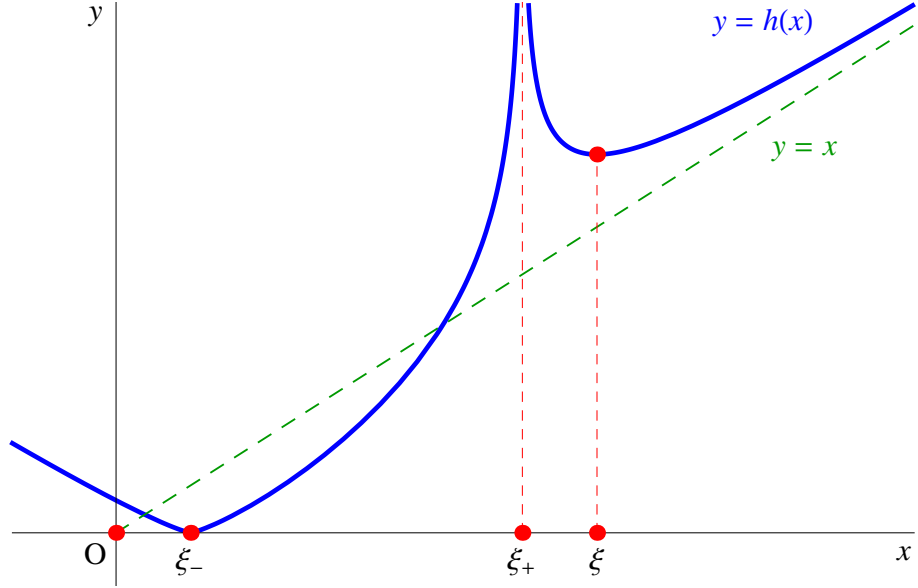}
\end{center}
\caption{The function $h(x)$ whose Newton's algorithm yields the recurrence for the sequence $x_n$ given by \eqref{eq:rec-x}. It is a non-negative concave function with a zero at $\xi_-=(1-b)^2$, a divergence at $\xi_+=(1+b)^2$ and a local minimum at $\xi=2(1+b^2)$.} \label{fig:h}
\end{figure}

Let us discuss now the restrictions imposed on these solutions by the hermiticity and CMV conditions. According to \eqref{eq:det=0}, the solutions associated with an Hermitian linear functional are characterized by
\begin{equation} \label{eq:herm-ex}
r_0^2 = 2(1+a), 
\qquad
a = b - \frac{\sigma^2(1-b)}{r_1^2}.
\end{equation} 
Taking into account the expression of $r_0$ in terms of $r_1$, $r_2$ given in \eqref{eq:inv-ex}, the hermiticity condition becomes $r_1 = r_2$. In view of \eqref{eq:sol-ex}, the solutions of the Darboux transformations leading to Hermitian functionals are given by \eqref{eq:inv-ex} with
$$
r_1=r_2=r \in (1-b,1+b],
$$
and the related value of $r_0$ is
$$
r_0^2 = 2(1+b) \left(1-\frac{(1-b)^2}{r^2}\right).
$$ 

The above solutions include the CMV ones which, according to \eqref{eq:CMVcond}, are characterized by the additional conditions $|a|<1$ and $r\rho=r_0\sigma$, $\rho=\sqrt{1-a^2}$. Using \eqref{eq:herm-ex} we obtain
$$
|a|<1 
\;\Leftrightarrow\; 
r > 1-b,
\qquad\qquad
r\rho = r_0\sigma
\;\Leftrightarrow\;
r = 1+b.
$$
Therefore, the only CMV solution with $s_1=0$ is that one obtained from \eqref{eq:inv-ex} with
$$
r_1=r_2=1+b,
$$ 
which corresponds to constant coefficients $r_n=1+b$ for $n\ge1$, associated with the fixed point $\xi_+$ of \eqref{eq:rec-x}.

From \eqref{eq:herm-ex}, \eqref{eq:inv-ex} and \eqref{eq:b-a} we find the corresponding Schur parameters
$$
a_n = 
\begin{cases}
\displaystyle \frac{3b-1}{1+b}, & n=1,
\\ 
b, & n\ge2,
\end{cases}
$$
which provide the referred CMV solution ${\cal C}=A{\cal D}A^{-1}$ of the Darboux transformation with parameters, 
$$
{\cal C} = 
\left(
\begin{smallmatrix} 
	-a & \kern1pt \rho
	\\[3pt]
	-\rho b & \kern1pt -ab & 
	\kern1pt -\sigma b & \kern1pt \sigma^2 
	\\[3pt]
    \rho\sigma & \kern1pt \sigma a & 
    \kern1pt -b^2 & \kern1pt \sigma b
    \\[3pt]
    & & \kern1pt -\sigma b & \kern1pt -b^2 &
    \kern1pt -\sigma b & \kern3pt \sigma^2
    \\[3pt]
    & & \kern1pt \sigma^2 & \kern1pt \sigma b &
    \kern1pt -b^2 & \kern3pt \sigma b
    \\[3pt]
    & & & & \kern1pt -\sigma b & \kern1pt -b^2 &
    \kern1pt -\sigma b & \kern3pt \sigma^2
    \\[3pt]
    & & & & \kern1pt \sigma^2 & \kern1pt \sigma b &
    \kern1pt -b^2 & \kern3pt \sigma b
    \\[3pt]
    & & & & & & \kern1pt \cdots & \kern3pt \cdots & \kern6pt \cdots
\end{smallmatrix}
\right),
\qquad
\begin{aligned}
& a=\frac{3b-1}{1+b}, 
\\[2pt]
& \rho=\sqrt{1-a^2}=\frac{2}{1+b}\sqrt{b(1-b)}.
\end{aligned}
$$
The matrix $A$, coming from a reversed Cholesky factorization $\ell({\cal D})=A^+A$, and giving the Cholesky factorization $\ell({\cal C})=AA^+$, has the form
$$
A = 
\left(
\begin{smallmatrix} 
	\\
	r_0
	\\[3pt] 
	s_0 & \kern3pt r
	\\[3pt] 
	t & \kern3pt 0 & \kern5pt r
	\\[3pt]
	& \kern3pt t & \kern5pt 0 & \kern2pt r
	\\[3pt]
	& \kern3pt & \kern5pt t & \kern2pt 0 & r
	\\[-3pt]
	& \kern3pt & \kern5pt 
	& \kern2pt \scriptsize\ddots 
	& \scriptsize\ddots
	& \scriptsize\ddots
\end{smallmatrix}
\right),
\qquad
\begin{aligned}
& r_0 = 2\sqrt{\frac{b}{1+b}}, & \quad & s_0=-\sigma\frac{1-b}{1+b},
\\[4pt]
& r=1+b, & & t=-(1-b).
\end{aligned}
$$

The CMV matrix $\cal D$ is related to an absolutely continuous measure $\nu$ supported on the arc $\Gamma=\{e^{i\theta}:|\sin\frac{\theta}{2}|\ge b\}$ given by
$$
d\nu(e^{i\theta}) = w(\theta)\,d\theta,
\qquad
w(\theta) = \frac{\sqrt{\sin^2\frac{\theta}{2}-b^2}}{|\sin\frac{\theta}{2}|},
\qquad
e^{i\theta}\in\Gamma.
$$
Since $\ell(z)=(z-1)(z^{-1}-1)$ has a single zero at $z=1$, the above CMV solution $\cal C$ of the Darboux transformation with parameters must be associated with a measure $d\mu=d\nu/\ell+m\delta_1$ for some $m\ge0$. 

To analyze the value of the mass $m$, let us have a look at the Schur function $f$ of $\mu$. The Schur function $f$ has an analytic continuation through the essential gap $\{e^{i\theta}:|\sin\frac{\theta}{2}|<b\}$, and $z=1$ is a mass point of $\mu$ iff it is a solution of $zf(z)=1$, i.e. $f(1)=1$. On the other hand, Geronimus' theorem asserts that the application of the Schur algorithm to $f$ generates the sequence $(-a,-b,-b,-b,\dots)$. Analogously, the Schur function $g$ of $\nu$ is characterized by the constant sequence $(-b,-b,-b,\dots)$ arising from the Schur algorithm. Therefore, $g$ is obtained from $f$ after a single step of the Schur algorithm,
$$
g(z) = \frac{1}{z} \frac{f(z)+a}{1+af(z)}.
$$
This relation implies that $f(1)=1$ iff $g(1)=1$, which is not possible because 1 is not a mass point of $\nu$. We conclude that $m=0$ and 
$$
d\mu(e^{i\theta}) = \frac{w(\theta)}{|e^{i\theta}-1|^2}\,d\theta
= \frac{w(\theta)}{2\sin^2\frac{\theta}{2}}\,d\theta, 
\qquad
e^{i\theta}\in\Gamma.
$$
In other words, the arc $\Gamma$ is the common spectrum of $\cal C$ and $\cal D$, which are isospectral.   

The spurious solutions can be also explicitly described. For instance, the choice $s_1=0$, $r_1=1-b$, $r_2=1+b$ satisfies \eqref{eq:sol-ex} and yields $s_n=0$, $r_{2n-1}=r_1$, $r_{2n}=r_2$ for $n\ge1$, so that
$$
A = 
\left(
\begin{smallmatrix} 
	\\
	r_0
	\\[3pt] 
	s_0 & \kern-1pt r_1
	\\[3pt] 
	-r_1 & \kern-1pt 0 & \kern-1pt r_2
	\\[3pt]
	& \kern-1pt -r_2 & \kern-1pt 0 & \kern1pt r_1
	\\[3pt]
	& \kern-1pt & \kern-1pt -r_1 & \kern1pt 0 & \kern2pt r_2
	\\[-3pt]
	& \kern-1pt & \kern-1pt  	
	& \kern1pt \scriptsize\ddots 
	& \kern2pt \scriptsize\ddots 
	& \kern2pt \scriptsize\ddots
\end{smallmatrix}
\right),
\qquad
\begin{aligned}
& r_0 = 2\sqrt{b}, & \quad & s_0=-\sigma,
\\[4pt]
& r_1=1-b, & & r_2=1+b.
\end{aligned}
$$
Since $r_1 \ne r_2$, the corresponding solution ${\cal C}=A{\cal D}A^{-1}$ is related to a non-Hermitian functional. Explicitly,
$$
{\cal C} =
\left(
\begin{smallmatrix}
 1 &  \sigma \sqrt{b} \frac{2}{1-b} & 0 & 0 & 0 & 0 & \cdots
 \\[2pt]
 \sigma \sqrt{b} \frac{2b}{1+b} & b & -\sigma b \frac{1-b}{1+b} 
 & \sigma^2 & 0 & 0 & \cdots 
 \\[2pt]
 \sqrt{b}\frac{2(1+b^2)}{1-b} & \sigma \frac{b^2+4b-1}{(1-b)^2} & -b^2 
 & \sigma b \frac{1+b}{1-b} & 0 & 0 & \cdots
 \\[2pt]
 \sigma \sqrt{b} \frac{4b}{(1+b)^2} & 0 & \sigma b \frac{b^2+4b-1}{(1+b)^2} 
 & -b^2 & -\sigma b \frac{1-b}{1+b} & \sigma^2 & \cdots
 \\[2pt]
 \sqrt{b} \frac{4b(1+b)}{(1-b)^2} & \sigma \frac{8b^2}{(1-b)^3} & \sigma^2 
 & \sigma b \frac{1+4b-b^2}{(1-b)^2} & -b^2 & \sigma b \frac{1+b}{1-b} & \cdots
 \\[4pt]
 \cdots & \cdots & \cdots & \cdots & \cdots & \cdots & \cdots
\end{smallmatrix}
\right),
$$
which, solving ${\cal C}\chi=z\chi$, yields an associated zig-zag basis $\chi$ with 
$$
\chi_0 = 1, 
\quad 
\chi_1 = \frac{1-b}{2\sigma\sqrt{b}}(z-1),
\quad
\chi_2 = \frac{1}{2\sqrt{b}(1-b)}(bz+3b-1+z^{-1}),
\quad
\dots
$$
Hence, the conditions $u[\chi_n\chi_{0*}]=\delta_{n,0}$ defining the associated functional $u$ give 
$$
u[1]=1=u[z], \qquad u[z^{-1}]=1-4b \ne \overline{u[z]}, 
$$
which show directly the lack of hermiticity of $u$.
\end{ex}

\section{Jacobi versus CMV}
\label{sec:J-CMV} 

So far, we have defined the Darboux transformations for CMV matrices by analogy with the Darboux transformations for Jacobi matrices. Nevertheless, this does not mean necessarily that Darboux for Jacobi and CMV share all their properties. An explicit comparison between these two versions of Darboux is necessary to understand to which extent the uses of Darboux for Jacobi could be exported to CMV. This comparison should highlight the similarities and differences between these two transformations, showing also the direct links between them if any. These are the objectives of the present section. Thus, we will first review the main features of Darboux for Jacobi, translating the standard approach based on LU factorizations into an equivalent one which uses Cholesky factorizations for a better comparison with Darboux for CMV. This review will be used simultaneously to exhibit the analogies and differences between Darboux for Jacobi and CMV. Besides, the classical Szeg\H{o} connection \cite{Sz} between orthogonal polynomials on the real line and the unit circle will provide a direct link between both transformations. This not only supports the present version of Darboux for CMV as the natural unitary analogue of Darboux for Jacobi, but also serves as a communicating channel between both transformations and their applications. We will also show the role of the Darboux transformations in a more recent connection between the real line and the unit circle due to Derevyagin, Vinet and Zhedanov \cite{DVZ}.    

The standard procedure for the Darboux transformation of a Jacobi matrix 
$$
{\cal J} = 
\left(
\begin{smallmatrix}
 * & \kern2pt +
 \\[3pt] 
 + & \kern2pt * & \kern1pt + 
 \\[3pt] 
 & \kern2pt + & \kern1pt * & + 
 \\[-2pt] 
 && \kern1pt \scriptsize\ddots & \scriptsize\ddots & \scriptsize\ddots 
\end{smallmatrix}
\right)
= {\cal J}^+
$$
uses LU instead of Cholesky factorizations. Actually, the usual starting point is not the Jacobi matrix itself, but the tridiagonal one    
\begin{equation} \label{eq:Jmonic}
\widetilde{\cal J} = 
\left( 
\begin{smallmatrix}
 * & \kern3pt 1
 \\[3pt] 
 + & \kern3pt * & \kern2pt 1 
 \\[3pt] 
 & \kern3pt + & \kern2pt * & 1
 \\[-1pt] 
 && \kern2pt \scriptsize\ddots & \scriptsize\ddots 
 & \scriptsize\ddots 
\end{smallmatrix}
\right)
\end{equation}
obtained by conjugating $\cal J$ with a positive diagonal matrix $\Pi_{\cal J}$, i.e. $\widetilde{\cal J}=\Pi_{\cal J}{\cal J}\Pi_{\cal J}^{-1}$. The Darboux transformation starts by choosing $\beta\in\R$ such that the LU factorization $\widetilde{\cal J}+\beta I=\widetilde{L}\widetilde{U}$ is available, and then generates via the identity $\widetilde{\cal K}+\beta I=\widetilde{U}\widetilde{L}$ a new tridiagonal matrix $\widetilde{\cal K}$ with the shape \eqref{eq:Jmonic}, except that the entries of its lower diagonal can be signed. When such a diagonal is positive, $\widetilde{\cal K}$ can be symmetrized by conjugation with a positive diagonal matrix $\Pi_{\cal K}$, leading to a Jacobi matrix ${\cal K}=\Pi_{\cal K}^{-1}\widetilde{\cal K}\Pi_{\cal K}$. 

Note that the LU factorization $\widetilde{\cal J}+\beta I=\widetilde{L}\widetilde{U}$ exists simultaneously to that of ${\cal J}+\beta I$, whose LU factors are $L=\Pi_{\cal J}^{-1}\widetilde{L}\Pi_{\cal J}$ and $U=\Pi_{\cal J}^{-1}\widetilde{U}\Pi_{\cal J}$. Therefore, both LU factorizations are possible whenever ${\cal J}+\beta I$ is positive definite. We are going to see that in this case the Darboux transformations of Jacobi matrices can be understood  in terms of Cholesky factorizations, similarly to the developed Darboux transformations for CMV matrices. This will make it easier the comparison between Darboux for Jacobi and CMV.   

The matrices involved in the above discussion are determined by two real sequences $v_n$ and $u_n$ defined by the lower and main diagonals of the factors $\widetilde{L}$ and $\widetilde{U}$, respectively,
\begin{equation} \label{eq:LU}
\widetilde{L} = 
\left( 
\begin{smallmatrix}
 1
 \\[3pt] 
 v_1 & 1 
 \\[3pt] 
 & v_2 & 1
 \\[-1pt] 
 && \scriptsize\ddots & \scriptsize\ddots 
\end{smallmatrix}
\right),
\qquad
\widetilde{U} = 
\left( 
\begin{smallmatrix}
 u_1 & 1
 \\[3pt] 
 & u_2 & 1 
 \\[3pt] 
 && u_3 & 1
 \\[-1pt] 
 &&& \scriptsize\ddots & \scriptsize\ddots 
\end{smallmatrix}
\right).
\end{equation}
More precisely, $\widetilde{\cal J}$, $\widetilde{\cal K}$ and their symetrizing matrices $\Pi_{\cal J}$, $\Pi_{\cal K}$ are given by 
\begin{equation} \label{eq:uv}
\begin{gathered}  
\kern-9pt \widetilde{\cal J}+\beta I = 
\left(
\begin{smallmatrix}
 u_1 & 1
 \\[5pt] 
 v_1u_1 & v_1+u_2 & 1 
 \\[5pt] 
 & v_2u_2 & v_2+u_3 & \kern9pt 1
 \\ 
 && \kern-22pt \scriptsize\ddots & \kern-10pt \scriptsize\ddots 
 & \kern7pt \scriptsize\ddots 
\end{smallmatrix}
\right),
\quad 
\widetilde{\cal K}+\beta I = 
\left(
\begin{smallmatrix}
 v_1+u_1 & 1
 \\[5pt] 
 v_1u_2 & v_2+u_2 & 1 
 \\[5pt] 
 & v_2u_3 & v_3+u_3 & \kern9pt 1
 \\ 
 && \kern-22pt \scriptsize\ddots & \kern-10pt \scriptsize\ddots 
 & \kern7pt \scriptsize\ddots 
\end{smallmatrix}
\right),
\\[5pt]
\kern-9pt \Pi_{\cal J}^2 = 
\left( 
\begin{smallmatrix}
 1 
 \\[3pt] 
 & \pi_1  
 \\[3pt] 
 && \pi_2 
 \\[-2pt] 
 &&& \scriptsize\ddots  
\end{smallmatrix}
\right),
\quad
\Pi_{\cal K}^2 = \frac{1}{u_1} \Pi_{\cal J}^2 U_0,
\quad
U_0 =
\left( 
\begin{smallmatrix}
 \\
 u_1 
 \\[3pt] 
 & u_2  
 \\[3pt] 
 && u_3 
 \\[-2pt] 
 &&& \scriptsize\ddots  
\end{smallmatrix}
\right),
\quad 
\pi_n = \prod_{k=1}^n u_kv_k.
\end{gathered}
\end{equation} 
By hypothesis, $\widetilde{\cal J}$ is symmetrizable, i.e. $u_nv_n>0$ for all $n$. Then, $\widetilde{\cal K}$ is also symmetrizable iff $v_nu_{n+1}>0$ for all $n$, which, bearing in mind the previous hypothesis, means that the sequences $u_n$ and $v_n$ have the same constant sign. This is the case when ${\cal J}+\beta I$ is positive definite, which is equivalent to state that $u_n,v_n>0$ for all $n$ and, thus, also to the positive definiteness of ${\cal K}+\beta I$.

It is known that LU and Cholesky factorizations of positive definite matrices are closely related. If ${\cal J}+\beta I$ is positive definite, its LU factors provide the Cholesky factorization 
\begin{equation} \label{eq:Ch-J}
{\cal J}+\beta I=AA^+, \qquad A=LU_0^{1/2}, \qquad A^+=U_0^{-1/2}U,
\end{equation} 
where $U_0^{1/2}$ is the positive square root of the matrix $U_0$ given in \eqref{eq:uv}. Then, using \eqref{eq:uv} and \eqref{eq:Ch-J} we get
$$
{\cal K}+\beta I = \Pi_{\cal K}^{-1} \widetilde{U}\widetilde{L} \Pi_{\cal K} 
= \Pi_{\cal K}^{-1}\Pi_{\cal J} UL \Pi_{\cal J}^{-1}\Pi_{\cal K}
= U_0^{-1/2} UL U_0^{1/2} = A^+A.
$$

We conclude that the standard Darboux transformations relating Jacobi matrices $\cal J$ and $\cal K$ can be rewritten in the following way: take a real polynomial $\wp(x)=x+\beta$ such that $\wp({\cal J})$ is positive definite. Then, a new Jacobi matrix $\cal K$ is defined resorting to the Cholesky factorization $\wp({\cal J})=AA^+$, $A\in\mathscr{T}$, and reversing it as $\wp({\cal K})=A^+A$. The above procedure can be directly generalized to an arbitrary real polynomial $\wp(x)=\alpha x+\beta$ of degree one such that $\wp({\cal J})$ is positive definite, leading to a triangular matrix $A$ with the 2-band structure
$$
A = 
\left(
\begin{smallmatrix}
 + &
 \\[4pt] 
 \circledast & \kern2pt + 
 \\[4pt] 
 & \kern2pt \circledast & + 
 \\[-2pt] 
 && \scriptsize\ddots & \scriptsize\ddots 
\end{smallmatrix}
\right),
$$ 
whose lower diagonal has non-zero entries with the same sign as $\alpha$. 
This defines the Darboux transformation ${\cal J}\overset{\wp}{\mapsto}{\cal K}$. 

Jacobi matrices encode the three-term recurrence relation of orthogonal polynomials with respect to a positive Borel measure supported on an infinite subset of the real line (`measure on the real line' in short). Analogously to the CMV case, this links Jacobi matrices and positive definite linear functionals in the space of real polynomials. More generally, tridiagonal matrices with non-null entries in the lower and upper diagonals are related to polynomials which are orthogonal with respect to a quasi-definite linear functional not necessarily associated with a positive measure on the real line. An advantage of defining Darboux transformations via LU factorizations is that they apply directly to the general quasi-definite case. Although the Cholesky approach to Darboux transformations previously described has {\it a priori} the drawback of being applicable only to the positive definite case, Section~\ref{ssec:quasi} will show how to deal with Darboux transformations relating quasi-definite functionals by using generalized Cholesky factorizations.  

Analogously to the CMV case, the Darboux transformations of Jacobi matrices based on Cholesky factorizations can be embedded into a general procedure for the factorization of a real polynomial $\wp$ evaluated on a Jacobi matrix $\cal J$. A key role is played by the basis $p=(p_0,p_1,p_2,\dots)^t$ of the real vector space $\PP=\spn\{x^n\}_{n=0}^\infty=\R[x]$ of real polynomials defined up to a constant factor by 
$$
xp={\cal J}p.
$$
These polynomials, which we will assume with positive leading coefficient in what follows, are orthonormal with respect to a positive Borel measure on the real line, or equivalently, with respect to a positive definite linear functional on $\PP$. Given any other basis $q=(q_0,q_1,q_2,\dots)^t$ of orthonormal polynomials related to a Jacobi matrix $\cal K$, the matrix $A\in\mathscr{T}$ and the lower Hessenberg type one $B$ defined by
\begin{equation} \label{eq:pAq}
p=Aq, \qquad \wp q=Bp,
\end{equation}
satisfy
$$
\wp({\cal J})=AB, \qquad \wp({\cal K})=BA, 
\qquad {\cal J}A=A{\cal K}, \qquad B{\cal J}={\cal K}B.
$$
The result is valid for any arbitrary real polynomial $\wp$ of any degree, and the proof is a direct translation of that of Proposition \ref{prop:D-fac}. As in the CMV case, the factor $A$ plays the role of a change of basis relating the Jacobi matrices by conjugation since ${\cal K}=A^{-1}{\cal J}A$ and ${\cal J}=A{\cal K}A^{-1}$. Nevertheless, in contrast to the CMV case, if $\wp$ has degree one, the relations $\wp({\cal J})=AB$ and $\wp({\cal K})=BA$ are enough to generate one of the matrices $\cal J$, $\cal K$ starting from the other one.

When $\wp$ has degree one, $\wp({\cal J})$ is tridiagonal and $B$ is strictly lower Hessenberg. The Darboux transformation amounts to choosing $A$ as the left factor of the Cholesky factorization of $\wp({\cal J})$ when it is positive definite, so that $A$ is 2-band with positive entries and $B=A^+$. Furthermore, Theorem \ref{thm:Darboux} has an exact analogue for Jacobi matrices whose proof follows from similar arguments. 

\begin{thm} \label{thm:Darboux-J}
Let $u$, $v$ be positive definite functionals in $\PP$ with orthonormal polynomials $p$, $q$ and Jacobi matrices $\cal J$, $\cal K$ respectively. Then, the following statements are equivalent:
\begin{itemize}
\item[(i)] ${\cal J}\overset{\wp}{\mapsto}{\cal K}$ for a real polynomial $\wp$ of degree one, i.e. $\wp({\cal J})$ is positive definite and ${\cal K}=A^{-1}{\cal J}A$, where $A\in\mathscr{T}$ comes from the Cholesky factorization $\wp({\cal J})=AA^+$.
\item[(ii)] $v=u\wp$ for a real polynomial $\wp$ of degree one.
\item[(iii)] $p_n\in\spn\{q_{n-1},q_n\}$ for $n\ge1$ and $p_n\notin\spn\{q_n\}$ for some $n$.
\item[(iv)] There is a real polynomial $\wp\notin\spn\{p_1\}$ of degree one such that $\wp q_n\in\spn\{p_n,p_{n+1}\}$ for $n\ge0$.
\end{itemize}
The polynomials $\wp$ mentioned in $(i)$ and $(ii)$ coincide up to a constant positive factor, and they also coincide with that one in $(iv)$ up to a constant real factor. If $\wp$ coincides exactly in $(i)$ and $(ii)$ then $p=Aq$ and $\wp q=A^+p$ without any rescaling.
\end{thm}

In particular, the transformation ${\cal J}\overset{\wp}{\mapsto}{\cal K}$ is equivalent to multiplying by (a positive rescaling of) $\wp$ in terms of orthogonality measures, which is known as a Christoffel transformation on the real line. As a consequence, the Darboux transformations of Jacobi matrices are almost isospectral since they preserve the spectrum up to the elimination of, at most, a single eigenvalue. This is similar to what happens to Darboux for CMV, although these transformations may delete two eigenvalues.  

The only spectral change caused by the corresponding inverse transformations for Jacobi matrices is therefore the possible addition of an eigenvalue.
The inverse Darboux transformations of Jacobi matrices are the matrix realization of the Geronimus transformations for measures on the real line, which divide a measure by a polynomial $\wp$ of degree one and add a Dirac delta at the zero of $\wp$. On the other hand, in the language of Cholesky factorizations, the Darboux transformations with parameters for a Jacobi matrix $\cal K$ start with a reversed factorization $\wp({\cal K})=A^+A$, $A\in\mathscr{T}$, and then generate a new Jacobi matrix $\cal J$ by either the relation $\wp({\cal J})=AA^+$ or equivalently the conjugation ${\cal J}=A{\cal K}A^{-1}$. 

At this point, a major difference appears concerning the comparison between Darboux for CMV and Jacobi matrices: if $\wp$ is of degree one, the freedom in the reversed Cholesky factorization $\wp({\cal K})=A^+A$ is encoded by a single real parameter, the $(0,0)$ entry of $A$. This coincides with the number of free parameters in the Geronimus transformations on the real line, which are parametrized by the mass of the Dirac delta at the single zero of $\wp$. Actually, contrary to the CMV case, every solution of the Jacobi version for the Darboux transformations with parameters corresponds to a solution of the corresponding inverse Darboux transformations. In other words, the Darboux transformations with parameters for Jacobi matrices present no spurious solutions, so that these transformations become an exact matrix realization of the Geronimus transformations for measures on the real line. This is in striking contrast with the CMV version of Darboux transformations with parameters. 

Due to the absence of spurious solutions, the Jacobi version of Darboux transformations only involves band matrices and their inverses. Therefore, in view of Proposition \ref{prop:asoc}, in the Jacobi case the associativity of matrix multiplications holds automatically for all the Darboux relations which exclude such inverses. This will be used in what follows to omit discussing associativity issues when dealing with relations involving only band matrices.    

Despite the differences above mentioned, we have found many analogies between Jacobi and CMV which justify the name `Darboux' for the transformations of CMV matrices introduced in this paper. This similarity can be further highlighted by resorting to the Szeg\H o connection between orthogonal polynomials on the real line and the unit circle \cite{Sz}. This connection will be used in next section for a closer comparison between Darboux for Jacobi and CMV, but it can eventually be used to exchange ideas, results, techniques and applications between these two versions of Darboux.

\subsection{The Szeg\H o connection}
\label{ssec:Szego}

The aim of this section is to show that the Szeg\H{o} connection \cite{Sz} maps exactly Darboux for Jacobi into Darboux for CMV as defined in this paper. As a byproduct, we will obtain a direct relation between the Darboux factors for Jacobi and CMV. With this in mind, we will first review a matrix approach to the Szeg\H{o} connection due to Killip and Nenciu \cite{KN2004} (see also \cite{Si2005}) which serves our interests.

The Szeg\H o connection is based on the map $z \mapsto x=z+z^{-1}$, which transforms the unit circle onto the interval $[-2,2]$ mapping conjugated points onto the same image. This induces the Szeg\H o projection $\mu\xmapsto{\Sz}\Sz(\mu)$, a one-to-one correspondence between symmetric measures $\mu$ on the unit circle (i.e., measures which are invariant under conjugation $z\to\overline{z}$) and measures $\Sz(\mu)$ on $[-2,2]$, which is given by
$$
\int P(x)\,d(\Sz(\mu))(x) = \int P(z+z^{-1})\,d\mu(z), 
\qquad
\forall P\in\PP. 
$$
The symmetry of $\mu$ means that it has orthonormal polynomials $\varphi_n$ with real coefficients, i.e. the corresponding Schur parameters $a_n$ are real. Then, the relations
\begin{equation} \label{eq:pepo}
\begin{aligned}
	& p_n^{(e)}(x) = 
	\frac{z^{-n}(\varphi_{2n}^*(z)+\varphi_{2n}(z))}
	{\sqrt{2(1+a_{2n})}},
	\\[2pt] 
	& p^{(o)}_{n-1}(x) = 
	\frac{z^{-n}(\varphi_{2n}^*(z)-\varphi_{2n}(z))}
	{\sqrt{2(1-a_{2n})}(z^{-1}-z)}, 
\end{aligned}
\qquad\qquad
x=z+z^{-1},
\end{equation}
provide two sequences of polynomials in the variable $x=z+z^{-1}$, 
$$
p^{(k)}=(p^{(k)}_0,p^{(k)}_1,p^{(k)}_2,\dots)^t, \qquad k=e,o.
$$
We use the convention $a_0=1$ so that $p^{(e)}_0(x)=\varphi_0(z)$. The sequences $p^{(e)}$ and $p^{(o)}$ are orthonormal with respect to the measures on $[-2,2]$ given respectively by
$$
\mu_e = \Sz(\mu), \qquad d\mu_o(x)=(4-x^2)\,d\mu_e(x),
$$
Due to the orthonormality with respect to a measure on the real line, there exist Jacobi matrices ${\cal J}_e$, ${\cal J}_o$, such that 
\begin{equation} \label{eq:rec-p}
xp^{(k)}={\cal J}_kp^{(k)}, \qquad k=e,o.
\end{equation}
The relation between the measures $\mu_e$ and $\mu_o$ implies that ${\cal J}_e$ and ${\cal J}_o$ are connected by a Darboux transformation ${\cal J}_e \xmapsto{4-x^2} {\cal J}_o$ generated by a polynomial of degree two (see Section \ref{ssec:deg} for the generalization of Darboux transformations to polynomials of higher degree). 

For convenience, we will work with the matrix functions
\begin{equation} \label{eq:pepobis}
\bs{p}_n =
\begin{pmatrix}
 p_n & 0 \\ 0 & \widetilde{p}_n
\end{pmatrix},
\qquad\quad
\begin{aligned}
 & p_n(z) = p^{(e)}_n(x), 
 \\
 & \widetilde{p}_n(z) = (z^{-1}-z)p^{(o)}_n(x),
\end{aligned}
\qquad\quad
x=z+z^{-1}, 
\end{equation}
which should be considered as functions of $z$ instead of $x$ due to the presence of the factor $z-z^{-1}$ in $\widetilde{p}_n$. Bearing in mind \eqref{eq:pepobis}, the recurrence relations \eqref{eq:rec-p} can be rewritten in terms of $p=(p_0,p_1,p_2,\dots)^t$ and $\widetilde{p}=(\widetilde{p}_0,\widetilde{p}_1,\widetilde{p}_2,\dots)^t$ as
$$
xp(z)={\cal J}_ep(z), 
\qquad 
x\widetilde{p}(z)={\cal J}_o\widetilde{p}(z), 
\qquad x=z+z^{-1}.
$$
Equivalently, using the notation $\bs{p}=(\bs{p}_0,\bs{p}_1,\bs{p}_2,\dots)^t$,
\begin{equation} \label{eq:rec-pp}
x \bs{p}(z) = \bs{\cal J} \bs{p}(z),
\qquad
\bs{\cal J} = {\cal J}_e \oplus {\cal J}_o,
\qquad
x=z+z^{-1},
\end{equation}
with ${\cal J}_e$ (${\cal J}_o$) acting on even (odd) indices. More explicitly, $\bs{\cal J}$ is the $2\times2$-block Jacobi matrix with entries 
$$
\bs{\cal J}_{2i,2j} = ({\cal J}_e)_{i,j}, 
\qquad 
\bs{\cal J}_{2i+1,2j+1} =({\cal J}_o)_{i,j}, 
\qquad 
\bs{\cal J}_{2i,2j+1} = \bs{\cal J}_{2i+1,2j}=0,
$$
and can be understood as the block Jacobi matrix associated with the diagonal matrix of measures 
$$
\bs{\mu} =
\begin{pmatrix}
\mu_e & 0 \\ 0 & \mu_o
\end{pmatrix}.
$$
On the other hand, \eqref{eq:OLP-OP} allows us to rewrite $p_n$ and $\widetilde{p}_n$ in terms of the orthonormal Laurent polynomials $\chi_n$ associated with $\mu$,
\begin{equation} \label{eq:p-chi}
p_n = \frac{\chi_{2n}+\chi_{2n*}}{\sqrt{2(1+a_{2n})}},
\qquad\qquad
\widetilde{p}_{n-1} = \frac{\chi_{2n}-\chi_{2n*}}{\sqrt{2(1-a_{2n})}}.
\end{equation}
If $\cal C$ is the CMV matrix related to $\mu$, we denote $\bs{\mu}=\bs{\Sz}(\mu)$, $\bs{p}=\bs{\Sz}(\chi)$, $\bs{\cal J}=\bs{\Sz}({\cal C})$, referring to them as the matrix Szeg\H o projections of $\mu$, $\chi$, $\cal C$, respectively, in contrast to the scalar Szeg\H o projections $\mu_e=\Sz(\mu)$, $p^{(e)}=\Sz(\mu)$, ${\cal J}_e=\Sz({\cal C})$.  

An explicit expression of the matrix Szeg\H o projection $\bs{\cal J}=\bs{\Sz}({\cal C})$ follows by inserting \eqref{eq:OLP*} into \eqref{eq:p-chi}, which yields
$$
\begin{aligned}
& p_n =  
\frac{\rho_{2n}\chi_{2n-1} + (1+a_{2n})\chi_{2n}}{\sqrt{2(1+a_{2n})}}
= \frac{1}{\sqrt{2}} (\sqrt{1-a_{2n}}\chi_{2n-1} + \sqrt{1+a_{2n}}\chi_{2n}),
\\
& \widetilde{p}_{n-1} = 
\frac{\rho_{2n}\chi_{2n-1} - (1-a_{2n})\chi_{2n}}{\sqrt{2(1-a_{2n})}}
= \frac{1}{\sqrt{2}} (-\sqrt{1+a_{2n}}\chi_{2n-1} + \sqrt{1-a_{2n}}\chi_{2n}),
\end{aligned}
$$   
or, in matrix form,
$$
\begin{pmatrix} 
 \widetilde{p}_{n-1} \\ p_n
\end{pmatrix}
= \frac{1}{\sqrt{2}} 
\begin{pmatrix}
 -\sqrt{1+a_{2n}} & \sqrt{1-a_{2n}}
 \\  
 \sqrt{1-a_{2n}} & \sqrt{1+a_{2n}}
\end{pmatrix}
\begin{pmatrix} 
 \chi_{2n-1} \\ \chi_{2n}
\end{pmatrix}.
$$
The above relation can be compactly written as 
\begin{equation} \label{eq:p-S}
\bs{p}
\begin{pmatrix}
 1 \\ 1
\end{pmatrix}
= \bs{S}\chi,
\kern15pt 
\bs{S} = 
\left(
\begin{smallmatrix}
 1 
 \\
 & S_2
 \\
 && S_4
 \\[-4pt]
 &&& \ddots
\end{smallmatrix}
\right),
\kern15pt
S_n =
\frac{1}{\sqrt{2}} 
\begin{pmatrix}
 -\sqrt{1+a_n} & \sqrt{1-a_n}
 \\  
 \sqrt{1-a_n} & \sqrt{1+a_n}
\end{pmatrix},
\end{equation}
in terms of a block diagonal symmetric unitary matrix $\bs{S}$. Combining \eqref{eq:p-S} with \eqref{eq:rec-pp} and \eqref{eq:C} we get
$$
\bs{\cal J} \bs{S} \chi 
= \bs{\cal J} \bs{p} \left(\begin{smallmatrix} 1 \\ 1 \end{smallmatrix}\right)
= (z+z^{-1}) \bs{p} \left(\begin{smallmatrix} 1 \\ 1 \end{smallmatrix}\right)
= (z+z^{-1}) \bs{S} \chi = \bs{S}({\cal C} + {\cal C}^+)\chi,
$$
which, bearing in mind that $\bs{S}^2=I$, leads to the following explicit expression for the matrix Szeg\H o projection of a CMV matrix $\cal C$,
\begin{equation} \label{eq:J-C}
\bs{\Sz}({\cal C}) = \bs{\cal J} = \bs{S}({\cal C}+{\cal C}^+)\bs{S}.
\end{equation}
That is, the matrix Szeg\H{o} projection of $\cal C$ stems not only from the evaluation of the Szeg\H{o} map $z+z^{-1}$ on $\cal C$, as one would naively expect, but also from a subsequent conjugation with the orthogonal matrix $\bs S$. We will refer to the matrix $\bs S$, given in \eqref{eq:p-S} in terms of the Schur parameters $a_n$ of $\cal C$, as the Szeg\H{o} rotation for the CMV matrix $\cal C$.

The Szeg\H o connection is restricted to symmetric measures on the unit circle, i.e. to real CMV matrices because symmetric measures are characterized by having real Schur parameters. Therefore, in order to relate Darboux transformations on the unit circle and the real line through the Szeg\H o connection we should consider a Darboux transformation ${\cal C}\overset{\ell}{\mapsto}{\cal D}$ between two real CMV matrices $\cal C$, $\cal D$. This means that $d\nu=\ell\,d\mu$ for certain symmetric measures $\mu$, $\nu$ on the unit circle related to $\cal C$, $\cal D$, respectively. If $\ell(z)=\alpha z+\beta+\overline\alpha z^{-1}$, the symmetry requirement for the measures implies that $\ell(z^{-1})=\ell(z)$, thus $\alpha\in\R$.  
Then, we have the following result which yields the announced Szeg\H{o} connection between Darboux for Jacobi and CMV.
\begin{thm} \label{thm:A-AA}
Let ${\cal C}\overset{\ell}{\mapsto}{\cal D}$ be a Darboux transformation between real CMV matrices $\cal C$, $\cal D$ with Laurent polynomial $\ell(z)=\alpha z+\beta+\overline\alpha z^{-1}$, $\alpha,\beta\in\R$. Then, the Szeg\H{o} projections ${\cal J}_e=\Sz({\cal C})$, ${\cal K}_e=\Sz({\cal D})$, $\bs{\cal J}=\bs{\Sz}({\cal C})={\cal J}_e\oplus{\cal J}_o$, $\bs{\cal K}=\bs{\Sz}({\cal D})={\cal K}_e\oplus{\cal K}_o$ provide Jacobi matrices related by the Darboux transformations ${\cal J}_e\overset{\wp}{\mapsto}{\cal K}_e$, ${\cal J}_o\overset{\wp}{\mapsto}{\cal K}_o$ with polynomial $\wp(x)=\alpha x+\beta$. The corresponding Cholesky factorizations
$$
\wp({\cal J}_k)=A_kA_k^+, 
\qquad\quad \wp({\cal K}_k)=A_k^+A_k,
\qquad\quad k=e,o,
$$
yield the following ones for the matriz Szeg\H{o} projections 
$$
\wp(\bs{\cal J})=\bs{AA^+}, 
\qquad\quad \wp(\bs{\cal K})=\bs{A^+A},
\qquad\quad 
\bs{A}=A_e\oplus A_o,
$$
where the direct sum $\bs{A}=A_e\oplus A_o$ is such that $A_e$ ($A_o$) acts on even (odd) indices, so that $\bs{A}$ has the structure
$$
\bs{A} =
\left(
\begin{smallmatrix}
 + &
 \\[3pt] 
 0 & \kern3pt + 
 \\[3pt] 
 \circledast & \kern3pt 0 & \kern2pt + 
 \\[3pt]
 & \kern3pt \circledast & \kern2pt 0 & +
 \\[-2pt] 
 && \kern2pt \scriptsize\ddots & \scriptsize\ddots & \scriptsize\ddots
\end{smallmatrix}
\right).
$$
The relation between the Cholesky factorizations for the CMV matrices,
$$
\ell({\cal C})=AA^+, 
\qquad\quad \ell({\cal D})=A^+A,
$$
and those of their Szeg\H{o} projections is given by
$$
A=\bs{S}\bs{A}\bs{T}, 
\qquad\quad 
\bs{A} = \bs{S}A\bs{T},
$$
where $\bs S$ and $\bs T$ are the Szeg\H{o} rotations for $\cal C$ and $\cal D$ respectively. 
\end{thm}

\begin{proof}
If $\mu$ is a measure for ${\cal C}$, the relation ${\cal C}\overset{\ell}{\mapsto}{\cal D}$ implies that $d\nu=\ell\,d\mu$ is associated with ${\cal D}$. Since $\wp(x)=\ell(z)$ under the Szeg\H{o} mapping $x=z+z^{-1}$, the Szeg\H{o} projections $\mu_e=\Sz(d\mu)$, $\nu_e=\Sz(d\nu)$ are related by $d\nu_e = \wp\,d\mu_e$. A similar relation is obviously valid for $d\mu_o(x)=(4-x^2)\,d\mu_e(x)$ and $d\nu_o(x)=(4-x^2)\,d\nu_e(x)$. This proves the Darboux relations ${\cal J}_k\overset{\wp}{\mapsto}{\cal K}_k$ for $k=e,o$. The Cholesky factorizations of $\bs{\cal J}$ and $\bs{\cal K}$ arise from regrouping those of ${\cal J}_k$ and ${\cal K}_k$ for $k=o,e$. Concerning the relation between the Darboux factors $A$ and $\bs A$, let us consider $\bs{p}=\bs{\Sz}(\chi)$, $\bs{q}=\bs{\Sz}(\omega)$, where $\chi$, $\omega$ are the orthonormal Laurent polynomials with respect to $\mu$, $\nu$ respectively. From Theorems~\ref{thm:Darboux} and \ref{thm:Darboux-J} we know that $\chi=A\omega$ and $\bs{p}=\bs{A}\bs{q}$. Combining these equalities with \eqref{eq:p-S} and the analogous relation 
$
\bs{q}
\left(
\begin{smallmatrix}
 1 \\ 1
\end{smallmatrix}
\right)
= \bs{T}\omega
$,
we get
$
\bs{S}A\omega = \bs{S}\chi 
= \bs{p} \left(\begin{smallmatrix} 1 \\ 1 \end{smallmatrix}\right)
= \bs{A} \bs{q} \left(\begin{smallmatrix} 1 \\ 1 \end{smallmatrix}\right)
= \bs{A} \bs{T} \omega,
$
so that $\bs{S}A=\bs{A}\bs{T}$. Since $\bs{S}^2=\bs{T}^2=I$ we can rewrite this identity as $A=\bs{S}\bs{A}\bs{T}$ or $\bs{A} = \bs{S}A\bs{T}$.
\end{proof}

The previous result is in agreement with \cite{GarHerMar2}, which shows that, in the case of symmetric measures, the linear spectral transformations associated with the Christoffel and Geronimus transformations on the unit circle are projected by the Szeg\H o mapping onto similar transformations on the real line. However, we have seen that there exist Darboux transformations of CMV matrices which do not correspond to Christoffel or Geronimus on the unit circle, namely, those associated with a Laurent polynomial $\ell(z)=\alpha z+\beta+\alpha z^{-1}$, $\alpha,\beta\in\R$, whose zeros lie on the unit circle, i.e. $|\beta|\le2|\alpha|$. In other words, there exist Christoffel and Geronimus transformations on the real line which are not the Szeg\H o projection of a Christoffel or Geronimus transformation on the unit circle, but correspond via the Szeg\H o mapping to general Laurent polynomial modifications of degree one on the unit circle.

The close relationship between Darboux for Jacobi and CMV is highlighted by the simple relation between the CMV Darboux factor $A$ and its matrix Szeg\H o projection $\bs{A}=\bs{\Sz}(A)$ given in Theorem~\ref{thm:A-AA}. This theorem also identifies $\bs{A}$ as the Darboux factor for the Darboux transformation $\bs{\cal J}\overset{\wp}{\mapsto}\bs{\cal K}$, which is consistent with the relation \eqref{eq:J-C} since it implies that 
$$
\wp(\bs{\cal J}) = \bs{S}\ell({\cal C})\bs{S} = \bs{S}AA^+\bs{S} 
= \bs{AA^+}, 
\qquad 
\wp(\bs{\cal K}) = \bs{T}\ell({\cal D})\bs{T} = \bs{T}A^+A\bs{T} 
= \bs{A^+A}.
$$  

We will illustrate the Szeg\H o connection between Darboux for Jacobi and CMV with an explicit example.

\begin{ex} \label{ex:Sz}
{\it The Szeg\H o connection and the Darboux transformation for 
$$
\ell(z)=(z+a)(z^{-1}+a)=a(z+z^{-1})+1+a^2, \qquad a\in(-1,1),
$$
and ${\cal C}$ with Schur parameters $a_1=a$ and $a_n=0$ for $n\ge2$, 
$$
{\cal C} = 
\left(
\begin{smallmatrix} 
	-a & \kern1pt \rho
	\\[3pt]
	0 & 0 & \kern4pt 0 & \kern4pt 1 
	\\[3pt]
    \rho & a & \kern4pt 0 & \kern4pt 0
    \\[3pt]
    & & \kern4pt 0 & \kern4pt 0 & \kern2pt 0 & 1
    \\[3pt]
    & & \kern4pt 1 & \kern4pt 0 & \kern2pt 0 & 0
    \\[3pt]
    & & & & \kern4pt 0 & \kern4pt 0 & \kern2pt 0 & 1
    \\[3pt]
    & & & & \kern4pt 1 & \kern4pt 0 & \kern2pt 0 & 0
    \\[3pt]
    & & & & \kern4pt & \kern4pt & \kern2pt \cdots & \cdots & \cdots
\end{smallmatrix}
\right),
\qquad
\rho=\sqrt{1-a^2}.
$$}
\end{ex}

It is known that $\cal C$ is associated with the measure 
$$
 d\mu(e^{i\theta}) = \frac{1}{|z+a|^2} \frac{d\theta}{2\pi} 
 = \frac{1}{\ell(z)} \frac{d\theta}{2\pi},
$$
whose orthonormal polynomials are the Bernstein-Szeg\H o ones
$$
\varphi_n(z) = 
\begin{cases}
\rho, & n=0,
\\
z^{n-1}(z+a), & n\ge1.
\end{cases}
$$
Therefore, the corresponding orthonormal zig-zag basis is given by
$$
\chi_{2n-1}(z) = z^{n-1}(z+a),
\qquad\quad
\chi_{2n}(z) = 
\begin{cases}
 \rho, & n=0,
 \\
 z^{-n}(1+az), & n\ge1.
\end{cases}
$$

The Darboux transformation ${\cal C}\overset{\ell}{\mapsto}{\cal D}$ is associated with the Christoffel transformation $\mu\overset{\ell}{\mapsto}\vartheta$, where $d\vartheta(e^{i\theta})=d\theta/2\pi$ is the Lebesgue measure. Hence, ${\cal D}={\cal S}$ is the shift matrix given in \eqref{eq:shift} corresponding to the Schur parameters $b_n=0$, and the related orthonormal zig-zag basis is 
$$
\omega_{2n-1}(z)=z^n, \qquad \omega_{2n}(z)=z^{-n}.
$$
The matrix $A$ of the change of basis
$$
\chi = A\omega,
\qquad
A = 
\left(
\begin{smallmatrix}
 \rho
 \\[3pt]
 a & \kern4pt 1
 \\[3pt]
 a & \kern4pt 0 & \kern2pt 1
 \\[3pt]
 & \kern4pt a & \kern2pt 0 & 1
 \\[-2pt]
 & \kern4pt & \kern2pt \scriptsize\ddots & \scriptsize\ddots  
 & \scriptsize\ddots 
\end{smallmatrix}
\right), 
$$
is the Darboux factor for the Cholesky fatorizations $\ell({\cal C})=AA^+$ and $\ell({\cal D})=A^+A$.

Concerning the Szeg\H o connection, in this case the matrices $\bs{S}$ and $\bs{T}$ coincide because $a_{2n}=b_{2n}=0$, thus
$$
S_n = T_n = 
\frac{1}{\sqrt{2}}
\begin{pmatrix}
 -1 & 1 \\ 1 & 1
\end{pmatrix}.
$$
The matrix Szeg\H o projection $\bs{A}=\bs{S}A\bs{T}$ of $A$  
$$
\bs{A} = \kern-2pt
\left(
\begin{smallmatrix}
 \rho
 \\[3pt]
 0 & 1
 \\[3pt]
 \sqrt{2}a & 0 & \kern5pt 1
 \\[3pt]
 & a & \kern5pt 0 & \kern2pt 1
 \\[3pt]
 & & \kern5pt a & \kern2pt 0 & 1
 \\[-1pt]
 & & \kern5pt & \kern2pt \scriptsize\ddots & \scriptsize\ddots 
 & \scriptsize\ddots 
\end{smallmatrix}
\right)
= A_e \oplus A_o,
\kern9pt
A_e = \kern-2pt
\left(
\begin{smallmatrix}
 \rho
 \\[3pt]
 \sqrt{2}a & 1
 \\[3pt]
 & a & \kern3pt 1
 \\[3pt]
 & & \kern3pt a & 1
 \\[-2pt]
 & & & \scriptsize\ddots & \scriptsize\ddots  
\end{smallmatrix}
\right),
\kern9pt
A_o = \kern-2pt 
\left(
\begin{smallmatrix}
 1
 \\[3pt]
 a & \kern3pt 1
 \\[3pt]
 & \kern3pt a & \kern3pt 1
 \\[3pt]
 & \kern3pt & \kern3pt a & 1
 \\[-2pt]
 & \kern3pt & & \scriptsize\ddots & \scriptsize\ddots  
\end{smallmatrix}
\right),
$$
provides the factorizations $\wp(\bs{\cal J})=\bs{AA^+}$ and $\wp(\bs{\cal K})=\bs{A^+A}$, where 
$$
\wp(x)=ax+(1+a^2),
$$ 
$\bs{\cal J}=\bs{S}({\cal C}+{\cal C}^+)\bs{S}$ is the block Jacobi matrix 
$$
\begin{gathered}
\bs{\cal J} =
\left(
\begin{smallmatrix}
 -2a & 0 & \sqrt{2}\rho
 \\[4pt]
 0 & -a & 0 & 1
 \\[4pt]
 \sqrt{2}\rho & 0 & a & 0 & \kern2pt 1
 \\[4pt]
 & 1 & 0 & 0 & \kern2pt 0 & \kern2pt 1
 \\[5pt]
 & & 1 & 0 & \kern2pt 0 & \kern2pt 0 & \kern2pt 1
 \\[-2pt]
 & & & \scriptsize\ddots & \kern2pt \scriptsize\ddots 
 & \kern2pt \scriptsize\ddots & \kern2pt \scriptsize\ddots 
 & \kern2pt \scriptsize\ddots
\end{smallmatrix}
\right)
= {\cal J}_e \oplus {\cal J}_o,
\\[3pt] 
{\cal J}_e = 
\left(
\begin{smallmatrix}
 -2a & \sqrt{2}\rho
 \\[4pt]
 \sqrt{2}\rho & a & 1
 \\[4pt]
 & 1 & 0 & \kern3pt 1
 \\[4pt]
 & & 1 & \kern3pt 0 & 1
 \\[-2pt]
 & & & \kern3pt \scriptsize\ddots & \scriptsize\ddots 
 & \scriptsize\ddots
\end{smallmatrix}
\right),
\qquad 
{\cal J}_o = 
\left(
\begin{smallmatrix}
 -a & \kern2pt 1
 \\[4pt]
 1 & \kern2pt 0 & \kern5pt 1
 \\[4pt]
 & \kern2pt 1 & \kern5pt 0 & \kern2pt 1
 \\[4pt]
 & & \kern5pt 1 & \kern2pt 0 & 1
 \\[-2pt]
 & & \kern5pt & \kern2pt \scriptsize\ddots & \scriptsize\ddots 
 & \scriptsize\ddots
\end{smallmatrix}
\right),
\end{gathered}
$$
and $\bs{\cal K}=\bs{T}({\cal S}+{\cal S}^+)\bs{T}$ is obtained by setting $a=0$ in $\bs{\cal J}$. The orthonormal polynomials related to 
$$
d\vartheta_e(x)=\frac{dx}{\pi\sqrt{4-x^2}}, 
\qquad\quad
d\vartheta_o(x)=\frac{1}{\pi}\sqrt{4-x^2}\,dx,
$$ 
are, up to normalization and rescaling of the variable, the first and second kind Chebyshev polynomials, respectively, 
$$
q^{(e)}_n(x) = 
\begin{cases}
 1, & n=0,
 \\
 \frac{1}{\sqrt{2}}(z^n+z^{-n}), & n\ge1,
\end{cases}
\qquad
q^{(o)}_{n-1}(x) = \frac{1}{\sqrt{2}} \frac{z^n-z^{-n}}{z-z^{-1}}, 
\qquad
x=z+z^{-1},
$$ 
while the orthonormal polynomials associated with $d\mu_k=d\vartheta_k/\wp$ 
can be obtained from the relations $p^{(k)}=A_kq^{(k)}$, i.e.
$$
 p^{(e)}_n = 
 \begin{cases}
 	\rho, & n=0,
	\\[2pt] 
	q^{(e)}_1+\sqrt{2}a, & n=1,
	\\[2pt] 
	q^{(e)}_n+aq^{(e)}_{n-1}, & n\ge2,
 \end{cases}
 \qquad\quad
 p^{(o)}_n = 
 \begin{cases}
 	1/\sqrt{2}, & n=0,
	\\[2pt] 
	q^{(o)}_n+aq^{(o)}_{n-1}, & n\ge1.
 \end{cases}
$$   

\subsection{The DVZ connection}
\label{ssec:DVZ}

Christoffel transformations, and thus Darboux transformations, play a central role in a new connection between orthogonal polynomials on the unit circle and the real line recently discovered by Derevyagin, Vinet and Zhedanov \cite{DVZ} (in short, the DVZ connection). In this section we will briefly describe the interrelations between the DVZ and Szeg\H o connections via Darboux transformations. For the details see \cite{DVZ} and also \cite{CMMV-DVZ}, where a different approach to the DVZ connection is presented. 

An important ingredient in these interrelations is the symmetrization process for orthogonal polynomials on the unit circle \cite{IsLi,MaSa}, which associates with any measure $\mu$ on the unit circle the only one $\widehat\mu$ which is invariant under $z \to -z$ and satisfies
$$
 \int f(z)\,d\mu(z) = \int f(z^2)\,d\widehat\mu(z), 
 \qquad \forall f\in\Lambda.
$$ 
We will express this relation as $\mu\xmapsto{z \mapsto z^2}\widehat\mu$.  
The corresponding orthonormal polynomials, $\varphi_n$ and $\widehat\varphi_n$, are related by 
\begin{equation} \label{eq:sym}
 \widehat\varphi_{2n}(z) = \varphi_n(z^2), 
 \qquad\quad 
 \widehat\varphi_{2n+1}(z) = z \varphi_n(z^2),
\end{equation}
so that the induced transformation of CMV matrices ${\cal C} \xmapsto{z \mapsto z^2}\widehat{\cal C}$ amounts to the transformation of Schur parameters
$
 (a_n)_{n\ge1} \xmapsto{z \mapsto z^2} (\widehat{a}_n)_{n\ge1}=(0,a_1,0,a_2,0,a_3,\dots).
$

Since $z \mapsto z^2$ reads as $x \mapsto x^2-2$ if $x=z+z^{-1}$, the combination of the Szeg\H o projection and the above symmetrization process maps the measure $\mu_e=\Sz(\mu)$ onto the only one $\widehat\mu_e=\Sz(\widehat\mu)$ on $[-2,2]$ which is symmetric under $x\to-x$ and satisfies
$$
 \int P(x)\,d\mu_e(x) = \int P(x^2-2)\,d\widehat\mu_e(x),
 \qquad \forall P\in\PP.
$$
We will write $\mu_e \xmapsto{x \mapsto x^2-2} \widehat\mu_e$. The orthonormal polynomials $\widehat{p}^{(e)}_n$ of $\widehat\mu_e$ are given by   
\begin{equation} \label{eq:hatp-pq}
 \widehat{p}^{(e)}_{2n}(x) = p^{(e)}_n(x^2-2),
 \qquad
 \widehat{p}^{(e)}_{2n+1}(x) = x \, q^{(e)}_n(x^2-2),
\end{equation}
where $p^{(e)}_n$ and $q^{(e)}_n$ are orthonormal with respect to $\mu_e$ and its Christoffel transform $(x+2)\,d\mu_e(x)$, respectively \cite{CMMV-DVZ}. The Jacobi matrices ${\cal J}_e$ and ${\cal K}_e$ of $p^{(e)}_n$ and $q^{(e)}_n$ are related by the Darboux transformation ${\cal J}_e \xmapsto{x+2} {\cal K}_e$. Since ${\cal J}_e=\Sz({\cal C})$, we conclude that ${\cal K}_e=\Sz({\cal D})$ with $\cal D$ the CMV matrix given by the Darboux transformation ${\cal C}\xmapsto{z+z^{-1}+2}{\cal D}$. The relations \eqref{eq:hatp-pq} imply that the Jacobi matrix $\widehat{\cal J}_e$ of $\widehat{p}^{(e)}_n$ satisfies
$$
 \widehat{\cal J}_e^2 - 2I = {\cal J}_e \oplus {\cal K}_e,
$$
where ${\cal J}_e$ and ${\cal K}_e$ act on even and odd indices respectively. We will express the above connections by writing ${\cal K}_e \xleftarrow{x+2} {\cal J}_e \xRightarrow{x \mapsto x^2-2} \widehat{\cal J}_e$, the double line in the right arrow indicating that both, ${\cal J}_e$ and ${\cal K}_e$, are involved in $\widehat{\cal J}_e$.

The DVZ connection starts from a known factorization of CMV matrices into block-diagonal unitary factors \cite{CMV2003},
$$
 {\cal C} = {\cal M}{\cal L}, 
 \quad 
 {\cal L} = 
 \left(
 \begin{smallmatrix}
 	\Theta_1 
 	\\[2pt]
 	& \Theta_3 
 	\\[2pt] 
 	& & \Theta_5 
 	\\[-3pt]
 	& & & \ddots
 \end{smallmatrix}
 \right),
 \quad
 {\cal M} = 
 \left(
 \begin{smallmatrix}
  	\\
 	a_0 
 	\\[2pt]
 	& \Theta_2 
 	\\[2pt] 
 	& & \Theta_4 
 	\\[-3pt]
 	& & & \ddots
 \end{smallmatrix}
 \right),
 \quad
 \Theta_n = 
 \begin{pmatrix}
 -a_n & \rho_n
 \\
 \rho_n & \overline{a}_n
 \end{pmatrix}. 
$$
Here $a_n$ are the Schur parameters of $\cal C$ and, as previously, we use the convention $a_0=1$. When these parameters are real, a Jacobi matrix ${\cal K}$ is defined by 
\begin{equation} \label{eq:L+M}
 {\cal K} = {\cal L}+{\cal M} =
 \left(
 \begin{smallmatrix}
 a_0-a_1 & \rho_1
 \\[3pt]
 \rho_1 & a_1-a_2 & \rho_2
 \\[3pt]
 & \rho_2 & a_2-a_3 & \kern9pt \rho_3
 \\[-1pt]
 & & \kern-15pt \ddots & \kern-12pt \ddots & \ddots
 \end{smallmatrix}
 \right).
\end{equation} 
This establishes the DVZ connection ${\cal C} \xmapsto{\DVZ} {\cal K}$. Derevyagin, Vinet and Zhedanov identify in \cite{DVZ} the measure and orthonormal polynomials of ${\cal K}$ in terms of those of ${\cal C}$. This identification appears surprisingly when combining a Christoffel transformation on the real line with a connection between measures on the unit circle and on the real line due to Delsarte and Genin \cite{DG}. A more direct approach to this problem can be found in \cite{CMMV-DVZ}, which also uncovers the role of the symmetrization process and the Szeg\H o mapping in the DVZ connection. If $\mu$ and $\varphi_n$ are the measure and orthonormal polynomials related to $\cal C$, then 
$$
 (x+2) \, d\widehat\mu_e(x),
 \qquad\qquad
 q_n(x) 
 = \frac{z^{-n}(\varphi_n^*(z^2)+z\varphi_n(z^2))}{\sqrt{2}(1+z)},
 \qquad
 x=z+z^{-1},
$$
are the measure and orthonormal polynomials associated with $\cal K$. Here, $\widehat\mu_e=\Sz(\widehat\mu)$ is the Szeg\H o projection of the symetrization $\widehat\mu$ of $\mu$. It turns out that ${\cal K}=\Sz(\widehat{\cal D})$ is the Szeg\H o projection of the CMV matrix $\widehat{\cal D}$ arising from the Darboux transformation $\widehat{\cal C}\xmapsto{z+z^{-1}+2}\widehat{\cal D}$, where $\widehat{\cal C}$ comes from the symmetrization ${\cal C} \xmapsto{z \mapsto z^2}\widehat{\cal C}$.

These results can be summarized in the following commutative diagram, where Darboux transformations play a prominent role. It shows explicitly that DVZ follows from a mix of symmetrization, Szeg\H o and Darboux. 

$$
 \xymatrix{
	{\cal D} \ar[d]|{\Sz} &&
 	{\cal C} \ar[ll]_{z+z^{-1}+2} \ar[rr]^{z \mapsto z^2} \ar[d]|{\Sz} 
	\ar@(r,r)[rrrrd]^{\DVZ} && 
	\widehat{\cal C} \ar[rr]^{z+z^{-1}+2} \ar[d]|{\Sz} &&
	\widehat{\cal D} \ar[d]|{\Sz}
	\\
	{\cal K}_e && 
	{\cal J}_e \ar[ll]_{x+2} \ar@{=>}[rr]^{x \mapsto x^2-2} && 
	\widehat{\cal J}_e \ar[rr]^{x+2} && 
	{\cal K} & 
	\kern-27pt = {\cal L}+{\cal M}   
 }  
$$
  
\section{Darboux versus QR}
\label{sec:D-QR}

Prior to Cholesky factorizations, other matrix factorizations have been related to Laurent polynomial modifications of measures on the unit circle. In the same work \cite{W1993} that introduces systematically for the first time the CMV matrices, D. Watkins links Christoffel transformations on the unit circle to QR factorizations of shifted CMV matrices. This can be seen as the unitary analogue of a similar connection between Jacobi matrices and polynomial modifications of measures on the real line first established by Kautsky and Golub \cite{KaGo} (see also \cite{BuMar2,BuIs,Gau2,W1993}). A version of this result for Hessenberg matrices related to orthogonal polynomials on the unit circle appears in \cite{DarHerMar,GarHerMar,HuVanB}. We will describe the QR interpretation of Christoffel transformations on the unit circle based on CMV matrices to compare with the Darboux transformations developed in the present paper. 

Given a measure $\mu$ on the unit circle, consider a general Christoffel transformation $\mu\overset{\ell}{\mapsto}\nu$, i.e.  
\begin{equation} \label{eq:Chris}
 \ell = q_*q, 
 \qquad q(z) = z-\zeta,
 \qquad |\zeta|<1.
\end{equation}
Let $\cal C$, $\cal D$ be the CMV matrices related to $\mu$, $\nu$, respectively, and $\chi$, $\omega$ the corresponding orthonormal zig-zag bases. We know that $\chi=A\omega$ with $A$ the 3-band lower triangular factor appearing in the Darboux transformation ${\cal C}\overset{\ell}{\mapsto}{\cal D}$.  

Since $d\nu(z) = |z-\zeta|^2\,d\mu(z)$, we have that $(1-|\zeta|)^2\mu \le \nu \le (1+|\zeta|)^2\mu$. Thus, the sets $L^2_\mu$ and $L^2_\nu$ of square integrable functions are identical as topological spaces, although they have different Hilbert structures given by the inner products
$$
 \<f,g\>_\mu = \int f\overline{g}\,d\mu,
 \kern50pt
 \<f,g\>_\nu = \int f\overline{g}\,d\nu.
$$ 
In consequence, the following mappings between Hilbert spaces define bounded linear operators with a bounded inverse,
$$
 \mathsf{A} \colon \mathop{L^2_\mu \to L^2_\nu} 
 \limits_{\ds f \xmapsto{\kern10pt} f}
 \kern30pt 
 \mathsf{U} \colon \mathop{L^2_\nu \to L^2_\mu} 
 \limits_{\ds f \xmapsto{\kern8pt} q_*f}
 \kern30pt
 \mathsf{UA} \colon \mathop{L^2_\mu \to L^2_\mu} 
 \limits_{\ds f \xmapsto{\kern8pt} q_*f}
 \kern30pt
 \mathsf{AU} \colon \mathop{L^2_\nu \to L^2_\nu} 
 \limits_{\ds f \xmapsto{\kern8pt} q_*f}
$$
Indeed, $\mathsf{U}$ is unitary because
$\<\mathsf{U}f,\mathsf{U}g\>_\mu = \<f,g\>_\nu$ for every $f,g\in L^2_\nu$ and $\ran(U)=L^2_\mu$.

To generate matrix representations of these operators, we will choose the zig-zag bases $\chi$ and $\omega$, respectively, as orthonormal bases for the Hilbert spaces $L^2_\mu$ and $L^2_\nu$ connected by such operators. Since $(\mathsf{A}\chi_0,\mathsf{A}\chi_1,\dots)^t=\chi=A\omega$, the corresponding matrix for the operator $\mathsf{A}$ is simply the Darboux factor $A$, so that $\mathsf{A}$ reads in coordinates as $X \mapsto Y=XA$. Here $X$ and $Y$ stand for row vectors of coordinates with respect to $\chi$ and $\omega$ respectively, i.e. $f=X\chi$ and $\mathsf{A}f=f=Y\omega$. Analogously, $q_*({\cal C})=({\cal C}-\zeta I)^+$ and $q_*({\cal D})=({\cal D}-\zeta I)^+$ are the matrices of the operators $\mathsf{UA}$ and $\mathsf{AU}$ respectively. 

Contrary to the previous discussions, in this context we do not need to worry about the associativity of matrix products because it is guaranteed by the associativity of compositions of operators in Hilbert spaces. Likewise, $A$ has a unique bounded inverse, that one representing the bounded operator $\mathsf{A}^{-1}$ in the bases $\omega$ and $\chi$. Hence, the linear operator $X \to XA$ has a trivial kernel in the space $\mathfrak{H}$ of square-summable sequences. This result is not in contradiction with Proposition~\ref{prop:kerB}, which refers to $\ker_R(B)=\ker_R(A^+)=\ker_L(A)^+$, since this formal kernel can have no non-trivial square-summable vectors. 

Let $U$ be the matrix of the unitary operator $\mathsf{U}$, i.e. the unitary matrix defined by the relation
$$
 q_*\omega = U\chi.
$$
Since $\mathsf{UA}$ and $\mathsf{AU}$ are obtained by composition of $\mathsf{A}$ and $\mathsf{U}$, their matrices are products of $A$ and $U$,
$$
\begin{aligned}
 & ({\cal C}-\zeta I)^+ \chi = q_*\chi = q_*A\omega = AU\omega 
 \;\Rightarrow\; ({\cal C}-\zeta I)^+ = AU,
 \\
 & ({\cal D}-\zeta I)^+ \omega = q_*\omega = U\chi = UA\omega 
 \;\Rightarrow\; ({\cal D}-\zeta I)^+ = UA.
\end{aligned} 
$$
Therefore, 
$$
 q({\cal C}) = {\cal C}-\zeta I = QR, 
 \qquad 
 q({\cal D}) = {\cal D}-\zeta I = RQ,
 \qquad 
 Q=U^+,
 \qquad
 R=A^+,
$$
with $Q$ unitary and $R$ upper triangular. In other words, $q({\cal C})$ and $q({\cal D})$ are connected by a QR factorization and a reversed one, much in the same way as $\ell({\cal C})=AA^+$ and $\ell({\cal D})=A^+A$ are connected by Cholesky factorizations. Moreover, these two kinds of factorizations are closely related because $R$ is the adjoint of the Darboux factor $A$. Actually, the Cholesky factorizations are a direct consequence of the QR ones because, due to the unitarity of $Q$, 
$$
\begin{aligned}
 & q({\cal C}) = QR \;\Rightarrow\; 
 \ell({\cal C}) = q_*({\cal C}) q({\cal C}) = R^+R,
 \\
 & q({\cal D}) = RQ \;\Rightarrow\; 
 \ell({\cal C}) = q({\cal D}) q_*({\cal D}) = RR^+.
\end{aligned}
$$

Like the upper triangular factor $R$, which is 3-band, the unitary factor $Q$ has a remarkable structure. Bearing in mind that $R=A^+$ has a bounded inverse, we find that $Q=({\cal C}-\zeta I)R^{-1}$. This implies that $Q$ has an upper CMV structure, i.e. it is upper Hessenberg type with the shape
\begin{equation} \label{eq:Q}
 Q = \left(\begin{smallmatrix}
      *&*&*&*&*&*&\cdots\\[2pt]
      *&*&*&*&*&*&\cdots\\[2pt]
      *&*&*&*&*&*&\cdots\\[2pt]
       & &*&*&*&*&\cdots\\[2pt]
       & &*&*&*&*&\cdots\\[2pt]
       & & & &*&*&\cdots\\[2pt]
       & & & &*&*&\cdots\\[2pt]
       \kern9pt&\kern9pt&\kern9pt&\kern9pt&\kern9pt&\kern9pt&\cdots
      \end{smallmatrix}\right).
\end{equation}

So far, we have described a QR interpretation of Christoffel transformations on the unit circle based on the use of CMV matrices. This result was first obtained by D.Watkins in \cite{W1993} using a different approach. Indeed, he proves that the transformations $d\mu(x) \to |x-\beta|^2\,d\mu(x)$, $\beta\in\R$, on the real line have a similar QR interpretation in terms of Jacobi matrices (see also \cite{BuMar2,BuIs,KaGo}). Therefore, QR suggests that the unit circle version of such transformations on the real line should be the Christoffel transformations on the unit circle. On the contrary, the Darboux transformations developed in the present paper highlight the role of the Christoffel transformations on the unit circle as the counterpart of the Christoffel transformations $d\mu(x) \to (x-\beta)\,d\mu(x)$ on the real line.  

As in the QR case, the Darboux transformation associated with a Christoffel transformation $\mu\overset{\ell}{\mapsto}\nu$ can be interpreted in terms of factorizations of operators in Hilbert spaces. To show this, let us introduce the following bounded linear operators with a bounded inverse
$$ 
 \mathsf{B} \colon \mathop{L^2_\nu \to L^2_\mu} 
 \limits_{\ds f \xmapsto{\kern8pt} \ell f}
 \kern30pt
 \mathsf{BA} \colon \mathop{L^2_\mu \to L^2_\mu} 
 \limits_{\ds f \xmapsto{\kern8pt} \ell f}
 \kern30pt
 \mathsf{AB} \colon \mathop{L^2_\nu \to L^2_\nu} 
 \limits_{\ds f \xmapsto{\kern8pt} \ell f}
$$
The equality $\<\mathsf{B}f,g\>_\mu = \<f,\mathsf{A}g\>_\nu$ for every $f\in L^2_\nu$ and $g\in L^2_\mu$ identifies $\mathsf{B=A^+}$ as the adjoint operator of $\mathsf{A}$. Therefore, $A^+$ is the matrix of $\mathsf{B}$ with respect to the bases $\chi$ of $L^2_\mu$ and $\omega$ of $L^2_\nu$, so that we recover the identity $\ell\omega=(\mathsf{B}\omega_0,\mathsf{B}\omega_1,\dots)^t=A^+\chi$. The corresponding matrices of $\mathsf{BA}$ and $\mathsf{AB}$ are $\ell({\cal C})$ and $\ell({\cal D})$, respectively. The fact that these two last operators are compositions of $\mathsf{A}$ and $\mathsf{B}$ translates into the Cholesky factorizations related to the Darboux transformation ${\cal C}\overset{\ell}{\mapsto}{\cal D}$. The relation between the QR and Cholesky factorizations is simply the matrix realization of the operator identities $\mathsf{BA=A^+U^+UA}$ and $\mathsf{AB=AUU^+A^+}$, which use the fact that $\mathsf{U^+=U^{-1}}$ due to the unitarity of $\mathsf{U}$. 

From the operator point of view, the main difference between QR and Darboux is that QR looks for the unitarity of one of the factors, $\mathsf{U}$, requiring no symmetry for the factorized operators, $\mathsf{UA}$ and $\mathsf{AU}$, while Darboux needs the self-adjointness of the factorized operators, $\mathsf{BA}$ and $\mathsf{AB}$. 

We have seen that the QR interpretation of Christoffel transformations on the unit circle yields also the Darboux transformations of the corresponding CMV matrices. Nevertheless, QR does not provide all the richness given by Darboux, which allows more general transformations of measures $\mu\overset{\ell}{\mapsto}\nu$, not necessarily with the form \eqref{eq:Chris} corresponding to $\ell$ with zeros outside of the unit circle. In particular, QR can never deal with the case of $\ell$ having two different zeros on the unit circle, a situation covered by Darboux as long as $\ell$ is non-negative on the support of $\mu$. As for the case of a double zero on the unit circle, i.e. \eqref{eq:Chris} with $|\zeta|=1$, it can always be tackled by Darboux but not by QR. When the double zero does not lie in the support of $\mu$ the situation is as good as for Christoffel transformations concerning QR. Otherwise the operator $\mathsf{U}$ is certainly isometric but not necessarily unitary because $\ran(U)=L^2_\mu$ iff $1/q\in L^2_\nu$. 

Let us clarify from the operator point of view the situation which distinguishes Darboux from QR, i.e. the case of $\ell$ having its zeros on the unit circle. Suppose $\ell$ non-negative on the support of $\mu$, so that $\ell({\cal C})$ is positive definite and Darboux can be implemented. Then, $d\nu=\ell\,d\mu$ is such that any function in $L^2_\mu$ lies also in $L^2_\nu$, but the converse is not necessarily true. On the other hand, 
the supports of $\mu$ and $\nu$ may differ only in the zeros of $\ell$, where $\mu$ may have mass points but $\nu$ not. Hence, $\ell L^2_\nu$ makes sense as a subset of $L^2_\mu$. All this ensures that $\mathsf{A}$ and $\mathsf{B}$ are well defined bounded operators which are adjoint of each other, but now the first of these factors may have a non-trivial kernel. It is not difficult to see that $\ker(\mathsf{B})=\{0\}$, while $\ker(\mathsf{A})$ is spanned by the characteristic functions of the zeros of $\ell$, which is non-trivial iff $\mu$ has a mass point at some of these zeros. From general results (see for instance \cite[Theorem 5.39]{We1980}) we know that $\mathsf{BA=A^+A}$ and $\mathsf{AB=AA^+}$ are unitarily equivalent when restricted to the orthogonal complement of $\ker(\mathsf{A^+A})=\ker(\mathsf{A})$ and $\ker(\mathsf{AA^+})=\ker(\mathsf{A^+})=\{0\}$, respectively. We conclude that $\mathsf{BA}$ and $\mathsf{AB}$ are isospectral unless $\mu$ has a mass point at a zero of $\ell$, in which case the spectrum of $\mathsf{BA}$ is obtained by adding the null eigenvalue to the spectrum of $\mathsf{AB}$. Bearing in mind that, using the notation \eqref{eq:Umu}, $\mathsf{BA}=\ell(\mathsf{U}_\mu)$ and $\mathsf{AB}=\ell(\mathsf{U}_\nu)$, this agrees with the fact that the spectrum of $\mathsf{U}_\mu$ differs from that of $\mathsf{U}_\nu$ only in the mass points of $\mu$ located at the zeros of $\ell$. Note that $\chi(z)^+$ are the coordinates of the characteristic function of $\{z\}$ with respect to the orthonormal basis $\chi$ of $L^2_\mu$. Therefore, the previous comments show that the square-summable vectors of the formal kernel $\ker_R(B)=\ker_R(A^+)=\ker_L(A)^+$ are spanned by the evaluations of $\chi$ at the zeros of $\ell$.


\section{Darboux, Ablowitz-Ladik and Schur flows}
\label{sec:IS}

It is well known that the Lax pair of the Toda lattice is built out of Jacobi matrices. This is the origin of a fruitful three-way relationship among integrable Toda chains, Darboux transformations of Jacobi matrices and orthogonal polynomials on the real line. On the other hand, the unitary analogue of the Toda lattice is given by the Schur flows, a name which refers to an integrable system closely related to the orthogonal polynomials on the unit circle. These polynomials also have a close link to another integrable system of interest in mathematical physics, the Ablowitz-Ladik model. However, no Darboux tool was available yet for the analysis of Schur flows or the Ablowitz-Ladik system. In this section we will show that such a tool is provided by Darboux transformations of CMV matrices, which play for Schur and Ablowitz-Ladik flows a similar role to that of Darboux for Jacobi in Toda chains. More precisely, we will see that Darboux for CMV constitutes an integrable discretization of such integrable flows, and also provides a mechanism to generate new flows from known ones. In order to highlight the analogies between Darboux for Jacobi and CMV regarding integrable systems, we will first summarize known connections between Darboux for Jacobi and the Toda lattice. It is worth stressing that this section illustrates eloquently the benefits of the interpretation of Darboux as a transformation of measures, a central result in this paper which now finds a very useful goal: the simplification of the arguments for statements related to integrable systems.     

Using Flaschka's variables, the evolution of a semi-infinite Toda lattice chain can be rewritten equivalently as a flow ${\cal J}(t)$ of Jacobi matrices governed by the Lax equations \cite{Teschl,Toda}
\begin{equation} \label{eq:Lax-T}
 \frac{d{\cal J}}{dt}=[\pi({\cal J}),{\cal J}] 
 = \pi({\cal J})\,{\cal J}-{\cal J}\pi({\cal J}),
 \qquad\quad
 \pi({\cal J}) = P_+{\cal J}-P_-{\cal J},
\end{equation}
where $P_\pm$ are the upper/lower triangular projectors 
$$
 P_+ 
 \left(\begin{smallmatrix} 
 	* & * & * & \cdots 
	\\[3pt] 
	* & * & * & \cdots 
	\\[3pt] 
	* & * & * & \cdots 
	\\[3pt] 
	\cdots & \cdots & \cdots & \cdots
 \end{smallmatrix}\right)
 = 
 \left(\begin{smallmatrix} 
 	0 & * & * & \cdots 
	\\[3pt] 
	0 & 0 & * & \cdots 
	\\[3pt] 
	0 & 0 & 0 & \cdots 
	\\[3pt] 
	\cdots & \cdots & \cdots & \cdots
 \end{smallmatrix}\right),
 \qquad\quad
 P_- 
 \left(\begin{smallmatrix} 
 	* & * & * & \cdots 
	\\[3pt] 
	* & * & * & \cdots 
	\\[3pt] 
	* & * & * & \cdots 
	\\[3pt] 
	\cdots & \cdots & \cdots & \cdots
 \end{smallmatrix}\right)
 = 
 \left(\begin{smallmatrix} 
 	0 & 0 & 0 & \cdots 
	\\[3pt] 
	* & 0 & 0 & \cdots 
	\\[3pt] 
	* & * & 0 & \cdots 
	\\[3pt] 
	\cdots & \cdots & \cdots & \cdots
 \end{smallmatrix}\right).
$$
The integrability, which relies on the isospectrality of the flow generated by a Lax pair, is also obvious when this flow is translated to the orthogonality measure $d\mu(x,t)$ associated with ${\cal J}(t)$, namely \cite{KaMo}
\begin{equation} \label{eq:mu-T}
 d\mu(x,t) = e^{-xt} \, d\mu(x,0).
\end{equation}
Actually, the relations \eqref{eq:Lax-T} and \eqref{eq:mu-T} are completely equivalent \cite{NiSo}. 

On the other hand, the iteration of a Darboux transformation with polynomial $\wp(x)=x+\beta$ starting with a given Jacobi matrix ${\cal J}_0$, 
$$
 {\cal J}_0 \overset{\wp}{\mapsto} {\cal J}_1 \overset{\wp}{\mapsto} 
 {\cal J}_2 \overset{\wp}{\mapsto} \cdots,
$$
leads to a sequence of Jacobi matrices ${\cal J}_n$. In view of \eqref{eq:uv}, this Darboux iteration reads as the difference equations
\begin{equation} \label{eq:D-Toda}
\begin{aligned}
 & u_{k,n+1} + v_{k-1,n+1} = u_{k,n} + v_{k,n}, 
 \\ 
 & u_{k,n+1} v_{k,n+1} = u_{k+1,n} v_{k,n},
\end{aligned}
\end{equation}
for the LU variables $u_{k,n}$, $v_{k,n}$ of $\widetilde{\cal J}_n+\beta I$, where we use the natural adaptation of the notation in \eqref{eq:LU} and \eqref{eq:uv}. If we think of ${\cal J}_n={\cal J}(n\delta)$ as a time discretization of a differentiable flow ${\cal J}(t)$ with time-step $\delta$, the difference equations \eqref{eq:D-Toda} become the Toda lattice equations \eqref{eq:Lax-T} in the continuous limit $\delta \to 0$ \cite{HTI,SpZh}. A quick argument for this assertion follows from the identity
$$
 e^{-xt} \, d\mu(x,0) = \lim_{n\to\infty} \, (1-\delta x)^n \, d\mu(x,0), 
 \qquad \delta=\frac{t}{n},
$$
which identifies the Toda flow as a limit of iterated Darboux transformations. In other words, the Darboux transformations of Jacobi matrices constitute an integrable time discretization of the Toda lattice. 

Such an integrable discretization is useful to integrate the dynamical system numerically with higher accuracy. Nevertheless, this is not the only interest of Darboux transformations for the Toda lattice. For instance, Darboux transformations can be used to generate new solutions of the Toda lattice starting from a known one. More precisely, if ${\cal J}(t)$ solves \eqref{eq:Lax-T}, for any polynomial $\wp$ of degree one, the Darboux transformation ${\cal J}(t) \overset{\wp}{\mapsto} {\cal K}(t)$ yields a new solution ${\cal K}(t)$ of \eqref{eq:Lax-T}. The simplest proof of this relies on the interpretation of Toda and Darboux in terms of the measures $d\mu(x,t)$, $d\nu(x,t)$ associated with ${\cal J}(t)$, ${\cal K}(t)$. Since a Toda flow ${\cal J}(t)$ is equivalent to $d\mu(x,t) = e^{-xt} \, d\mu(x,0)$ and Darboux reads as $d\nu(x,t)=\wp(x)\,d\mu(x,t)$, we get $d\nu(x,t) = e^{-xt} \, d\nu(x,0)$, so ${\cal K}(t)$ is a Toda flow too. Therefore, the iterative Darboux transformations of a Toda solution ${\cal J}(t)$ given by any sequence $\wp_n$ of polynomials of degree one leads to a hierarchy of Toda solutions
$$
 {\cal J}(t) \overset{\wp_1}{\mapsto} {\cal J}_1(t) \overset{\wp_2}{\mapsto} 
 {\cal J}_2(t) \overset{\wp_3}{\mapsto} \cdots.
$$

Inverse Darboux transformations can be also used to generate such hierarchies of Toda solutions, although this needs a careful choice of the free parameter appearing in inverse Darboux for Jacobi. If ${\cal K}(t)$ follows from an inverse Darboux transformation ${\cal K}(t) \xmapsto{x-\zeta} {\cal J}(t)$ of a Toda flow ${\cal J}(t)$, the measures $d\mu(x,t)$, $d\nu(x,t)$ associated with ${\cal J}(t)$, ${\cal K}(t)$ are related by 
$$
 d\nu(x,t) = \frac{d\mu(x,t)}{x-\zeta} + m(t) \, \delta_\zeta.
$$
Then, ${\cal K}(t)$ is another Toda flow provided that the mass has the time dependence $m(t)=m(0)\,e^{-\zeta t}$, so that $d\mu(x,t)=e^{-xt}\,d\mu(x,0)$ implies $d\nu(x,t)=e^{-xt}\,d\nu(x,0)$. On the other hand, the freedom in the mass at $\zeta$ is equivalent to the freedom in the $(0,0)$ entry of the triangular factor $A$ which determines the reversed factorization ${\cal J}-\zeta I=A^+A$. Denoting this entry by $r_0$ as in the CMV case, from the relation $q=Ap$ between the orthonormal polynomials $p$, $q$ related to ${\cal J}$, ${\cal K}$, we find that
$r_0^2 = q_0^2/p_0^2 = \int d\mu/\int d\nu$. Thus, the time dependence of the mass which preserves the Toda solutions is equivalent to a time dependent choice of the free parameter $r_0(t)$ given by
\begin{equation} \label{eq:r0-T}
 r_0^2(t) = 
 \frac{\int e^{(\zeta-x)t} \, d\mu(x,0)}
 {\int e^{(\zeta-x)t} \, \frac{d\mu(x,0)}{x-\zeta} + m(0)}.
\end{equation}
The arbitrariness of $m(0)$ is equivalent to the arbitrariness of $r_0(0)$, which therefore determines $r_0(t)$ at any time $t$. This means that, once $\wp(x)=x-\zeta$ is chosen, the inverse Darboux transformations of ${\cal J}(t)$ which preserve Toda solutions constitute a one-parameter family parametrized by $m(0)$, equivalently by the entry $r_0(0)$ of $A(0)$ which determines the factorization $\wp({\cal J}(0))=A(0)^+A(0)$. The factorizations $\wp({\cal J}(t))=A(t)^+A(t)$ giving at other times the Toda preserving inverse Darboux transformations are determined by \eqref{eq:r0-T}. 

Therefore, inverse Darboux transformations are more efficient than direct ones in generating new Toda solutions since, for each choice of the polynomial driving the transformation, they yield a one-parameter family of Toda solutions instead of a single one. Hence, given a Toda flow ${\cal J}(t)$, the iterative inverse Darboux transformations defined by a sequence $\wp_n$ of polynomials of degree one leads to a hierarchy of families of Toda flows ${\cal J}_n(m_1,\dots,m_n;t)$ depending on an increasing number of parameters, which can be interpreted as a mass $m_n=m_n(0)$ at the zero of $\wp_n$ at initial time, 
$$
 \cdots \overset{\wp_3}{\mapsto} {\cal J}_2(m_1,m_2;t) 
 \overset{\wp_2}{\mapsto} {\cal J}_1(m_1;t) \overset{\wp_1}{\mapsto} 
 {\cal J}(t).  
$$
 
The above connections between Darboux for Jacobi and the Toda lattice have unitary analogues that have been ignored so far. They involve Darboux for CMV and known unitary counterparts of Toda, the so called Schur and Ablowitz-Ladik flows. Schur flows are the unit circle version of the Toda lattice, connected to Toda via the Szeg\H{o} mapping, while the Ablowitz-Ladik model follows from the space discretization of the nonlinear Schr\"odinger equation \cite{AFMa,AG94,FaGe,GeNe,Gol,KN2007,LN12,MuNa,Nenciu05,Nenciu06,Si2007bis}. Both integrable systems define a dynamics in the space of ``Schur sequences" (sequences in the open unit disk) via the difference-differential equations
$$
\begin{gathered}
 \begin{gathered}
 	\text{\underline{Schur flows}} 
 	\\
 	\frac{da_n}{dt} = \rho_n^2 (a_{n+1}-a_{n-1}),
 \end{gathered}
 \hskip50pt
 \begin{gathered}
 	\text{\underline{Ablowitz-Ladik}} 
 	\\
 	-i\frac{da_n}{dt} = \rho_n^2 (a_{n+1}+a_{n-1}),
 \end{gathered}
 \\[3pt]
 \quad
 |a_n|<1, 
 \qquad
 \rho_n=\sqrt{1-|a_n|^2},
 \qquad
 n\ge1,
 \qquad
 a_0=1.
\end{gathered}
$$
L.~Golinskii discovered in 2006 \cite{Gol} that the Schur flows are equivalent to flows ${\cal C}(t)$ of CMV matrices with Schur parameters $a_n(t)$ governed by the Lax equations
\begin{equation} \label{eq:Lax-S}
 \frac{d{\cal C}}{dt}=[\pi(\re\,{\cal C}),{\cal C}],
 \qquad
 \re\,{\cal C} = \frac{1}{2}({\cal C}+{\cal C}^+),
\end{equation}
where $\pi(\cdot)$ is defined as in \eqref{eq:Lax-T}. In \cite{Gol}, this author also proved the equivalence of the Schur flows and the evolution 
\begin{equation} \label{eq:mu-S}
 d\mu(z,t) = e^{(z+z^{-1})t} \, d\mu(z,0)
\end{equation}
for the measure $d\mu(z,t)$ related to ${\cal C}(t)$ (see \cite{IsWi} for an analysis of the flow \eqref{eq:mu-S} starting from the Lebesgue measure). These results constitute the unitary analogue of the equivalence among the semi-infinite Toda lattice, \eqref{eq:Lax-T} and \eqref{eq:mu-T}.

The equivalence among Schur flows, \eqref{eq:Lax-S} and \eqref{eq:mu-S} can be easily extended to slightly more general flows including Ablowitz-Ladik, which allows for a unified treatment of both integrable systems. The transformation $a_n \to \lambda^n a_n$, $|\lambda|=1$, exchanges Schur flows and what we will call Schur flows with parameter $\lambda$, given by
\begin{equation} \label{eq:gS}
 \frac{da_n}{dt} = \rho_n^2 (\lambda a_{n+1} - \overline\lambda a_{n-1}).
\end{equation}
The Ablowitz-Ladik flows become the Schur flows with parameter $\lambda=i$. Since the transformation $a_n \to \lambda^n a_n$ is equivalent to the rotation $d\mu(z,t) \to d\mu(\overline\lambda z,t)$ of the measure, from \eqref{eq:mu-S} we find that \eqref{eq:gS} reads as
\begin{equation} \label{eq:mu-gS}
 d\mu(z,t) = e^{(\lambda z + \overline\lambda z^{-1})t} \, d\mu(z,0).
\end{equation}
Besides, from the explicit Schur parametrization \eqref{eq:CMV-schur} of CMV matrices we can see that $a_n \to \lambda^n a_n$ is equivalent to ${\cal C}\to\lambda\Lambda{\cal C}\Lambda^+$ with $\Lambda$ the diagonal unitary matrix
$$
 \Lambda = 
 \left(\begin{smallmatrix}
    \\
 	1 
	\\ 
	& \kern3pt \lambda 
	\\ 
	& & \kern3pt \overline\lambda 
	\\ 
	& & & \kern3pt \lambda^2 
	\\ 
	& & & & \kern1pt \overline\lambda^2
	\\[-3pt]
	& & & & & \kern-3pt \ddots
 \end{smallmatrix}\right).
$$ 
Hence, using \eqref{eq:Lax-S} we conclude that \eqref{eq:gS} can be rewritten as the Lax equation
\begin{equation} \label{eq:Lax-gS}
 \ds \frac{d{\cal C}}{dt}=[\pi(\re(\lambda{\cal C})),{\cal C}].
\end{equation}
Actually, the equivalence among \eqref{eq:gS}, \eqref{eq:mu-gS} and \eqref{eq:Lax-gS} is a particular case of a much more general one concerning what is known as generalized Schur flows \cite{Si2007bis}. Nevertheless, Schur flows with parameter $\lambda$ (in short, Schur flows) are enough to understand the role of Darboux for CMV in integrable systems.

The exponential character of the evolution for the measure under Schur flows allows us to search for an integrable discretization of such dynamical systems using arguments which parallel those of the Toda case. Rewriting
$$
 e^{(\lambda z+\overline\lambda z^{-1})t} \, d\mu(z,0) 
 = \lim_{n\to\infty} \, 
 (1+\delta\lambda z+\delta\overline\lambda z^{-1})^n \, d\mu(z,0), 
 \qquad \delta=\frac{t}{n},
$$
shows that Darboux for CMV can be understood as an integrable discretization of the Schur flows \cite{MuNa}. This means that \eqref{eq:gS} should be obtained as a continuous limit of the relation \eqref{eq:D=dS} between the Schur parameters of CMV matrices connected by Darboux. To check this consider the iteration of a Darboux transformation with Laurent polynomial $\ell(z)=1+\alpha z+\overline\alpha z^{-1}$ starting with a CMV matrix ${\cal C}_0$,
$$
 {\cal C}_0 \overset{\ell}{\mapsto} {\cal C}_1 \overset{\ell}{\mapsto} 
 {\cal C}_2 \overset{\ell}{\mapsto} \cdots,
$$
and suppose that ${\cal C}_n={\cal C}(n\delta)$ come from a time discretization of a differentiable CMV flow ${\cal C}(t)$ with time-step $\delta=|\alpha|$. Then, \eqref{eq:D=dS} yields the following relations among the Schur parameters $a_{k,n}$ of ${\cal C}_n$,
\begin{equation} \label{eq:dS}
\begin{aligned}
 a_{k,n+1}-a_{k,n} & = 
 \frac{\kappa_{k-1,n+1}^2}{\kappa_{k,n}^2} 
 (\alpha a_{k+1,n} - \overline\alpha a_{k-1,n+1})
 \\
 & = \frac{1}{1-2\re(\alpha a_{1,n})}
 \frac{\rho_{1,n}^2\rho_{2,n}^2\cdots\rho_{k,n}^2}
 {\rho_{1,n+1}^2\rho_{2,n+1}^2\cdots\rho_{k-1,n+1}^2}
 (\alpha a_{k+1,n} - \overline\alpha a_{k-1,n+1}),
\end{aligned}
\end{equation}
where $\rho_{k,n}=\sqrt{1-|a_{k,n}|^2}$, $\kappa_{k,n}^{-1}=\rho_{k,n}\cdots\rho_{k,1}\sqrt{\int d\mu_n}$ and the measures $\mu_n$ associated with each ${\cal C}_n$ are related by $d\mu_{n+1}=\ell d\mu_n$ so that $\int d\mu_{n+1} = (1-2\re(\alpha a_{1,n})) \int d\mu_n$. The difference equations \eqref{eq:dS} define explicitly an integrable discretization of the Schur flows \cite{MuNa}. Indeed, expressing $\alpha=\delta\lambda$, $|\lambda|=1$, we get
$$
 \frac{a_{k,n+1}-a_{k,n}}{\delta} 
 = \frac{1}{1-2\delta\re(\lambda a_{1,n})}
 \frac{\rho_{1,n}^2\rho_{2,n}^2\cdots\rho_{k,n}^2}
 {\rho_{1,n+1}^2\rho_{2,n+1}^2\cdots\rho_{k-1,n+1}^2}
 (\lambda a_{k+1,n} - \overline\lambda a_{k-1,n+1}),
$$
which in the limit $\delta\to0$ obviously gives a Schur flow with parameter $\lambda$.

It is also possible to generate hierarchies of Schur flows starting from a given one ${\cal C}(t)$ by iterating Darboux transformations of CMV matrices,
$$
 {\cal C}(t) \overset{\ell_1}{\mapsto} {\cal C}_1(t) \overset{\ell_2}{\mapsto} 
 {\cal C}_2(t) \overset{\ell_3}{\mapsto} \cdots.
$$
Analogously to the Toda case, this is a simple consequence of the characterization of Schur flows and Darboux for CMV in terms of measures. Any Darboux transformation ${\cal C}(t) \overset{\ell}{\mapsto} {\cal D}(t)$ translates into the relation $d\nu(z,t)=\ell(z)\,d\mu(z,t)$ for the measures $d\mu(z,t)$, $d\nu(z,t)$ of ${\cal C}(t)$, ${\cal D}(t)$, thus $d\mu(z,t) = e^{(\lambda z+\overline\lambda z^{-1})t} \, d\mu(z,0)$ implies $d\nu(z,t) = e^{(\lambda z+\overline\lambda z^{-1})t} \, d\nu(z,0)$. 

Inverse Darboux transformations can improve the generation of new Schur flows because each iteration of inverse Darboux provides in general a parametric family of such flows under a suitable choice of the freedom in the reversed Cholesky factorizations. Consider for instance the case $\ell(z)=(z-\zeta)(z^{-1}-\overline\zeta)$ with $|\zeta|=1$. In terms of the measures $d\mu(z,t)$, $d\nu(z,t)$ of ${\cal C}(t)$, ${\cal D}(t)$, an inverse Darboux transformation ${\cal D}(t) \overset{\ell}{\mapsto} {\cal C}(t)$ reads as 
$$
 d\nu(z,t) = \frac{d\mu(z,t)}{|z-\zeta|^2} + m(t)\,\delta_\zeta.
$$
Bearing in mind the characterization \eqref{eq:mu-gS} of Schur flows, this transformation maps a Schur flow with parameter $\lambda$ into another one exactly when $m(t)=m(0)e^{2\re(\lambda\zeta)t}$. On the other hand, Theorem~\ref{thm:spurious} shows that the freedom in the reversed factorizations $\ell({\cal C})=A^+A$ leading to CMV solutions of inverse Darboux is encoded in the first Schur parameter $b_1$ of the inverse transform ${\cal D}$. Since $b_1=-\int z\,d\nu/\int d\nu$, we find that the inverse Darboux transformations which preserve Schur flows with parameter $\lambda$ are characterized by a time dependent choice  
\begin{equation} \label{eq:r0-S}
 b_1(t) = 
 - \frac{\int z \, e^{2\re(\lambda(z-\zeta))t} \, \frac{d\mu(z,0)}{|z-\zeta|^2}
 + m(0)\zeta}
 {\int e^{2\re(\lambda(z-\zeta))t} \, \frac{d\mu(z,0)}{|z-\zeta|^2} + m(0)},
\end{equation}
determined by its initial value $b_1(0)$ via the arbitrary parameter $m(0)$. Hence, this mass parametrizes the inverse Darboux transformations preserving Schur flows. 

Thus, given a sequence $\ell_n$ of Hermitian Laurent polynomials of degree one with a double zero, the iterative inverse Darboux transformations of a Schur flow ${\cal C}(t)$ yield a hierarchy of families of Schur flows ${\cal C}_n(m_1,\dots,m_n;t)$ parametrized by the initial mass $m_n=m_n(0)$ at the zero of each $\ell_n$, 
$$
 \cdots \overset{\ell_3}{\mapsto} {\cal C}_2(m_1,m_2;t) 
 \overset{\ell_2}{\mapsto} {\cal C}_1(m_1;t) \overset{\ell_1}{\mapsto} 
 {\cal C}(t).  
$$
In case of $\ell_n$ with two distinct zeros on the unit circle, the number of parameters increases by two at the $n$-th iteration of inverse Darboux, namely, the initial masses at both zeros of $\ell_n$.

The above ideas suggest that exploiting the links between Darboux for CMV and Schur flows should be as fruitful as for Darboux for Jacobi and Toda lattices, although this requires further research beyond the scope of the present work.   
 
\section{Further extensions}
\label{sec:FE}

The Darboux transformations of CMV matrices introduced in this paper admit a couple of generalizations of interest. One of them has to do with the degree of the Hermitian Laurent polynomial involved in the transformation. The results of the previous sections carry over directly to polynomial modifications of higher degree. Indeed, announcing this generalization, some of the results previously obtained were proved for Laurent polynomials of arbitrary degree (see Section~\ref{ssec:Dfac}).

On the other hand, as we pointed out previously, Darboux transformations of Jacobi matrices can deal also with tridiagonal matrices whose lower and upper diagonals are not necessarily positive, but simply non-null. This connects such transformations with general quasi-definite functionals on the real line or, in other words, with orthogonal polynomials with respect to signed measures supported on the real line. Although this generalization seems to be missing in the Cholesky approach to Darboux transformations, the quasi-definite case can be tackled resorting to generalized Cholesky factorizations $AEA^+$, $A\in\mathscr{T}$, involving diagonal matrices $E$ whose diagonal entries are $\pm1$ (`sign matrices' in short). The matrices playing the role of CMV for this generalization constitute a wider class of matrices $\cal C$ which can be non-unitary, but are always quasi-unitary \cite{Wil1937} with respect to a certain sign matrix $E$, i.e. 
$$
 {\cal C}E{\cal C}^+ = E = {\cal C}^+E{\cal C}.
$$
We will call quasi-CMV matrices to such a generalization of CMV, given by \eqref{eq:q-CMV-schur}. Although not named in this way, quasi-CMV matrices were introduced for the first time by some of the authors in the paper \cite{CMV2003} which rediscovered CMV a decade after Watkins' one \cite{W1993}. Actually, \cite{CMV2003} introduces directly quasi-CMV matrices, from which CMV arises as a subclass.  

In this section we will briefly discuss these two generalizations of Darboux transformations for CMV matrices. 

\subsection{Higher degree transformations}
\label{ssec:deg}

Darboux transformations of CMV matrices can be associated with any Hermitian Laurent polynomial of arbitrary degree. They are defined exactly as in the case of degree one. Since Section~\ref{ssec:Dfac}, central for the rest of the paper, is developed for an arbitrary degree, the generalization to higher degree of the results previously obtained for degree one is straightforward. Therefore, the aim of this section is not a detailed discussion of higher degree Darboux transformations, but to give an account of those features which depend on the degree. 

Let $\ell$ be a Hermitian Laurent polynomial of degree $d$, i.e. with $N_z=2d$ zeros counting multiplicity. Given a CMV matrix $\cal C$, the Cholesky factorization $\ell({\cal C})=AA^+$, $A\in\mathscr{T}$, exists iff $\ell({\cal C})$ is positive definite, and then it is unique (see Appendix~\ref{app:cholesky}). Due to the CMV shape \eqref{eq:CMVshape} of $\cal C$, the Hermitian matrix $\ell({\cal C})$ is $2N_z+1$-diagonal with non-null entries in the upper and lower diagonals, so $A$ is $N_z+1$-band with non-null entries in the lower diagonal, i.e.
\begin{equation} \label{eq:A-Nband}
\begin{gathered}
 \scs N_z \times N_z
 \\[-5pt]
 \tikz \draw[rounded corners=2pt]
 (0,0) -- (0,-1.85) -- (2.1,-1.85) -- (2.1,0) -- cycle;
 \\[-61.5pt]
 \kern21.5pt
 A = 
 \left(
 \begin{smallmatrix}
 	\\
	+ 
	\\[2pt]
   	* & + 
   	\\[2pt]
	* & * & \kern-1.5pt +
	\\[2pt]
	* & * & \kern-1.5pt * & \kern-3pt +
	\\[-4pt]
	\sss\vdots & \sss\vdots & \kern-1.5pt \sss\ddots & \kern-3pt \sss\ddots 
	& \kern-3pt \sss\ddots
	\\
	* & * & \kern-1.5pt \text{\tiny$\cdots$} & \kern-3pt * & \kern-3pt * 
	& \kern-3pt +
	\\[2pt]
    \circledast & * & \kern-1.5pt * & \kern-3pt \text{\tiny$\cdots$} 
    & \kern-3pt * & \kern-3pt * & \kern-3pt +
    \\[2pt]
   	& \circledast & \kern-1.5pt * & \kern-3pt * 
	& \kern-3pt \text{\tiny$\cdots$} & \kern-3pt * & \kern-3pt * & \kern-3pt + 
	\\[2pt]
   	&& \kern-1.5pt \circledast & \kern-3pt * & \kern-3pt * 
	& \kern-3pt \text{\tiny$\cdots$} & \kern-3pt * & \kern-3pt * & \kern-3pt +
	\\[1pt]
    &&& \kern-3pt \text{\tiny$\cdots$} & \kern-3pt \text{\tiny$\cdots$} 
    & \kern-3pt \text{\tiny$\cdots$} & \kern-3pt \text{\tiny$\cdots$} 
    & \kern-3pt \text{\tiny$\cdots$} & \kern-3pt \text{\tiny$\cdots$} 
    & \kern-3pt \text{\tiny$\cdots$}
 \end{smallmatrix}
 \right).
\end{gathered}
\end{equation}
This defines the higher degree Darboux transformations ${\cal C}\xmapsto{\ell}{\cal D}=A^{-1}{\cal C}A$ of CMV matrices. As in the case of degree one, they constitute the matrix translation of the mapping $u \xmapsto{\ell} v=u\ell$ between positive definite functionals. Actually, the positive definiteness of $\ell({\cal C})$, which guarantees its Cholesky factorization, is also equivalent to the positive definiteness of $u\ell$, as well as to the unitarity of $\cal D$, which is therefore CMV. Besides, the same factor $A$ appearing in the Cholesky factorization of $\ell({\cal C})$ gives the reversed Cholesky factorization $\ell({\cal D})=A^+A$.

Since every non-constant Hermitian Laurent polynomial factorizes into degree one ones, the functional interpretation of higher degree Darboux transformations implies their equivalence with compositions of degree one transformations. However, this is not the case for higher degree Darboux transformations with parameters, which exhibit a higher richness of spurious solutions than mere compositions of degree one transformations. To show this, consider a Hermitian Laurent polynomial $\ell$ with $N_z$ zeros, a CMV matrix $\cal D$ and a $N_z+1$-band triangular matrix $A\in\mathscr{T}$. Let us rewrite $\ell({\cal D})$ and $A$ as $N_z \times N_z$-block matrices,
$$
\begin{aligned}
 & \ell({\cal C}) = 
 \left(
 \begin{smallmatrix}
 	\\
 	L_0 & \kern-1pt M_0^+ \\[2pt] 
 	M_0 & \kern-1pt L_1 & \kern-1pt M_1^+ \\[2pt] 
 	& \kern-1pt M_1 & \kern-1pt L_2 & \kern-1pt M_2^+ \\[-3pt]
 	& & \kern-4pt \ddots & \kern-6pt \ddots & \kern-2pt \ddots
 \end{smallmatrix}
 \right),
 \qquad 
 L_n = L_n^+,
 \qquad 
 M_n = 
 \left(
 \begin{smallmatrix}
	\\
 	\circledast & \kern-1pt * & \kern-2pt * & \kern-2pt \text{\tiny$\cdots$} 
	& \kern-2pt * & \kern0.5pt * 
	\\[-4pt]	
    & \kern-1pt \sss\ddots & \kern-2pt \sss\ddots & \kern-2pt \sss\ddots 
    & \kern-2pt \sss\vdots & \kern0.5pt \sss\vdots 
    \\
   	& & \kern-2pt \circledast & \kern-2pt * & \kern-2pt * & \kern0.5pt *
	\\[1pt]
   	&&& \kern-2pt \circledast & \kern-2pt * & \kern0.5pt *
	\\[1pt]
   	&&&& \kern-2pt \circledast & \kern0.5pt * 
	\\[1pt]
	&&&&& \kern0.5pt \circledast
 \end{smallmatrix}
 \right),
 \\[3pt]
 & A = 
 \left(
 \begin{smallmatrix}
 	\\
 	R_0 \\[3pt] 
 	S_0 & R_1 \\[3pt] 
 	& S_1 & R_2 \\[-2pt]
 	& & \ddots & \ddots
 \end{smallmatrix}
 \right),
 \qquad
 R_n =
 \left(
 \begin{smallmatrix}
 	+ 
	\\[2pt]
	* & + 
	\\[2pt]
	* & * & \kern-1.5pt + 
	\\[2pt]
	* & * & \kern-1.5pt * & \kern-3pt + 
	\\[-4pt]
	\sss\vdots & \sss\vdots & \kern-1.5pt \sss\ddots & \kern-3pt \sss\ddots 
	& \kern-3pt \sss\ddots 
	\\
	* & * & \kern-1.5pt \text{\tiny$\cdots$} & \kern-3pt * & \kern-3pt * 
	& \kern-1pt +
 \end{smallmatrix}
 \right),
 \qquad
 S_n = 
 \left(
 \begin{smallmatrix}
	\\
 	\circledast & \kern-1pt * & \kern-2pt * & \kern-2pt \text{\tiny$\cdots$} 
	& \kern-2pt * & \kern0.5pt * 
	\\[-4pt]	
    & \kern-1pt \sss\ddots & \kern-2pt \sss\ddots & \kern-2pt \sss\ddots 
    & \kern-2pt \sss\vdots & \kern0.5pt \sss\vdots 
    \\
   	& & \kern-2pt \circledast & \kern-2pt * & \kern-2pt * & \kern0.5pt *
	\\[1pt]
   	&&& \kern-2pt \circledast & \kern-2pt * & \kern0.5pt *
	\\[1pt]
   	&&&& \kern-2pt \circledast & \kern0.5pt * 
	\\[1pt]
	&&&&& \kern0.5pt \circledast
 \end{smallmatrix}
 \right).
\end{aligned}
$$
Note that $R_0$ is the $N_z \times N_z$ leading submatrix of $A$ highlighted in \eqref{eq:A-Nband}, and $M_n$, $R_n$, $S_n$ are non-singular for all $n\ge0$. Then, the reversed Cholesky factorization $\ell({\cal D})=A^+A$ reads as
\begin{equation} \label{eq:b-inverse}
 R_n^+R_n + S_n^+S_n = L_n, 
 \qquad
 R_{n+1}^+S_n = M_n,
 \qquad
 n\ge0.
\end{equation} 
Since $S_n$ is non-singular, it has a unique polar decomposition $S_n=V_n|S_n|$ with $V_n$ unitary and $|S_n|$ the positive definite square root of $S_n^+S_n$. This allows us to rewrite \eqref{eq:b-inverse} as
$$
 |S_n|^2 = L_n - R_n^+R_n, 
 \qquad 
 |S_n|M_n^{-1} = V_n^+(R_n^+)^{-1}.
$$
Therefore, $R_n$ determines $|S_n|$ via the first equation, while $|S_n|$ fixes $V_n$ and $R_{n+1}$, provided by the QR factors of $|S_n|M_n^{-1}$ as the second equation shows. This implies that all the blocks of $A$ are inductively generated starting from a choice of $R_0$, which is only restricted by the condition that $|S_n|$ must be positive definite for all $n\ge0$. Therefore, apart from this constraint, the entries of $R_0$ can be considered as `free' parameters describing all the reversed Cholesky factorizations of $\ell({\cal D})$. In the case $N_z=2$, using the notation \eqref{eq:Aexplicit} we get
$R_0=\left(\begin{smallmatrix}r_0&0\\s_0&r_1\end{smallmatrix}\right)$, which encodes the `free' parameters already known for degree one.

The number of non-trivial entries of $R_0$ is $N_z(N_z+1)/2$, among them $N_z$  positive (the diagonal ones) and the rest are, in general, complex numbers. Therefore, the number of `free' real parameters of a Darboux transformation with parameters is $N_z^2$, where $N_z$ is the number of zeros of the corresponding Hermitian Laurent polynomial. This in contrast with the composition of degree one Darboux transformations with parameters associated with $d=N_z/2$ factors of the Hermitian Laurent polynomial, which only leads to  $4d=2N_z$ `free' real parameters. 

Although the composition of degree one Darboux transformations with parameters gives in principle less spurious solutions than higher degree ones, this iterative procedure to deal with higher degree inverse Darboux has a drawback because such a composition is not feasible in most of the cases. The reason is that every single degree one step gives many spurious solutions which are not CMV neither unitary, preventing us from composing with a new CMV Darboux transformation. This forces us to take into account the CMV conditions \eqref{eq:CMVcond} at any step. Then, the resulting contrained iterative procedure should end up with a set of solutions parametrized by $d$ complex numbers in the open unit disk, playing the role of the first $d$ Schur parameters of the CMV solutions. 

Another reason could increase even more the interest of higher degree Darboux transformations with parameters for CMV: The spurious solutions might have their own interest if associated with non-standard orthogonal polynomials on the unit circle. There is a strong reason to expect an interesting connection with non-standard orthogonality. Higher degree Darboux transformations with parameters for Jacobi matrices have been already considered in the literature \cite{EK,Gau}. Contrary to the case of degree one, they present spurious solutions which are related to Sobolev orthogonal polynomials on the real line \cite{DGAM,DM}.  

The parametric analysis of spurious solutions for Jacobi follows the same arguments given previously, since they work for any matrix which, like $\ell({\cal C})$, can be considered as a block-Jacobi one. Then, if a real polynomial $\wp$ has $N_z$ zeros (coinciding now with the degree), the corresponding Darboux transformations with parameters of a Jacobi matrix $\cal J$ are described by $N_z(N_z+1)/2$ `free' parameters, those encoded in the $N_z \times N_z$ leading submatrix of $\wp({\cal J})$. The only quirks of the Jacobi case is that all these parameters are now real and that, in contrast to the CMV case, any (not necessarily even) number $N_z$ of zeros is possible. Actually, the fact that the minimal step for Jacobi and CMV is $N_z=1$ and $N_z=2$, respectively, is the cause of the possibility of avoiding spurious solutions in the Jacobi but not in the CMV case.

Due to the equivalence with polynomial transformations of measures, higher degree Darboux transformations of Jacobi and CMV matrices are both almost isospectral. More precisely, the Darboux and inverse Darboux transformations may delete or add, respectively, at most $N_z$ eigenvalues.

\subsection{Quasi-CMV matrices}
\label{ssec:quasi}

Quasi-CMV matrices are the matrix representations of the multiplication operator \eqref{eq:z} with respect to zig-zag bases which are `orthonormal' with respect to quasi-definite functionals on the unit circle. Such functionals are those Hermitian linear functionals $u$ on $\Lambda$ whose Gram matrix in $\PP_n=\spn\{1,z,z^2,\dots,z^n\}$ is non-singular for all $n=0,1,2,\dots$. (equivalently, the Gram matrix of $u$ in $\PP$ with respect to a basis $p_n\in\PP_n\setminus\PP_{n-1}$ is quasi-definite, i.e. all its leading submatrices are non-singular). This property, independent of the basis chosen for the Gram matrix, is equivalent to the existence of a sequence of polynomials $\varphi_n\in\PP_n\setminus\PP_{n-1}$ which are `orthonormal' with respect to $u$, i.e. $u[\varphi_n\varphi_{m*}]=e_n\delta_{n,m}$ with $e_n\in\{1,-1\}$. The polynomials $\varphi_n$ satisfy forward and backward recurrence relations which generalize \eqref{eq:RR}, 
\begin{equation} \label{eq:q-RR}
\begin{cases}
 \rho_n \varphi_n(z) = z \varphi_{n-1}(z) + a_n \varphi_{n-1}^*(z),
 \\[3pt]
 \hat\rho_n z\varphi_{n-1}(z) = \varphi_n(z) - a_n \varphi_ n^*(z),
\end{cases}
\quad
\begin{gathered}
 \rho_n = \sqrt{|1-|a_n|^2|},
 \\[2pt] 
 \hat\rho_n = \frac{e_n}{e_{n-1}}\rho_n = \sg(1-|a_n|^2)\rho_n,
\end{gathered}
\end{equation}
where $\sg(x)=x/|x|$, $x\in\R\setminus\{0\}$, is the sign function and the Schur parameters $a_n$ are no longer restricted to the open unit disk, but they can lie anywhere in the complex plane outside of the unit circle. 

The previous extension of orthogonal polynomials on the unit circle goes back at least to the works of Geronimus \cite{Ger54}. However, the connection of the quasi-definite case with `orthonormal' Laurent polynomials and quasi-CMV matrices, which we sumarize below, is of much more recent vintage \cite{CMV2003}.
   
Since $f(z) \mapsto z^{-[n/2]}f(z)$ is a one-to-one correspondence between $\PP_n$ and $\LL_n$ which preserves the sesquilinear form $(f,g) \mapsto u[fg_*]$, the quasi-definiteness of $u$ also means that its Gram matrix in $\LL_n$ is non-singular for all $n=0,1,2,\dots$ (equivalently, the Gram matrix of $u$ in $\Lambda$ with respect to a basis $l_n\in\LL_n\setminus\LL_{n-1}$ is quasi-definite). Therefore, analogously to the previous discussion, the quasi-definiteness of $u$ is equivalent to the existence of a zig-zag basis $\chi$ which is `orthonormal' with respect to $u$, i.e. $u[\chi_n\chi_{m*}]=e'_n\delta_{n,m}$, $e'_n\in\{1,-1\}$. Actually, as in the positive definite case, the relation between the orthonormal polynomials $\varphi_n$ and the orthonormal Laurent polynomials $\chi_n$ is given by \eqref{eq:OLP-OP}. This implies that $e'_n=e_n$, so that the orthonormality condition for $\chi$ can be rewritten as $u[\chi\chi^+]=E$, where $E$ denotes the sign matrix 
\begin{equation} \label{eq:sign}
 E = 
 \begin{pmatrix}
 	e_0 \\[-3pt] 
 	& \kern-3pt e_1 \\[-3pt] 
 	& & \kern-3pt e_2 \\[-3pt] 
	& & & \kern-3pt \ddots
 \end{pmatrix},
 \qquad 
 \begin{aligned}
  	& e_0=\sg(u[1]),
 	 \\
  	& e_n=e_0\prod_{k=1}^n\sg(1-|a_k|^2), \quad n\ge1.
 \end{aligned}
\end{equation} 
It follows from \eqref{eq:OLP-OP} and \eqref{eq:q-RR} that $z\chi={\cal C}\chi$, where
\begin{equation} \label{eq:q-CMV-schur}
{\cal C} = 
\begin{pmatrix} 
	-\overline{a}_0a_1 & \overline{a}_0\rho_1
	\\
	-\hat\rho_1a_2 & -\overline{a}_1a_2 & 
	-\rho_2a_3 & \rho_2\rho_3 
	\\
    \hat\rho_1\hat\rho_2 & \overline{a}_1\hat\rho_2 & 
    -\overline{a}_2a_3 & \overline{a}_2\rho_3
    \\
    && -\hat\rho_3a_4 & -\overline{a}_3a_4 &
    -\rho_4a_5 & \rho_4\rho_5
    \\
    && \hat\rho_3\hat\rho_4 & \overline{a}_3\hat\rho_4 &
    -\overline{a}_4a_5 & \overline{a}_4\rho_5
    \\
    &&&& \cdots & \cdots & \cdots
\end{pmatrix},
\qquad
a_0=1.
\end{equation}
This matrix, which is the representation of the multiplication operator with respect to $\chi$, has a unique inverse, the matrix representation of $z^{-1}$ with respect to $\chi$. We will refer to $\cal C$ as a quasi-CMV matrix due to its quasi-unitarity with respect to the sign matrix $E$ given in \eqref{eq:sign}. Since $\pm u$ have the same orthonormal zig-zag basis, the quasi-definite functional related to a quasi-CMV matrix is unique up to a non-null real rescaling. 

A drawback of Darboux transformations for a quasi-CMV matrix $\cal C$ is the possible lack of hermiticity of $\ell({\cal C})$ for a Hermitian Laurent polynomial $\ell(z)=\alpha z+\beta+\overline\alpha z^{-1}$. This is overcome by the quasi-unitarity of $\cal C$, which implies that ${\cal C}^{-1}=E{\cal C}^+E$. As a consequence, we get two Hermitian matrices, 
$$
 \ell({\cal C})E = \alpha {\cal C}E + \beta E + \overline\alpha E{\cal C}^+,
 \qquad
 E\ell({\cal C}) = \alpha E{\cal C} + \beta E + \overline\alpha {\cal C}^+E,
$$
which will share in the quasi-CMV case the role played by $\ell({\cal C})$ for CMV matrices. The hermiticity of $\ell({\cal C})E$ or $\ell({\cal C})E$ characterizes the sign matrix $E$ which makes $\cal C$ quasi-unitary, a result proved in the following lemma, which also shows that such a sign matrix is determined by $\cal C$ up to a sign. 

\begin{lem} \label{lem:q-sign}
Let $\cal C$ be a quasi-CMV matrix, $\widehat{E}$ a sign matrix and $\ell$ a Hermitian Laurent polynomial of degree one. Then, the following statements are equivalent:
\begin{itemize}
\item[(i)] $\cal C$ is quasi-unitary with respect to $\widehat{E}$.
\item[(ii)] $\ell({\cal C})\widehat{E}$ is Hermitian.
\item[(iii)] $\widehat{E}\ell({\cal C})$ is Hermitian.
\item[(iv)] $\widehat{E}=\pm E$, where $E$ is given by \eqref{eq:sign}.
\end{itemize}
\end{lem}

\begin{proof}
We already know that $(i) \Rightarrow (ii)$, $(i) \Rightarrow (iii)$ and $(iv) \Rightarrow (i)$. On the other hand, $(ii) \Leftrightarrow (iii)$ obviously because $\widehat{E}\ell({\cal C})=\widehat{E}(\ell({\cal C})\widehat{E})\widehat{E}$. Hence, it only remains to see that $(ii) \Rightarrow (iv)$. Let us write $\ell(C)\widehat{E}=M\widetilde{E}$, where $M=\ell({\cal C})E$ is a five-diagonal matrix with the shape \eqref{eq:LC-bis} and $\widetilde{E}=E\widehat{E}$ is a sign matrix. Assuming $(ii)$ means that $M\widetilde{E}$ is Hermitian, which is equivalent to state that $M\widetilde{E}=\widetilde{E}M$ because $M$ and $\widetilde{E}$ are also Hermitian. It is easy to see that, due to the structure \eqref{eq:LC-bis} of non-null matrix coefficients of $M$, the only sign matrices which commute with it are $\pm I$. We conclude that $\widetilde{E}=\pm I$, which proves $(iv)$.   
\end{proof}  

If $u$ is a quasi-definite functional associated with the quasi-CMV matrix $\cal C$ and $E$ is the corresponding sign matrix, then we get 
$$
 u\ell[\chi\chi^+] = u[\ell\chi\chi^+] = \ell({\cal C}) \, u[\chi\chi^+] 
 = \ell({\cal C})E.
$$
As a consequence, we obtain the quasi-definite version of Corollary~\ref{cor:pd}: $u\ell$ is quasi-definite iff $\ell({\cal C})E$ is quasi-definite. The last condition is also equivalent to the existence of a generalized Cholesky factorization (see Appendix~\ref{app:cholesky})
\begin{equation} \label{eq:LCE}
 \ell({\cal C})E = AFA^+, 
 \qquad A\in\mathscr{T}, 
 \qquad F\in\mathscr{S} = \, \text{set of infinite sign matrices},
\end{equation} 
which is also unique. Since $\ell({\cal C})$ has the five-diagonal structure \eqref{eq:LC}, as in the CMV case, $A$ has the 3-band shape \eqref{eq:A-3band}. The factorization \eqref{eq:LCE} is the starting point for the Darboux transformations of quasi-CMV matrices.

To introduce such transformations we just follow the same steps as in the case of CMV matrices. From Proposition \ref{prop:D-fac}, the zig-zag basis $\omega$ given by $\chi=A\omega$ and the matrix $B$ defined by $\ell\omega=B\chi$ yield the factorization $\ell({\cal C})=AB$. Comparing this with \eqref{eq:LCE}, taking care of the associativity with Proposition \ref{prop:asoc} and bearing in mind that $\ker_R(A)=\{0\}$, we conclude that $B$ is the 3-band upper triangular matrix
\begin{equation} \label{eq:q-B}
 B=FA^+E. 
\end{equation}
Using again Proposition \ref{prop:D-fac} we also find that ${\cal D}=A^{-1}{\cal C}A$ is the matrix of the multiplication operator with respect to $\omega$ and satisfies
\begin{equation} \label{eq:FLD}
 F\ell({\cal D}) = A^+EA,
 \qquad A\in\mathscr{T}, 
 \qquad E\in\mathscr{S},
\end{equation}
which is a generalized reversed Cholesky factorization. 
Besides, the associativity properties of Proposition \ref{prop:asoc} and \eqref{eq:q-B} yield
$$
 u\ell[\omega\omega^+] = u[\ell\omega\omega^+] = Bu[\chi\chi^+](A^+)^{-1} 
 = F.
$$ 
This identity has several consequences: The zig-zag basis $\omega$ is orthonormal with respect to the quasi-definite functional $v=u\ell$, thus ${\cal D}$ is a quasi-CMV matrix and $F$ is the corresponding sign matrix.    

Sumarizing, given a quasi-CMV matrix $\cal C$ with sign matrix $E$ and a Hermitian Laurent polynomial $\ell$ such that $\ell(C)E$ is quasi-definite, we have found a transformation ${\cal C}\xmapsto{\ell}{\cal D}=A^{-1}{\cal C}A$ which generates a new quasi-CMV matrix $\cal D$. The ingredients of this transformation, i.e. the lower triangular matrix $A$ and the sign matrix $F$ of $\cal D$, are given by the generalized Cholesky factorization \eqref{eq:LCE} of $\ell(C)E$. The  matrix $F\ell(\cal D)$ satisfies a generalized reversed Cholesky factorization obtained from that of $\ell({\cal C})E$ by changing $A \to A^+$ and $F \to E$. This is what we call the Darboux transformation for quasi-CMV matrices. In perfect agreement with the CMV case, this transformation is the matrix translation of the transformation $u \xmapsto{\ell} v=u\ell$ connecting quasi-definite functionals.

Concerning inverse Darboux transformations, starting from a quasi-CMV matrix $\cal D$ and its sign matrix $F$, the matrix transformation ${\cal C}=A{\cal D}A^{-1}$, with $A$ and $E$ given by \eqref{eq:FLD}, defines the Darboux transformations with parameters for quasi-CMV matrices. This matrix procedure solves the inverse problem, which consists in finding all the quasi-CMV matrices $\cal C$ such that ${\cal C}\xmapsto{\ell}{\cal D}$, but also yields spurious solutions $\cal C$ which are neither quasi-CMV nor quasi-unitary. Analogously to the CMV case, the solutions of the inverse Darboux transformations coincide with the quasi-CMV solutions of the Darboux transformations with parameters. To prove this note that, if $\omega$ is an orthonormal zig-zag basis for $\cal D$ and $\chi=A\omega$ then, from \eqref{eq:FLD} and Proposition \ref{prop:asoc}, $\ell\omega = \ell({\cal D})\omega = FA^+EA\omega = B\chi$ with $B$ given by \eqref{eq:q-B}. In view of Proposition~\ref{prop:D-fac}, we find that ${\cal D}=A^{-1}{\cal C}A$ with $\ell({\cal C})=AB=AFA^+E$. Since $\ell({\cal C})E=AFA^+$ is Hermitian, Lemma \ref{lem:q-sign} proves that, whenever $\cal C$ is quasi-CMV, $E$ is its sign matrix up to a sign. Therefore, ${\cal C}\xmapsto{\ell}{\cal D}$ for such a quasi-CMV matrix $\cal C$.  

The theory of Darboux transformations for quasi-CMV matrices can be developed just by mimicking the steps and arguments given for the CMV case. All the results obtained in the previous sections have a natural generalization to the quasi-CMV case, the details are left to the reader. We will simply mention some specific results of interest for the example discussed below.

Assuming the notation \eqref{eq:Aexplicit} for the matrix coefficients of the factor $A$ and denoting by $f_n$ the diagonal entries of the sign matrix $F$, the factorization \eqref{eq:LCE} reads as 
\begin{equation} \label{eq:q-direct} 
\left\{
\begin{aligned}
 & f_{n-2}|t_{n-2}|^2+f_{n-1}|s_{n-1}|^2+f_nr_n^2 = 
 e_n\left(\beta-2\re(\alpha\overline{a}_na_{n+1})\right),
 \\
 & f_{n-1}s_{n-1}\overline{t}_{n-1}+f_nr_ns_n = 
 e_{n+1}\rho_{n+1}(\overline\alpha a_n-\alpha a_{n+2}),
 \\
 & f_nr_nt_n = \alpha e_{n+2}\rho_{n+1}\rho_{n+2},
\end{aligned}
\right.
\qquad 
n\ge0.
\end{equation}
These equations, together with the initial conditions $t_{-2}=t_{-1}=s_{-1}=0$, make explicit the Darboux transformations for quasi-CMV matrices. Besides, generalizing the identities in \eqref{eq:a-b}, \eqref{eq:r-t} and \eqref{eq:b-a} we get the following direct relations between the matrix coefficients of $A$ and the Schur parameters of the quasi-CMV matrices involved in a Darboux transformation, 
\begin{equation} \label{eq:q-aAb}
\begin{gathered}
 \frac{r_{n-1}}{r_n} = \frac{\rho_n}{\sigma_n},
 \qquad 
 \frac{t_{n-1}}{t_n} = \frac{\hat\sigma_n}{\hat\rho_{n+2}},
 \\[3pt]
 a_{n} = b_{n} + s_{n-1}\frac{\sigma_{n}}{r_n},
 \qquad
 \alpha a_{n+2} = 
 \overline\alpha b_{n} - e_{n+1}f_ns_{n}\frac{r_{n+1}}{\sigma_{n+1}}. 
\end{gathered}
\end{equation}

As for the Darboux transformations with parameters, inserting \eqref{eq:Aexplicit} into the factorization \eqref{eq:FLD} shows that they are explicitly given by 
\begin{equation} \label{eq:q-inverse}
\left\{
\begin{aligned}
 & e_nr_n^2+e_{n+1}|s_n|^2+e_{n+2}|t_n|^2 = 
 f_n\left(\beta-2\re(\alpha\overline{b}_nb_{n+1})\right),
 \\
 & e_{n+1}s_nr_{n+1}+e_{n+2}t_ns_{n+1} = 
 f_n\sigma_{n+1}(\overline\alpha b_n-\alpha b_{n+2}),
 \\
 & e_{n+2}t_nr_{n+2} = \alpha f_n\sigma_{n+1}\sigma_{n+2},
\end{aligned}
\right.
\qquad n\ge0,
\end{equation}
where $b_n$ are the Schur parameters of $\cal D$ and $\sigma_n=\sqrt{|1-|b_n|^2|}$. These relations determine $e_n$, $r_n$, $s_n$, $t_n$ once $e_0$, $r_0$, $s_0$, $e_1$, $r_1$ are chosen, as follows from rewriting \eqref{eq:q-inverse} as
\begin{equation} \label{eq:q-inverse2}
\begin{gathered}
e_{n+2} = \sg(\varepsilon_n),
\qquad 
r_{n+2} = |\alpha| \frac{\sigma_{n+1}\sigma_{n+2}}{\tau_n},
\qquad
t_n = \frac{\alpha}{|\alpha|} e_{n+2} f_n \tau_n,
\\[2pt]
s_{n+1} = \frac{\overline\alpha}{|\alpha|} 
\frac{\sigma_{n+1}(\overline\alpha b_n-\alpha b_{n+2})-e_{n+1}f_ns_nr_{n+1}}{\tau_n},
\\[3pt] 
\tau_n = \sqrt{|\varepsilon_n|},
\qquad
\varepsilon_n = f_n
\left(\beta-2\re(\alpha\overline{b}_nb_{n+1})\right)-e_nr_n^2-e_{n+1}|s_n|^2.
\end{gathered}
\end{equation}
Therefore, the solutions of the Darboux transformations with parameters can be parametrized by $e_0r_0^2$, $s_0$, $e_1r_1^2$, which give again four real `free' parameters. The existence of such a solution for a given value of these free parameters is equivalent to state that $\varepsilon_n\ne0$ for all $n\ge0$. Among these solutions, the quasi-CMV ones are characterized by a generalization of the CMV conditions \eqref{eq:CMVcond}, namely,  
\begin{equation} \label{eq:q-CMVcond} 
 \begin{aligned} 
 	& e_0f_0r_0^2 = \beta-2\re(\alpha a),
 	& \kern15pt & r_1=r_0\frac{\sigma_1}{\rho},
	& \kern15pt & |a|\ne1,
 	\\
	& a = b_1 + \frac{s_0}{r_1}\sigma_1,
	& & e_1=e_0\,\sg(1-|a|^2),
	& & \rho=\sqrt{|1-|a|^2|}.
 \end{aligned}
\end{equation}
The conditions in the second column of \eqref{eq:q-CMVcond} can be combined into $e_1f_1r_1^2(1-|a|^2)=e_0f_0r_0^2(1-|b_1|^2)$. Like in the CMV case, $a$ is the first Schur parameter $a_1$ of the corresponding quasi-CMV solution, while the pair of equalities in the first column of the quasi-CMV conditions \eqref{eq:q-CMVcond} select the solutions associated with (not necessarily quasi-definite) Hermitian functionals.

\begin{ex} \label{ex:QD-spurious} 
{\it Darboux transformation with parameters for a general Hermitian Laurent polynomial of degree one $\ell(z)=\alpha z+\beta+\overline\alpha z^{-1}$ and
${\cal D}={\cal S}$ the shift matrix given in \eqref{eq:shift}.}

We will illustrate with this example the coexistence of CMV, quasi-CMV and spurious solutions in the Darboux transformations with parameters. The shift matrix ${\cal D}={\cal S}$ is a CMV with Schur parameters $b_n=0$ (so, $\sigma_n=1$) for $n\ge1$, thus $F=I$ is a sign matrix for $\cal S$. We will search for solutions $\cal C$ of the corresponding Darboux transformation with parameters, considered in the generalized sense described in this section. This allows the appearance of quasi-CMV solutions $\cal C$ with a non-trivial sign matrix $E$. Without loss of generality we can take $|\alpha|=1$, which simply means that we avoid a trivial positive rescaling of the factor $A$ in the reversed factorization $\ell({\cal S})=A^+EA$, where 
\begin{equation} \label{eq:LS}
\ell({\cal S}) = 
\left(
\begin{smallmatrix}
  	\beta & \kern4pt \alpha & \kern4pt \overline\alpha
	\\[3pt]
  	\overline\alpha & \kern4pt \beta & \kern4pt 0 & \kern2pt \alpha
	\\[3pt]
  	\alpha & \kern4pt 0 & \kern4pt \beta & \kern2pt 0 & \overline\alpha
	\\[3pt]
  	& \kern4pt \overline\alpha & \kern4pt 0 & \kern2pt \beta & 0 & \alpha
	\\[3pt]
  	& & \kern4pt \alpha & \kern2pt 0 & \beta & 0 & \overline\alpha
	\\[-3pt]
  	& & & \kern2pt \ddots & \ddots & \ddots & \ddots & \ddots
\end{smallmatrix}
\right).
\end{equation}

The diagonal entries $e_n$ of the sign matrix $E$, as well as the coefficients $r_n$, $s_n$, $t_n$ giving and the three diagonals of $A$, are determined by \eqref{eq:q-inverse2}. 
In what follows we will restrict our attention to the subset of solutions with $s_1=0$, a condition which becomes a constraint between the parameters $e_0r_0^2$, $s_0$, $e_1r_1^2$ describing the whole set of solutions. Therefore, the referred subset of solutions can be parametrized by $e_0r_0^2$ and $e_1r_1^2$, although instead we will use $e_1r_1^2$ and $e_2r_2^2$ for convenience. The corresponding parametrization of $A$ and $E$, as follows from \eqref{eq:q-inverse2}, is given by 
\begin{equation} \label{eq:q-inv-ex}
\begin{gathered}
 s_0 = \overline\alpha\frac{e_1}{r_1},
 \qquad 
 e_0r_0^2 = \beta-e_1|s_0|^2-\frac{1}{e_2r_2^2} 
 = \beta-\frac{1}{e_1r_1^2}-\frac{1}{e_2r_2^2},
 \\
 s_n = 0,
 \qquad
 e_{n+2}r_{n+2}^2 = \frac{1}{\beta-e_nr_n^2},
 \qquad
 t_{n-1} = \alpha\frac{e_{n+1}}{r_{n+1}}, 
 \qquad 
 n\ge1. 
\end{gathered}
\end{equation}
These equations are governed by a single recurrence relation 
\begin{equation} \label{eq:q-rec-x}
 x_{n+1} = \frac{1}{\beta-x_n}, 
 \qquad
 x_n = 
 \begin{cases} 
 	e_{2n+1}r_{2n+1}^2, 
	\\ 
	\kern9pt \text{\footnotesize or} 
	\\ 
	e_{2n+2}r_{2n+2}^2, 
 \end{cases}
 \quad
 n\ge0,
\end{equation}
which provides separately the odd and even coefficients $r_n$, $t_n$ depending whether we use the initial condition $x_0=e_1r_1^2$ or $x_0=e_2r_2^2$. The initial conditions generating true solutions of the Darboux transformation with parameters are those giving $x_n\ne\beta$ for $n\ge0$ and guaranteeing that $e_0r_0^2\ne0$, i.e.  
\begin{equation} \label{eq:q-cond-r0}
 \frac{1}{e_1r_1^2}+\frac{1}{e_2r_2^2} \ne \beta.
\end{equation}

Recurrence \eqref{eq:q-rec-x} yields the continued fraction expansion 
$$
 x_n = 
 \polter{1}{\beta}-\polter{1}{\beta}-\polter{1}{\beta}-\polter{1}{\beta}
 -\cdots-\polter{1}{\beta}-\polter{1}{\beta-x_0}.
$$
Using basic results from the theory of continued fractions \cite{Wall}, we obtain
$$
 x_n = 
 \frac{U_{n-1}(\frac{\beta}{2})-x_0U_{n-2}(\frac{\beta}{2})}
 {U_n(\frac{\beta}{2})-x_0U_{n-1}(\frac{\beta}{2})},
 \qquad n\ge1,
$$ 
where $U_n$ are the second kind Chebyshev polynomials, given by
\begin{equation} \label{eq:cheby}
 U_{n+1}(t) = 2tU_n(t) - U_{n-1}(t), \qquad U_{-1}=0, \qquad U_0=1.
\end{equation}
Therefore, if $s_1=0$, the reversed Cholesky factorization $\ell({\cal C})=A^+EA$ is possible exactly for those choices of $e_1r_1^2$ and $e_2r_2^2$ satisfying \eqref{eq:q-cond-r0} and  
\begin{equation} \label{eq:q-ex-rev}
 e_1r_1^2, e_2r_2^2 \ne \frac{U_n(\frac{\beta}{2})}{U_{n-1}(\frac{\beta}{2})},
 \qquad n\ge1.
\end{equation}

The solutions associated with a Hermitian functional are characterized by the two equalities in the first column of \eqref{eq:q-CMVcond}, which in this case become
\begin{equation} \label{eq:q-ex-herm}
 e_0r_0^2 = \beta - \frac{2}{e_1r_1^2}.
\end{equation}
Using the expression for $r_0$ given in \eqref{eq:q-inv-ex}, the above hermiticity condition turns out to be $e_1=e_2$ and $r_1=r_2$. Therefore, bearing in mind \eqref{eq:q-cond-r0} and \eqref{eq:q-ex-rev}, the solutions leading to Hermitian functionals are those generated by \eqref{eq:q-rec-x} with initial conditions satisfying
\begin{equation} \label{eq:q-ex-herm2} 
 \frac{2}{\beta} \ne x_0 \ne 
 \frac{U_n(\frac{\beta}{2})}{U_{n-1}(\frac{\beta}{2})},
 \qquad
 e_1r_1^2 = x_0 = e_2r_2^2,
 \qquad n\ge1.
\end{equation}

The quasi-CMV solutions are those satisfying the previous hermiticity conditions together with those summarized in the second and third column of \eqref{eq:q-CMVcond}. In our case, these additional conditions read as  
$$ 
 |a|\ne1, \qquad x_0(1-|a|^2)=e_0r_0^2, \qquad a=\frac{\overline\alpha}{x_0}, 
$$
which, using \eqref{eq:q-ex-herm}, become
$$
 |x_0|\ne1, \qquad x_0 = \frac{1}{\beta-x_0}.
$$
This means that the quasi-CMV solutions are associated with those fixed points $x_n=x_0$ of \eqref{eq:q-rec-x} lying outside of the unit circle. Since $s_n=0$ for $n\ge1$, we find from \eqref{eq:q-aAb} that the Schur parameters of these quasi-CMV solutions are $a_1=a$ and $a_n=0$ for $n\ge2$. In other words, all the quasi-CMV solutions with $s_1=0$ is are Bernstein-Szeg\H o type. 

The fixed points $x_n=x_0$ of \eqref{eq:q-rec-x} are 
$x_0 = \frac{\beta}{2} \pm \sqrt{\left(\frac{\beta}{2}\right)^2-1}$, which are inverse of each other, while the requirement $|x_0|\ne1$ translates into the inequality $|\beta|>2$. It is convenient to rewrite these two fixed points as
\begin{equation} \label{eq:q-fp}
 x_0^\pm = \sg(\beta)  
 \left(\ts\left|\frac{\beta}{2}\right| 
 \pm \sqrt{\ts\left(\frac{\beta}{2}\right)^2-1}\right),
\end{equation}
so that $|x_0^\pm|=\left|\frac{\beta}{2}\right| \pm \sqrt{\ts\left(\frac{\beta}{2}\right)^2-1}$ for $|\beta|>2$. Then, $|x_0^+|>1$ and $|x_0^-|<1$. Both fixed points satisfy the hermiticity conditions \eqref{eq:q-ex-herm2}, thus they give rise to two different quasi-CMV solutions characterized by the Schur parameters $(a_\pm,0,0,0,\dots)$, with $a_\pm=\overline\alpha/x_0^\pm=\overline\alpha x_0^\mp$.
Since $|a_+|<1$ and $|a_-|>1$, the Bernstein-Szeg\H o solution related to the initial condition $x_0=x_0^+$ is positive definite, while that one corresponding to $x_0=x_0^-$ is quasi-definite but not positive definite. The related CMV matrix, ${\cal C}_+$, and quasi-CMV matrix, ${\cal C}_-$, have the form
$$
{\cal C}_\pm = 
 \left(
 \begin{smallmatrix} 
	-a_\pm & \rho_\pm
	\\[3pt]
	0 & 0 & \kern6pt 0 & \kern9pt 1 
	\\[3pt]
    \pm\rho_\pm & \overline{a}_\pm & \kern6pt 0 & \kern9pt 0
    \\[3pt]
    & & \kern6pt 0 & \kern9pt 0 & \kern7pt 0 & \kern5pt 1
    \\[3pt]
    & & \kern6pt 1 & \kern9pt 0 & \kern7pt 0 & \kern5pt 0
    \\[3pt]
    & & & & \kern6pt 0 & \kern9pt 0 & \kern7pt 0 & \kern5pt 1
    \\[3pt]
    & & & & \kern6pt 1 & \kern9pt 0 & \kern7pt 0 & \kern5pt 0
    \\[3pt]
    & & & & & & \kern7pt \cdots & \kern5pt \cdots & \kern4.5pt \cdots
 \end{smallmatrix}
 \right),
 \qquad
 \begin{aligned}
 	& a_\pm = \overline\alpha x_0^\mp, 
 	\\[2pt]
 	& \rho_\pm = \sqrt{|1-|a_\pm|^2|} = \sqrt{|1-(x_0^\mp)^2|}.
 \end{aligned}
$$
From the previous results we can obtain the explicit expressions of the factors $A_\pm$ and $E_\pm$ for both solutions, given by
$$
\begin{aligned}
 & A_\pm = 
 \left(
 \begin{smallmatrix} 
	\\
	r_0
	\\[3pt] 
	\overline{t}_\pm & r_\pm
	\\[3pt] 
	t_\pm & 0 & r_\pm
	\\[3pt]
	& \overline{t}_\pm & 0 & r_\pm
	\\[3pt]
	& & t_\pm & 0 & r_\pm
	\\[3pt]
	& & & \overline{t}_\pm & 0 & r_\pm
	\\[-3pt]
	& & & & \scriptsize\ddots & \scriptsize\ddots & \scriptsize\ddots
 \end{smallmatrix}
 \right),
 \qquad
 \begin{aligned}
 	& r_0^2 = 2\sqrt{\ts\left(\frac{\beta}{2}\right)^2-1}, 
	\qquad r_\pm^2 = |x_0^\pm|,
	\\[6pt]
	& t_\pm = \sg(\beta) \, \alpha r_\mp.
 \end{aligned}
 \\[2pt]
 & E_\pm = \sg(\beta)
 \left(
 \begin{smallmatrix} 
	\\
	\pm1
	\\[3pt] 
	& \kern4pt 1
	\\[3pt] 
	& & \kern4pt 1
	\\[3pt]
	& & & \kern4pt 1	
	\\[-3pt]
	& & & & \scriptsize\ddots
 \end{smallmatrix}
 \right).
\end{aligned}
$$
The CMV solution ${\cal C}_+$ has a sign matrix $E_+=\sg(\beta)I$ which is $-I$ for $\beta<0$, which is in agreement with the indetermination of the global sign for $E_+$. This means that we can take $E_+=I$ for ${\cal C}_+$, although, in that case, preserving the factorizations involved in the Darboux transformation would require to choose the sign matrix $F=\sg(\beta)I$ for the shift matrix ${\cal D}={\cal S}$. 
   
Summarizing, we have found that quasi-CMV solutions for the inverse Darboux problem with $s_1=0$ exist iff $|\beta|>2$. Under these conditions, we have a pair of such solutions, both Bernstein-Szeg\H o type, one of them CMV and the other one strictly quasi-CMV. Since the shift matrix $\cal S$ is associated with the Lebesgue measure $d\vartheta(e^{i\theta})=d\theta/2\pi$ on the unit circle, the CMV solution is that one of the Bernstein-Szeg\H o measure $d\vartheta/\ell$, and the sign of $E_+$ has to do with the fact $\ell$ is positive or negative on the unit circle depending whether $\beta>2$ or $\beta<-2$. Bearing in mind that $\ell$ has roots $\zeta_\pm=-a_\pm$ for $|\beta|>2$, every Hermitian solution $u$ of $u\ell\equiv d\vartheta$ is in that case given by $u \equiv d\vartheta/\ell+m\delta_{\zeta_+}+\overline{m}\delta_{\zeta_-}$ for some complex mass $m$. Thus, the quasi-CMV solution ${\cal C}_-$ must be related to such a functional $u$ for a value of $m\in\C$ which makes it quasi-definite.  

If $s_1=0$, no quasi-CMV solution appears when $|\beta|\le2$, which corresponds to $\ell$ having its roots on the unit circle. This is expected when searching for CMV solutions because no positive measure $\mu$ on the unit circle satisfies $\ell d\mu = d\vartheta$ in that case. Nevertheless, it is a less trivial result about quasi-CMV solutions of inverse Darboux and quasi-definite solutions $u$ of $u \ell \equiv d\vartheta$.  

So far, we have completed the search for quasi-CMV matrices in the subset of solutions with $s_1=0$. Now we will have a look at some spurious solutions of this subset. The simplest ones should be those obtained by fixing $e_1r_1^2=e_2r_2^2=x_0$ as a fixed point of \eqref{eq:q-rec-x} for $|\beta|\le2$. These fixed points lie on the unit circle, so they are candidates to generate spurious solutions only when $x_0=\pm1$, which corresponds to $\beta=\pm2$. However, choosing $e_1r_1^2=e_2r_2^2=\pm1$ for $\beta=\pm2$ gives no solution of the Darboux transformation with parameters because, using \eqref{eq:q-inv-ex}, we find that $e_0r_0^2=0$. Thus, no spurious solution with $s_1=0$ has an eventually constant sequence $e_nr_n^2$.

We can reinterpret the previous results by stating that, in the case $s_1=0$, the quasi-CMV solutions are those whose factors $A$ and $E$ have constant diagonals up to finitely many entries. Therefore, the solutions given by factors $A$ and $E$ with eventually periodic diagonals are spurious whenever the period is greater than one. These are the simplest spurious solutions, which are generated by the periodic solutions of \eqref{eq:q-rec-x}.  

The periodic solutions of \eqref{eq:q-rec-x} with period $j$ are characterized by the relation 
$$
 x_0 = 
 \frac{U_{j-1}(\frac{\beta}{2})-x_0U_{j-2}(\frac{\beta}{2})}
 {U_j(\frac{\beta}{2})-x_0U_{j-1}(\frac{\beta}{2})}.
$$
Using \eqref{eq:cheby}, this equation can be rewritten as
$$
 (x_0^2+\beta x_0+1) \, U_{j-1}(\ts\frac{\beta}{2}) = 0,
$$
which means that either $x_0^2+\beta x_0+1=0$ or $U_{j-1}(\frac{\beta}{2})=0$. 

The zeros of $x_0^2+\beta x_0+1$ are the fixed points $x_0=x_0^\pm$ given in \eqref{eq:q-fp} and leading to the quasi-CMV solutions previously found. Therefore, if $U_j(\frac{\beta}{2})\ne0$ for every $j\ge1$, the spurious solutions we are interested in may arise only from two alternatives: $e_1r_1^2=x_0^+$, $e_2r_2^2=x_0^-$ or $e_1r_1^2=x_0^-$, $e_2r_2^2=x_0^+$. However, \eqref{eq:q-inv-ex} shows that both options lead to $e_0r_0^2=\beta-x_0^+-x_0^-=0$, thus none of them yield a solution of the Darboux transformation with parameters. 

Therefore, the kind of spurious solutions we are looking for only appear when $\beta=\beta_{j,k}$ is a zero of $U_{j-1}(\frac{\beta}{2})$,  
$$
 \beta_{j,k} = 2\cos\left(\frac{k}{j}\pi\right), 
 \qquad
 k=1,2,\dots,j-1.
$$
Actually, for these values of $\beta$ all the solutions are spurious because quasi-CMV solutions require $|\beta|>2$.

Consider the simplest case $j=2$, so that $\beta=0$. Then, 
\eqref{eq:q-rec-x} generates a 2-periodic sequence $x_0,-x_0^{-1},x_0,-x_0^{-1},\dots$ for every initial value $x_0\ne0$. Hence, if $s_1=0$, the solutions of the Darboux transformation with parameters whose factors $A$ and $E$ have eventually 4-periodic diagonals appear for $\beta=0$. The parameters $e_1r_1^2,e_2r_2^2\in\R\setminus\{0\}$ describing these solutions are only restricted by \eqref{eq:q-cond-r0}, which now reads as $e_1r_1^2 + e_2r_2^2 \ne 0$. Although these solutions are necessarily spurious, the associated functionals are Hermitian when $e_1r_1^2=e_2r_2^2$, as required by \eqref{eq:q-ex-herm2}. In the Hermitian case, $A$ and $E$, obtained from \eqref{eq:q-inv-ex}, are given in terms of $e=e_1=e_2$ and $r=r_1=r_2$ by
\begin{equation} \label{eq:beta=0}
\begin{aligned} 
 & A = 
 \left(
 \begin{smallmatrix} 
	\\
	r_0
	\\[3pt] 
	\overline{t}_0 & \kern3pt r
	\\[3pt] 
	t_0 & \kern4pt 0 & \kern3pt r
	\\[3pt]
	& \kern3pt \overline{t}_1 & \kern3pt 0 & \kern2pt 1/r
	\\[3pt]
	& & \kern3pt t_1 & \kern2pt 0 & 1/r
	\\[3pt]
	& & & \kern2pt \overline{t}_0 & 0 & r
	\\[3pt]
	& & & & t_0 & 0 & r
	\\[-3pt]
	& & & & & \scriptsize\ddots & \scriptsize\ddots & \scriptsize\ddots
 \end{smallmatrix}
 \right),
 \qquad
 \begin{aligned}
 	& r_0 = \frac{\sqrt{2}}{r}, 
	\\[6pt]
	& t_0 = e\frac{\alpha}{r}, \qquad t_1 = -e\alpha r.
 \end{aligned}
 \\[2pt]
 & E = e
 \left(
 \begin{smallmatrix} 
	\\
	-1
	\\[3pt] 
	& \kern4pt 1
	\\[3pt] 
	& & \kern4pt 1
	\\[3pt]
	& & & -1	
	\\[3pt]
	& & & & -1
	\\[3pt]	
	& & & & & \kern4pt 1	
	\\[3pt]
	& & & & & & \kern4pt 1	
	\\[-3pt]
	& & & & & & & \scriptsize\ddots
 \end{smallmatrix}
 \right).
\end{aligned}
\end{equation}
The corresponding spurious solution has the form
$$
 {\cal C} =
 \left(\begin{smallmatrix}
 	-e\frac{\overline\alpha}{r^2} & \frac{\sqrt{2}}{r^2} & 0 & 0 & 0 & 0 
	& \kern5pt 0 & \kern5pt 0 & \kern5pt 0 & \kern10pt 0 & \cdots
	\\[2pt]
 	-\frac{\overline\alpha^2\vartriangle}{\sqrt{2}} 
	& e\overline\alpha \vartriangle & 0 & r^2 & 0 & 0 & \kern5pt 0 & \kern5pt 0 
	& \kern5pt 0 & \kern10pt 0 & \cdots
	\\[2pt]
 	\frac{\triangledown}{\sqrt{2}} & e\frac{\alpha}{r^2} & 0 & 0 & 0 & 0 
	& \kern5pt 0 & \kern5pt 0 & \kern5pt 0 & \kern10pt 0 & \cdots 
	\\[2pt]
 	e\frac{\overline\alpha^3\vartriangle}{\sqrt{2}} 
	& -\overline\alpha^2\vartriangle & 0 & -e\overline\alpha\vartriangle & 0 
	& \frac{1}{r^2} & \kern5pt 0 & \kern5pt 0 & \kern5pt 0 & \kern10pt 0 
	& \cdots
	\\[2pt]
 	-e\frac{\alpha\vartriangle}{\sqrt{2}} & 0 & \frac{1}{r^2} & 0 & 0 & 0 
	& \kern5pt 0 & \kern5pt 0 & \kern5pt 0 & \kern10pt 0 & \cdots
	\\[2pt]
 	\frac{\overline\alpha^4\vartriangle}{\sqrt{2}} 
	& -e\overline\alpha^3\vartriangle & 0 & -\overline\alpha^2\vartriangle & 0 
	& e\overline\alpha\vartriangle & \kern5pt 0 & \kern5pt r^2 & \kern5pt 0 
	& \kern10pt 0 & \cdots
	\\[2pt]
 	-\frac{\alpha^2\vartriangle}{\sqrt{2}} & 0 & e\alpha \vartriangle & 0 & r^2 
	& 0 & \kern5pt 0 & \kern5pt 0 & \kern5pt 0 & \kern10pt 0 & \cdots
	\\[2pt]
 	-e\frac{\overline\alpha^5\vartriangle}{\sqrt{2}} 
	& \overline\alpha^4\vartriangle & 0 & e\overline\alpha^3\vartriangle & 0 
	& -\overline\alpha^2\vartriangle & \kern5pt 0 
	& \kern5pt -e\overline\alpha\vartriangle & \kern5pt 0 
	& \kern10pt \frac{1}{r^2} & \cdots
	\\[2pt]
 	e\frac{\alpha^3\vartriangle}{\sqrt{2}} & 0 & -\alpha^2\vartriangle & 0 
	& -e\alpha\vartriangle & 0 & \kern5pt \frac{1}{r^2} & \kern5pt 0 
	& \kern5pt 0 & \kern10pt 0 & \cdots
	\\[5pt]
	\dots & \dots & \dots & \dots & \dots & \dots & \kern5pt \dots 
	& \kern5pt \dots & \kern5pt \dots & \kern10pt \dots & \dots 
	\\[4pt]
 \end{smallmatrix}\right),
 \qquad
 \begin{aligned}
 	& \kern-3pt \vartriangle \; = r^2+r^{-2},
 	\\[6pt]
 	& \triangledown = r^2-r^{-2}.
 \end{aligned}
$$
\end{ex}
 
\section{Conclusions and Outlook} 
\label{sec:CO}

We have developed a theory of Darboux transformations for CMV matrices, the unitary analogue of Jacobi matrices. They share many properties with the Darboux transformations of Jacobi matrices, among them the equivalence with (Laurent) polynomial modifications of the underlying measures and, as a consequence, the almost isospectrality. Indeed, the Szeg\H o connection between orthogonal polynomials on the real line and the unit circle identifies the Darboux transformations of Jacobi matrices as the Szeg\H o projection of the Darboux transformations of CMV matrices introduced in this paper.   

Nevertheless, the fact that Darboux for CMV has to deal with unitary instead of Hermitian matrices causes some dissimilarities with the Jacobi case which are worth highlighting. In particular, the unitarity forces us to use Hermitian Laurent polynomials instead of simply real polynomials, so that, in contrast with the Jacobi case, the minimal step of Darboux for CMV involves polynomials with two zeros. These facts cause a number of differences between Darboux for Jacobi and CMV which we summarize below, together with the open problems that they suggest: 
\begin{itemize}
\item Darboux transformations of Jacobi matrices are based on a factorization of a matrix which retains the hermiticity of the Jacobi one. However, the Darboux transformations of CMV matrices introduced in this paper need the factorization of matrices which do not share the unitarity of CMV, but are also Hermitian. The reason for this is that the method we have used to obtain a unitary version of Darboux is to relate a unitary matrix with a Hermitian one to which we apply the standard Darboux procedure for self-adjoint operators. Is it possible to develop a new theory of Darboux transformations for unitary operators without abandoning unitarity (i.e. based on splittings of unitary operators built out of the previous ones)?
\item Darboux transformations of Jacobi matrices are usually implemented via LU factorizations of non-symmetric tridiagonal matrices associated with monic instead of orthonormal polynomials. This allows for the unified treatment of the general quasi-definite case. Nevertheless, in the positive definite case they can be formulated in terms of Cholesky factorizations. This opens a way for the CMV generalization of Darboux transformations as defined in this paper, while Section~\ref{ssec:quasi} shows that generalized Cholesky factorizations permit their extension to the quasi-definite case. However, a question remains: Is there any CMV analogue of the LU version for the Darboux transformations (i.e. based on matrix representations with respect ``monic" zig-zag bases)? 
\item Darboux transformations for Jacobi matrices can be implemented via Cholesky factorizations and reversed ones or, equivalently, using the Cholesky factor as a change of basis relating Jacobi matrices by conjugation. While the latter method is directly generalizable to CMV matrices, the first one seems hard to tackle in the CMV case due to the drawbacks in recovering an infinite matrix from its image by an Hermitian Laurent polynomial. This leads to the following question: Given a Hermitian Laurent polynomial of degree one, is there any simple procedure to recover a CMV matrix from the result of evaluating the Laurent polynomial on such a matrix?  
\item In the Jacobi case, the Darboux transformations with parameters can be identified with the inverse Darboux transformations, which correspond exactly to the Geronimus transformations of measures on the real line. The Geronimus transformations of measures on the unit circle are a particular case of dividing a measure by an Hermitian Laurent polynomial and adding Dirac deltas at their zeros. These more general transformations of measures on the unit circle correspond exactly to the inverse Darboux transformations of CMV matrices, whose matrix realization is given by the Darboux transformations with parameters. However, these latter transformations present spurious solutions which are not CMV neither unitary matrices, and can be eventually associated with non-Hermitian functionals. Is there any interesting interpretation of these spurious solutions? Can they be understood in terms of orthonormal Laurent polynomials with respect to non-standard or matrix inner products? 
\end{itemize} 

A positive answer to the last item is suggested by the recently found connection between higher degree Darboux transformations of Jacobi matrices and Sobolev orthogonal polynomials on the real line \cite{DGAM,DM}. A similar connection for the CMV case should be the starting point for a new approach to non-standard orthogonality on the unit circle. Other promising lines of future research are indicated by the success of Darboux techniques for Jacobi matrices in problems of numerical linear algebra, the study of integrable systems or the search for bispectral situations. The recent connection found between CMV matrices and the so called quantum walks \cite{CGMV} also suggests a possible use of Darboux for CMV in the study of discrete quantum dynamical systems.

\begin{appendices}

\section{Cholesky factorizations} 
\label{app:cholesky}

Given an infinite Hermitian complex matrix $M$, the existence of a generalized Cholesky factorization (GCF) 
\begin{equation} \label{eq:D}
 M = AEA^+, 
 \qquad 
 A\in\mathscr{T},
 \qquad 
 E\in\mathscr{S},
\end{equation} 
is equivalent to state that $M$ is quasi-definite, i.e. all its leading submatrices are non-singular. Obviously the condition is necessary because, denoting by the subindex $n$ the leading submatrix of order $n$, \eqref{eq:D} implies that $M_n = A_nE_nA_n^+$. The sufficiency of the condition follows by an inductive reasoning on the order of the leading submatrices of $M$. More precisely, if $M_n=A_nE_nA_n^+$ is the GCF of the non-singular leading submatrix $M_n$ and
$$
\begin{gathered}
 M_{n+1} = 
 \begin{pmatrix}
 	M_n & Y_n^+ \\ Y_n & y_n
 \end{pmatrix},
 \quad
 Y_n\in\C^n, 
 \quad
 y_n\in\R,
 \\
 A_{n+1} = 
 \begin{pmatrix}
 	A_n & 0 \\ X_n & x_n
 \end{pmatrix},
 \quad
 X_n\in\C^n, 
 \quad
 x_n>0;
 \qquad
 E_{n+1} = 
 \begin{pmatrix}
 	E_n & 0 \\ 0 & e_n
 \end{pmatrix},
 \quad
 e_n\in\{1,-1\},
\end{gathered} 
$$
then
$$
 M_{n+1}=A_{n+1}E_{n+1}A_{n+1}^+
 \,\Leftrightarrow\, 
 \begin{cases}
 	A_nE_nX_n^+ = Y_n^+,
 	\\ 
 	X_nE_nX_n^+ + e_nx_n^2 = y_n,
 \end{cases}
 \kern-9pt\Leftrightarrow\,
 \begin{cases}
 	X_n^+ = E_nA_n^{-1}Y_n^+,
 	\\ 
 	e_nx_n^2 = y_n - Y_nM_n^{-1}Y_n^+.
 \end{cases}
$$
These equations have solutions $X_n\in\C^n$, $e_n\in\{1,-1\}$, $x_n>0$, whenever $M_{n+1}$ is non-singular because, then, $y_n \ne Y_nM_n^{-1}Y_n^+$ since otherwise $(-Y_nM_n^{-1},1)M_{n+1}=0$. This proves that the quasi-definiteness of $M$ guarantees the existence of the GCF of its leading submatrices, obtained by enlarging those of the smaller leading submatrices. This fact is key to ensure the existence of the factorization \eqref{eq:D} for the infinite matrix $M$ because we can take $A$ and $E$ as the only infinite matrices whose leading submatrices of order $n$ are respectively $A_n$ and $E_n$ for all $n$. 

GCF of quasi-definite Hermitian infinite matrices $M$ involve only admissible products (in the sense introduced in Section~\ref{sec:DT-LP-HM}), but the reversed ones 
\begin{equation} \label{eq:R}
 M = A^+EA,
 \qquad 
 A\in\mathscr{T},
 \qquad 
 E\in\mathscr{S},
\end{equation}
can be nonsense due to the presence of non-admissible products. A class of infinite matrices for which the existence of reversed GCF is a well posed problem concerning admissible products, are the band matrices, which lead to band triangular factors $A$. 

Nevertheless, even for finite matrices, quasi-definiteness is not necessary neither sufficient for the existence of reversed GCF. For instance, 
$
 \left(\begin{smallmatrix} 1&1\\1&0 \end{smallmatrix}\right)
$
has no reversed factorization despite its quasi-definiteness, while 
$
 \left(\begin{smallmatrix} 0&1\\1&1 \end{smallmatrix}\right) =
 \left(\begin{smallmatrix} 1&1\\0&1 \end{smallmatrix}\right)
 \left(\begin{smallmatrix} -1&0\\0&1 \end{smallmatrix}\right)
 \left(\begin{smallmatrix} 1&0\\1&1 \end{smallmatrix}\right)
$
is a reversed factorization of a non-quasi-definite matrix. Actually, the existence of a reversed GCF of a finite matrix $M$ is equivalent to the non-singularity of all the principal submatrices which are a south-east corner of $M$ (in what follows we will say that $M$ is `reversed' quasi-definite). This follows from similar inductive arguments to those given previously for GCF. 

The problem of the reversed quasi-definite condition is that it is nonsense for infinite matrices. The core of the difficulties in characterizing the existence of reversed GCF for infinite matrices is that, in contrast to the non-reversed ones, they are not necessarily obtained by enlarging those of the leading submatrices because \eqref{eq:R} does not imply $M_n=A_n^+E_nA_n$. This argument also shows that the relation between the reversed GCF of the leading submatrices is non-trivial. Indeed, an infinite Hermitian matrix having a reversed GCF may have no such a reversed factorization for any leading submatrix. This is illustrated by  
$$ 
 \left(\begin{smallmatrix}
 	0 & \kern3pt 1 \\[4pt]
 	1 & \kern3pt 0 & -1 \\[4pt]
 	& \kern3pt -1 & 0 & 1 \\[4pt]
 	&& 1 & 0 & -1 \\[4pt]
 	&&& -1 & 0 & 1 \\[-3pt]
 	&&&& \ddots & \ddots & \ddots
 \end{smallmatrix}\right)
 = A^+EA,
 \kern7pt
 A =
 \left(\begin{smallmatrix}
 	1 \\[4pt]
 	1 & \kern5pt 1 \\[4pt]
 	& \kern5pt 1 & \kern5pt 1 \\[4pt]
	&& \kern5pt 1 & \kern5pt 1 \\[4pt]
 	&&& \kern5pt 1 & \kern5pt 1 \\[-3pt]
 	&&&& \ddots & \ddots
 \end{smallmatrix}\right),
 \kern7pt
 E =
 \left(\begin{smallmatrix}
 	-1 \\[4pt]
 	& \kern2pt 1 \\[4pt]
 	&& \kern-1pt -1 \\[4pt]
	&&& \kern2pt 1 \\[4pt]
 	&&&& \kern-1pt -1 \\[-3pt]
 	&&&&& \ddots
 \end{smallmatrix}\right), 
$$
where no leading submatrix is reversed quasi-definite due to its null last diagonal entry. Another example is \eqref{eq:LS} for $\beta=0$, which has reversed GCF given by \eqref{eq:beta=0}. 

These drawbacks of the reversed GCF dissapear when restricting the attention to standard Cholesky factorizations (CF), i.e. $E=I$. The same kind of arguments given at the beginning of this section prove that the existence of a CF, 
$$
 M=AA^+, \qquad A\in\mathscr{T},
$$
for an infinite Hermitian matrix $M$ is equivalent to the positivity of its leading principal minors. This condition also characterizes those Hermitian matrices $M$ which are positive definite, i.e. $XMX^+>0$ for every finitely non-null row vector $X\ne0$. Besides, positive definiteness is necessary for the existence of a reversed CF, 
$$
 M=A^+A, \qquad A\in\mathscr{T},
$$
because $XMX^+=XA^+AX^+=\|AX^+\|^2>0$ for every finitely non-null vector $X$ since $\ker_R(A)=\{0\}$. 

In the case of finite matrices, positive definiteness is also characterized by the positivity of the principal minors corresponding to all the south-east corner submatrices. As a consequence, the existence of CF and reversed CF takes place simultaneously and is equivalent to positive definiteness. However, establishing a similar result for infinite matrices requires more technical hypothesis and arguments to deal with the non-trivial relation between the reversed factorizations of an infinite matrix and its leading submatrices. This is accounted for in the following proposition, which also summarizes the previous results for convenience. 

\begin{prop} \label{prop:cholesky}
If $M$ is an infinite Hermitian complex matrix, then
$$
\begin{gathered} 
 M \; \text{has a GCF} 
 \;\Leftrightarrow\;
 M \; \text{is quasi-definite} 
 \\  
 M \; \text{has a CF}
 \;\Leftrightarrow\;
 M \; \text{is positive definite}
 \;\Leftarrow\;
 M \; \text{has a reversed CF}
\end{gathered}
$$
If, besides, $M$ is band matrix with bounded entries and its lower and upper diagonals have only non-null entries, then
$$
 M \; \text{has a reversed CF} 
 \;\Leftrightarrow\;
 M \; \text{is positive definite} 
$$
\end{prop}

\begin{proof}
It only remains to prove the left-handed implication of the last statement. The positive definiteness of $M$ is equivalent to the existence of a reversed CF, $M_n=A_n^+A_n$, for every leading submatrix $M_n$ of $M$. The subindex $n$ in $A_n$ only refers to the dependence on the order of the leading submatrix $M_n$. In other words, we cannot assume that all the factors $A_n$ are leading submatrices of the same infinite matrix $A$ because $A_k$ is not necessarily a leading submatrix of $A_n$ for $k<n$. Our strategy consists in proving that, despite the non-trivial relation between the different matrix factors $A_n$, a subsequence of them ``converges" in some sense to an infinite matrix $A$ which provides a reversed CF, $M=A^+A$.

Bearing in mind that $M$ is a band matrix, the boundedness of its entries is equivalent to state that it represents a bounded operator $X \mapsto MX$ in the Hilbert space $\mathfrak{H}$ of square-summable sequences. If $P_n$ is the orthogonal projection of $\mathfrak{H}$ onto the subspace $\mathfrak{H}_n$ spanned by the first $n$ canonical vectors of $\mathfrak{H}$, then 
$$
 \hat{M}_n := P_nMP_n = M_n \oplus I = \hat{A}_n^+\hat{A}_n,
 \qquad
 \hat{A}_n := A_n \oplus I,
$$ 
where $M_n$ and $A_n$ act on $\mathfrak{H}_n$, while the identity acts on its orthogonal subspace. Since $\|\hat{A}_n\|^2=\|\hat{M}_n\|\le\|M\|$, there exists a subsequence of $\hat{A}_n$ weakly converging to some bounded operator $A$ on $\mathfrak{H}$ \cite[Theorem 5.70]{K2011}. The proposition is proved if we show that $M=A^+A$ and $A\in\mathscr{T}$, where, as in the case of $M$, we identify $A$ with its matrix representation in the canonical basis of $\mathfrak{H}$.

Suppose that $M$ is $2N+1$-diagonal. Then, each $M_n$ is also $2N+1$-diagonal, hence $A_n$ is $N+1$-band lower triangular with positive diagonal entries. Since weak convergence implies convergence of the matrix coefficients, we conclude that $A$ is $N+1$-band lower triangular with non-negative diagonal entries. To prove that $A\in\mathscr{T}$ we only need to see that its diagonal entries are non-null. This follows directly from the identity $M=A^+A$, which we will prove below, and the hypothesis that the lower and upper diagonals of $M$ have no null entries.     

As for the identity $M=A^+A$, note first that $\hat{M}_{n+N}X=MX$ for every $X\in\mathfrak{H}_n$ because $M$ is $2N+1$-diagonal. In consequence, $\hat{M}_n$ 
is strongly convergent to $M$ because 
$$
 \|(\hat{M}_{n+N}-M)X\| = \|(\hat{M}_{n+N}-M)(X-P_nX)\| \le 2\|M\|\|X-P_nX\|
$$ 
and $\|X-P_nX\|\to0$ for every $X\in\mathfrak{H}$. 

Let us denote also by $\hat{A}_n$ the subsequence weakly converging to $A$, no misunderstanding will appear from this abuse of language. In contrast with strong convergence, weak limits do not commute with compositions of operators, so we cannot conclude directly that $M=A^+A$ just taking weak limits in $\hat{M}_n=\hat{A}_n^+\hat{A}_n$. Luckily enough, the weak convergence of $\hat{A}_n$ implies its strong convergence due to the $N+1$-band structure of $\hat{A}_n$ and the uniform bound $\|\hat{A}_n\|\le\|M\|$ (see Lemma~\ref{lem:ws} below). Bearing in mind that strong convergence preserves compositions \cite{Kato}, taking strong limits in $\hat{M}_n=\hat{A}_n^+\hat{A}_n$ we finally obtain $M=A^+A$.     
\end{proof}

The proof of the previous proposition makes use of the following lemma.

\begin{lem} \label{lem:ws}
Let $B_n$ be a sequence of infinite band matrices with uniformly bounded entries and a uniformly bounded size of the band. Then, if $B_n$ is weakly convergent, it is also strongly convergent.
\end{lem}

\begin{proof}
First, $\|B_n\|$ is bounded since $B_n$ is uniformly banded with uniformly bounded entries. Therefore, the weak limit $B$ of $B_n$ is also bounded. Since the entries of $B$ are the limits of the entries of $B_n$, the weak limit $B$ inherits the band structure of $B_n$. Due to this band structure, if $X\in\mathfrak{H}$ is finitely non-null, $(B_n-B)X$ only involves a finite $n$-independent number of matrix entries of both, $B_n$ and $B$. Therefore, $\|(B_n-B)X\|\to0$ for every finitely non-null vector $X$. Given an arbitrary vector $X\in\mathfrak{H}$, consider the inequalities 
$$
\begin{aligned}
 \|(B_n-B)X\| & \le \|(B_n-B)(X-P_kX)\| + \|(B_n-B)P_kX\|
 \\
 & \le (\sup_n\|B_n\|+\|B\|)\|X-P_kX\| + \|(B_n-B)P_kX\|.
\end{aligned}
$$
We know that $\|X-P_kX\|\overset{k}{\to}0$, while $\|(B_n-B)P_kX\|\overset{n}{\to}0$ because $P_kX$ is finitely non-null for any fixed $k$. Thus, we can make both summands of the last line arbitrarily small by choosing $k$ and $n$ big enough. This implies that $\|(B_n-B)X\|\to0$, proving that $B$ is a strong limit of $B_n$. 
\end{proof}

The inductive proof of the existence of GCF given at the beginning of this section also shows that it is unique. On the contrary, the proof of Proposition \ref{prop:cholesky} suggests that a positive definite Hermitian matrix can have multiple reversed CF, at least so many as limit points has the sequence $\hat{A}_n$ arising from the CF, $M_n=A_n^+A_n$, of the leading submatrices $M_n$. The examples presented in the previous sections prove that this is indeed the case, and reversed GCF are in general not unique.

A more abstract proof of the uniqueness of a GCF sheds light on the lack of uniqueness for reversed factorizations and its relation with the algebra of infinite matrices. Suppose that 
$$
 A_1E_1A_1^+=A_2E_2A_2^+, \qquad A_i\in\mathscr{T}, \qquad E_i\in\mathscr{S}.
$$
Using the associativity laws of Proposition~\ref{prop:asoc}, we get
$E_1A_1^+(A_2^+)^{-1}=A_1^{-1}A_2E_2$. The left-hand side of this equality is upper triangular, while the right-hand side is lower triangular, so both must be diagonal. Thus, $A_1^+(A_2^+)^{-1}=D=A_1^{-1}A_2$ is a positive diagonal matrix and $E_1=E_2$. Taking care of the associativity again with Proposition~\ref{prop:asoc}, we find that $D^+D=(A_1^+(A_2^+)^{-1})^+(A_1^{-1}A_2)=I$, which implies that $D=I$ and $A_2=A_1$.

The above proof of the uniqueness of GCF works equally well for finite or infinite matrices thanks to the associativity of the involved operations, ensured by Proposition~\ref{prop:asoc} in the case of infinite matrices. A similar proof of uniqueness also works for reversed GCF of finite matrices, but fails in the case of infinite matrices because Proposition~\ref{prop:asoc} does not guarantee the associativity of the required operations. To show this, assume that  
$$
 A_1^+E_1A_1=A_2^+E_2A_2, \qquad A_i\in\mathscr{T}, \qquad E_i\in\mathscr{S},
$$
where $A_i$ are band matrices so that the products are admissible. Proposition~\ref{prop:asoc} allows us to deduce that 
$$
 A_1^+E_1 = (A_1^+E_1A_1)A_1^{-1} = (A_2^+E_2A_2)A_1^{-1} 
 = A_2^+(E_2A_2A_1^{-1}).
$$
However, when trying to move to the left the factor $A_2^+$ we only can conclude that
$$
 (A_2^+)^{-1}A_1^+E_1 = (A_2^+)^{-1}(A_1^+E_1) 
 = (A_2^+)^{-1}A_2^+(E_2A_2A_1^{-1}),
$$
but the right-hand side cannot be further simplified using Proposition~\ref{prop:asoc} because $(A_2^+)^{-1}$ is upper triangular and $A_1^{-1}$ is lower triangular. Thus, the lack of associativity for products of infinite matrices is the algebraic reason for the lack of uniqueness of reversed GCF. This problem does not appear for finite matrices, whose reversed GCF are unique.
 
If ${\cal C}$ is a CMV matrix and $\ell$ is a Hermitian Laurent polynomial with $N_z$ zeros, then $\ell({\cal C})$ is a Hermitian $2N_z+1$-diagonal matrix with bounded entries due to the unitarity of $\cal C$. Besides, the lower and upper bands of $\ell({\cal C})$ have no null entries, as follows from the zig-zag shape \eqref{eq:CMVshape} of $\cal C$. Therefore, all the previous results apply to the matrices $\ell({\cal C})$ whose CF and reversed CF are involved in Darboux for CMV.  
     
\end{appendices}

\vskip0.75cm
  
\noindent{\bf \large Acknowledgements}

\bigskip
 
The work of M. J. Cantero, L. Moral and L. Vel\'azquez has been partially supported by the Spanish Government together with the European Regional Development Fund (ERDF) under grants MTM2011-28952-C02-01 (from MICINN) and MTM2014-53963-P (from MINECO), and by Project E-64 of Diputaci\'{o}n General de Arag\'{o}n (Spain).

The work of F. Marcell\'an has been partially supported by the research project MTM2012-36732-C03-01 from Direcci\'on General de Investigaci\'on Cient\'{\i}fica y T\'ecnica, Ministerio de Econom\'{\i}a y Competitividad of Spain.

We also wish to thank the anonymous referee for very useful suggestions.

\end{document}